\documentclass[a4paper, reqno, twoside]{amsart}
\pdfoutput=1
\usepackage{etex}
\usepackage{amsfonts,amssymb,graphics,amsthm,amsmath}
\usepackage{latexsym}
\usepackage[2emode]{psfrag}
\usepackage{mathrsfs}
\usepackage[all]{xy}
\usepackage{enumerate}
\usepackage[utf8]{inputenc}
\usepackage{lscape}
\usepackage{a4wide}
\usepackage[bookmarks=true]{hyperref}

\usepackage{txfonts}

\usepackage[all]{pstricks}
\usepackage{pst-text,pst-node,pst-coil}
\usepackage{tikz}

\usepackage{cancel}
\usepackage[normalem]{ulem}
\usetikzlibrary{arrows,shapes,backgrounds}

\usepgflibrary{arrows}
\usetikzlibrary{topaths}
\usetikzlibrary{er}

\usepackage{multicol}
\newtheorem{theorem}{\sc Theorem}[section]
\newtheorem{proposition}[theorem]{\sc Proposition}

\newtheorem{lemma}[theorem]{\sc Lemma}
\newtheorem{corollary}[theorem]{\sc Corollary}
\theoremstyle{definition}
\newtheorem{definition}[theorem]{\sc Definition}

\newtheorem{example}[theorem]{\sc Example}

\theoremstyle{remark}
\newtheorem{remark}[theorem]{\sc Remark}

\newcommand{\tensor}[1]{\otimes_{\scriptscriptstyle{#1}}}

\newcommand{\Sf}[1]{\mathsf{#1}}

\renewcommand{\hom}[3]{\mathrm{Hom}_{\Sscript{#1}}\left(#2,\,#3\right)}

\newcommand{\td}[1]{\widetilde{#1}}

\newcommand{\bd}[1]{\boldsymbol{#1}}

\newcommand{\bara}[1]{\overline{#1}}

\newcommand{\LR}[1]{\left\{\underset{}{} #1 \right\}}

\newcommand{\Go}{\cG_{\Sscript{0}}}
\newcommand{\Ga}{\cG_{\Sscript{1}}}
\newcommand{\Ho}{\cH_{\Sscript{0}}}
\newcommand{\Ha}{\cH_{\Sscript{1}}}
\newcommand{\Ko}{\cK_{\Sscript{0}}}
\newcommand{\Ka}{\cK_{\Sscript{1}}}

\newcommand{\Mo}{\cM_{\Sscript{0}}}

\newcommand{\Lo}{\cL_{\Sscript{0}}}
\newcommand{\La}{\cL_{\Sscript{1}}}

\newcommand{\Ad}[1]{{\bd{ad}}_{\Sscript{g}}}
\newcommand{\Gop}{\cG^{\Sscript{op}}}
\newcommand{\Hop}{\cH^{\Sscript{op}}}

\newcommand{\gG}{\mathscr{G}}

\newcommand{\oO}{\mathscr{O}}

\newcommand{\xX}{\mathscr{X}}

\newcommand{\cG}{{\mathcal G}}
\newcommand{\cH}{{\mathcal H}}

\newcommand{\cK}{{\mathcal K}}
\newcommand{\cL}{{\mathcal L}}
\newcommand{\cM}{{\mathcal M}}

\newcommand{\cR}{{\mathcal R}}

\newcommand{\cU}{{\mathcal U}}
\newcommand{\cV}{{\mathcal V}}

\newcommand{\Gsets}{\cG\text{-}{\rm Sets}}

\newcommand{\Sscript}[1]{\scriptscriptstyle{#1}}
\newcommand{\due}[3]{{}_{{#2 }} {#1}_{{ #3}}\,}

\newcommand{\phio}{\phiup_{\Sscript{0}}}

\newcommand{\Stab}[2]{\Stabi_{ \scriptscriptstyle{\mathcal{#2}}}\left({#1}\right)}

\newcommand{\rhaction}{\text{\scriptsize{\(\leftharpoonup\)}}}
\newcommand{\lgaction}{\text{\scriptsize{\(\rightharpoonup\)}}}
\newcommand{\rgaction}{\text{\scriptsize{\(\leftharpoondown\)}}}
\newcommand{\lhaction}{\text{\scriptsize{\(\rightharpoondown\)}}}

\newcommand{\cHop}{\cH^{\Sscript{op}}}
\newcommand{\cGop}{\cG^{\Sscript{op}}}
\newcommand{\Hlcoset}[1]{\cH[#1]}
\newcommand{\Glcoset}[1]{\cG[#1]}
\newcommand{\Hrcoset}[1]{[#1]\cH}
\newcommand{\Grcoset}[1]{[#1]\cG}

\usepackage{braket, upgreek}
\DeclareMathOperator{\Id}{Id}
\DeclareMathOperator{\pr}{pr}
\DeclareMathOperator{\Stabi}{Stab}
\DeclareMathOperator{\Orb}{Orb}

\DeclareMathOperator{\rep}{rep}

\begin{document}
\allowdisplaybreaks

\title[Mackey formula for bisets over groupoids]{Mackey formula for bisets over groupoids}

\author{Laiachi El Kaoutit}
\address{Universidad de Granada, Departamento de \'{A}lgebra and IEMath-Granada. Facultad de Educaci\'{o}n, Econon\'ia y Tecnolog\'ia de Ceuta. Cortadura del Valle, s/n. E-51001 Ceuta, Spain}
\email{kaoutit@ugr.es}
\urladdr{http://www.ugr.es/~kaoutit/}

\author{Leonardo Spinosa}
\address{University of Ferrara, Department of Mathematics and Computer Science\newline
Via Machiavelli 30, Ferrara, I-44121, Italy}
\email{leonardo.spinosa@unife.it}
\urladdr{https://orcid.org/0000-0003-3220-6479}

\date{\today}

\begin{abstract}
In this paper we establish the Mackey formula for groupoids, extending the well known formula in abstract groups context.
This formula involves the notion of groupoid-biset, its orbit set and the tensor product over groupoids, as well as cosets by subgroupoids.  
\end{abstract}

\thanks{The work of Leonardo Spinosa was partially supported by   the  "National Group for Algebraic and Geometric Structures, and their Applications" (GNSAGA– INdAM).
Research supported by the Spanish Ministerio de Econom\'{\i}a y Competitividad  and the European Union FEDER, grant MTM2016-77033-P}
\maketitle
\vspace{-0.8cm}
\begin{small}
\tableofcontents
\end{small}

\pagestyle{headings}

\vspace{-1.2cm}

\section{Introduction}

We will describe the motivations behind our work and how the Mackey formula for groupoids fits into the contemporary mathematical framework.
Thereafter, we will briefly describe our main result.

\subsection{Motivations and overview}\label{ssec:1}   
The classical Mackey formula, which deals with linear representations of finite groups, appeared for the first time in \cite[Theorem 1]{MackeyIndRepGr}.  Roughly speaking, this formula involves simultaneously the restriction and the induction functors (with respect to two different subgroups) and gives a decomposition, as direct sum, of their composition, although in a non canonical way.   As was explained in \cite[Section 7.4]{SerreRepFinGr}, the Mackey formula is a key tool in proving  Mackey irreducibility criterion, which gives necessary and sufficient conditions for the irreducibility of an induced representation, and proves to be useful to study linear representations of a semidirect product by an abelian group, see  \cite[Proposition~25]{SerreRepFinGr}. Another  formulation of the classical Mackey formula, using modules over groups algebras, was stated in \cite[Theorem 44.2]{CurtReiRepThFinGrAsAlg}. Successively, in \cite[Theorems 7.1 and 12.1]{MackeyIndRepLocCoGr},  the Mackey formula was extended to the context of locally compact groups (with opportune hypotheses), and used to prove a generalization of the Frobenius Reciprocity Theorem, see  \cite[Theorems 8.1, 8.2 and 13.1]{MackeyIndRepLocCoGr}. Later on, 
many variants and different formulations of the Mackey formula have been investigated. For example, in \cite{TaylorMackeyConGr}, \cite{BonMackeyTypeA} and \cite{BonCorMackeyTypeA} Taylor and Bonnaf\'e proved opportune versions of this formula for algebraic groups. The importance of Mackey formula version in this context had already been made clear in \cite{DigMiRepFinGrLieT} and previous work had been done in \cite[Theorem 6.8]{DeliLusRepRedGrFinFi}, \cite[Lemma 2.5]{LusSpaIndUniCl}  and \cite[Theorem 7]{DeliLusDuaRepRedGrFinFiII}.

Apparently, classical Mackey formula is so intuitive that can be applied, somehow in a non trivial way, in more general and different contexts.
In this direction, motivated by the study of the structure of biset functors over finite groups  (see \cite[Definitions 3.1.1, 3.2.2 ]{Bouc:2010} for pertinent definitions),  Serge Bouc  proved  in \cite[Lemma 2.3.24]{Bouc:2010} a different kind of the classical Mackey formula in the framework of group-bisets.  The gist is that, given two groups \(H\) and \(G\) and a field \(\mathbb{F}\), an \(\left({H,G}\right)\)-biset (of groups) is a left \(H\)-invariant and right \(G\)-invariant \(\mathbb{F}\)-basis of an \(\left({\mathbb{F}H, \mathbb{F}G}\right)\)-bimodule.
Since the classical Mackey formula on linear representations can be rephrased using bimodules, and bimodules induce bisets, the Mackey formula can be reformulated using an isomorphism of group-bisets (see the end of \cite[Section 1.1.5]{Bouc:2010}) which is furthermore reformulated in \cite[Lemma 2.3.24]{Bouc:2010}. We have to mention that, in \cite{BoucBisetsCatTensProd}, Bouc himself proved an additional version of the Mackey formula, which is expressed using bimodules and group-bisets.

Groupoids are natural generalization of groups and prove to be useful in different branches of mathematics, see  \cite{Brown:1987}, \cite{Cartier:2008} and \cite{WeiGrpdUnInExtSym} (and the references therein) for a brief survey. Specifically, a groupoid is defined as a small category whose every morphism is an isomorphism and can be thought as a ``group with many objects''.
In the same way, a group can be seen as a groupoid with only one object.
As explained in \cite{Brown:1987}, while a groupoid in its very simple facet can be seen as a disjoint union of groups, this forces unnatural choices of base points and obscures the overall structure of the situation.
Besides, even under this simplicity,  structured groupoids, like topological or differential groupoids, cannot even be thought like a disjoint union of topological or differential groups, respectively. Different specialists realized (see  for instance \cite{Brown:1987} and \cite[page 6-7]{Connes:1994}) in fact that the passage from groups to groupoids is not  a trivial research and have its own difficulties and  challenges.
Thus extending a certain well known result in groups context to the framework of groupoids, is not an easy task and there are often difficulties to overcome.  

The research of this paper fits in this context.  Thus, our main aim here is to extend the formula in  \cite[Lemma 2.3.24]{Bouc:2010} to groupoid-bisets and apply it to the particular case of groupoids of equivalence relations.  
The paper is written in very elementary set-theoretical language, in order to make its content accessible to all kinds of readers.

\subsection{Description of the main result}\label{ssec:raroes} 
A right  \emph{groupoid-set} over a groupoid $\cG$, is a set $X$ with  two maps:  \emph{the structure map} $\varsigma: X \to \Go$  and    \emph{the action map} $\rho: X\, \due \times {\Sscript{\varsigma}} {\, \Sscript{{\Sf{t}}}} \, \Ga \to X$, satisfying  pertinent conditions (see Definition \ref{def:Gset} below). A morphism of groupoid-sets (or \emph{$\cG$-equivariant map}) is a map which commutes with both structure and action maps.  Left groupoid-sets are defined by interchanging the source with the target, and both categories are isomorphic. In this way a goupoid-biset is nothing but  a set endowed simultaneously with  left and right actions, which are compatible in the obvious sense. Moreover, as in the case of group, any groupoid-biset leads to a left (or right) groupoid-set over a cartesian product  of the involved groupoids, and vice-versa (see Proposition \ref{pBisetsLeftSets} below for the complete proof of this fact).  

The category of (left or right) groupoid-sets over a fixed groupoid is in fact a monoidal symmetric category. In the case of right groupoid-sets, this category is equivalent to the category of functors from the opposite groupoid to the (core) of the category of sets. Formally there should not be a more advantageous choice between these two definitions. Nevertheless, for technical reasons, which we will explain in Remark~\ref{rem:Core} below,  we will not adopt here the functorial approach to groupoid-set;  indeed we will work with the above elementary definition.

Let $\cH$, $\cG$ and $\cK$ three groupoids and consider $\cM$ and $\cL$ two subgroupoids of $\cK \times \cH$ and $\cH \times \cG$, respectively, with $\Mo=\Ko \times \Ho$ and $\Lo=\Ho\times \Go$. Let $\left({\frac{\mathcal{K} \times \mathcal{H} }{\mathcal{M} } }\right)^{\Sscript{\mathsf{L} }} $ and $\left({\frac{\mathcal{H} \times \mathcal{G} }{\mathcal{L} } }\right)^{\Sscript{\mathsf{L}} }$ be the corresponding  left cosets of $\cM$ and $\cL$, respectively (see the left version of Definition \ref{def:coset}, equation \eqref{Eq:L}). 

Consider the following  $(\cM, \cL)$-biset, see Proposition \ref{pModifDoubleCosets}:
\begin{equation*}\label{Eq:XxA} 
X= \Set{ \left({w, u, h, v, a}\right) \in \mathcal{K}_{ \scriptscriptstyle{0} } \times \mathcal{H}_{ \scriptscriptstyle{0} } \times \mathcal{H}_{ \scriptscriptstyle{1} } \times \mathcal{H}_{ \scriptscriptstyle{0} } \times \mathcal{G}_{ \scriptscriptstyle{0} }  |  \begin{gathered}
\left({w, u}\right) \in \mathcal{M}_{ \scriptscriptstyle{0} }, \, \left({v, a}\right) \in \mathcal{L}_{ \scriptscriptstyle{0} }, \\
 u = \mathsf{t}\left({h}\right),\,  v= \mathsf{s}\left({h}\right) 
\end{gathered} }.
\end{equation*}
Denote by  \(\rep_{ \scriptscriptstyle{\left(\mathcal{M},\,  \mathcal{L}\right)}}(X)\) a fixed set of representatives of the orbits of $X$ as \(\left({\mathcal{M}, \mathcal{L}}\right)\)-biset.
For each element $\left({w,\, u,\, h, \,v,\, a}\right) \,  \in  \,  \rep_{\scriptscriptstyle{\left(\mathcal{M}, \,  \mathcal{L}\right)}}(X)$, 
we define as in Definition \ref{dGrpdStarGrpd}, the subgroupoid $\mathcal{M}^{\Sscript{( w, \, u)}}  \ast \left({ ^{\left({h, \, \iota_{ \scriptscriptstyle{a} } }\right) } \mathcal{L}^{\Sscript{( u, \, a)}} }\right)$ (with only one object) of the groupoid $\cK \times \cG$. Our main result stated as 
Theorem \ref{thm:Main} below says:

{\renewcommand{\thetheorem}{{\bf A}}
\begin{theorem}[Mackey Formula for groupoid-bisets]\label{thm:A}
There is  a (non canonical) isomorphism of bisets
\begin{equation}
  \begin{gathered}
\left({\frac{\mathcal{K} \times \mathcal{H} }{\mathcal{M} } }\right)^{\Sscript{\mathsf{L} }} \otimes_{ \scriptscriptstyle{\mathcal{H} } } \left({\frac{\mathcal{H} \times \mathcal{G} }{\mathcal{L} } }\right)^{\Sscript{\mathsf{L}} } 
\cong   \biguplus_{\left({w,\, u,\, h, \,v,\, a}\right) \,  \in  \,  \rep_{\scriptscriptstyle{\left(\mathcal{M}, \,  \mathcal{L}\right)}}(X)    } \left({\frac{\mathcal{K} \times \mathcal{G} }{\mathcal{M}^{\Sscript{( w, \, u)}}  \ast \left({ ^{\left({h, \, \iota_{ \scriptscriptstyle{a} } }\right) } \mathcal{L}^{\Sscript{( u, \, a)}} }\right)  } }\right)^{\Sscript{\mathsf{L}} },
\end{gathered}
\end{equation} 
where the symbol $-\tensor{\cH}-$ stand for the tensor product over $\cH$, and where the right hand-side term in the formula is a coproduct in the category of $(\cM,\cL)$-bisets of left cosets.
\end{theorem}
}

Since the proof of this formula  is technical and for the sake of completeness, we decided to include  most of the details of the proof,  as well as the  notions involved therein. At the end of the paper (subsection~\ref{ssec:EqMF}), we provide an elementary application of this formula involving groupoids of equivalence relations (see Example~\ref{exam:X}).

\section{Abstract groupoids: General notions and basic properties}\label{sec:Grpd}
In this section we will recall from \cite{ElKaoutit:2017} the basic notions and properties that will be used in the forthcoming sections, as well as some motivating examples. 
We will expound the main concepts regarding groupoid actions: equivariant maps, orbit sets and stabilizers.

\subsection{Notations, basic notions and examples}\label{ssec:basic}
A \emph{groupoid}  is a small category where each morphism is an isomorphism. That is, a pair of two sets $\cG:=(\cG_{\Sscript{1}}, \cG_{\Sscript{0}})$ with diagram of sets
$$
\xymatrix@C=35pt{\cG_{\Sscript{1}}\ar@<0.70ex>@{->}|-{\scriptscriptstyle{{\Sf{s}}}}[r] \ar@<-0.70ex>@{->}|-{\scriptscriptstyle{{\Sf{t}}}}[r] & \ar@{->}|-{ \scriptscriptstyle{\iota}}[l] \cG_{\Sscript{0}}},$$
where $\Sf{s}$ and $\Sf{t}$ are resp.~ the source and the target of a given arrow, and $\iota$ assigns to each object its identity arrow; together with an associative and unital multiplication  $\cG_{\Sscript{2}}:= \cG_{\Sscript{1}}\, \due  \times {\Sscript{{\Sf{s}}}} {\, \Sscript{{\Sf{t}}}} \, \cG_{\Sscript{1}} \to \cG_{\Sscript{1}}$  as well as a map $\cG_{\Sscript{1}} \to \cG_{\Sscript{1}}$ which associated to each arrow its inverse. Notice, that $\iota$ is an injective map, and so $\Go$ is identified with a subset of $\Ga$. 
A groupoid is then a small category with more structure, namely, the map which send any arrow to its inverse. We implicitly identify a groupoid with its underlying category. Interchanging the source and the target will lead to \emph{the opposite groupoid} which we denote by $\Gop$.

Given a groupoid $\cG$, consider two objects $x, y \in \cG_{\Sscript{0}}$. We denote by $\cG(x,y)$ the set of all arrows with source $x$ and target $y$.  \emph{The isotropy group of $\cG$ at $x$} is then the group:
\begin{equation}\label{Eq:isotropy}
\cG^{\Sscript{x}}:=\cG(x,x)\,=\,\Big\{ g \in \cG_{\Sscript{1}}|\, \Sf{t}(g)=\Sf{s}(g)=x \Big\}.
\end{equation}
Clearly each of the sets $\cG(x,y)$ is, by the groupoid multiplication, a left $\cG^{\Sscript{y}}$-set and right $\cG^{\Sscript{x}}$-set. In fact, each of the $\cG(x,y)$ sets is a $(\cG^{\Sscript{y}}, \cG^{\Sscript{x}})$-biset,  in the sense of \cite{Bouc:2010}.

A \emph{morphism of groupoids} $\phiup: \cH \to \cG$ is a functor between the underlying categories.  Obviously any one of these morphisms  induces  homomorphisms of groups between the isotropy groups: $\phiup^{\Sscript{y}}: \cH^{\Sscript{y}}  \to \cG^{\Sscript{\phiup_0(y)}}$, for every $y \in H_{\Sscript{0}}$. The family of  homomorphisms  $\{\phiup^{\Sscript{y}}\}_{ y \, \in \, \Ho }$ is refereed to as \emph{the isotropy maps of $\phiup$}. 
For a fixed object $x \in \Go$, its fibre $\phio^{-1}(\{x\})$, if not empty, leads to  the following "star" of homomorphisms of groups:
\begin{scriptsize}
\begin{center}
\tikzstyle{every pin edge}=[<-,shorten <=1pt]
\tikz[pin distance=5mm]
\node [circle, draw,pin=right:{$\cH^{\Sscript{y}}$},
pin=above right:{$\cH^{\Sscript{y}}$},
pin=above:{$\cH^{\Sscript{y}}$}, pin=above left:{$\cH^{\Sscript{y}}$}, pin=left:{$\cH^{\Sscript{y}}$}, pin=below:{$\cH^{\Sscript{y}}$}, pin=below left:{$\cH^{\Sscript{y}}$},pin=below right:{$\cH^{\Sscript{y}}$} ]
{$\cG^{\Sscript{x}}$};
\end{center}
\end{scriptsize}
where $y$ runs in the fibre $\phio^{-1}(\{x\})$.

\begin{example}[Trivial and product groupoids]\label{exam:trivial}
Let $X$ be a set. Then the pair $(X, X)$ is obviously a groupoid (in fact a small discrete category, i.e., with only identities as arrows) with trivial structure. This is known as \emph{the trivial groupoid}. 

The \emph{product $\cG \times \cH$ of two groupoids} $\cG$ and $\cH$ is the groupoid whose sets of objects and arrows, are respectively, the  Cartesian products  $\Go \times \Ho$ and  $\Ga \times \Ha$. The multiplication, inverse and unit arrow are canonically given as follows:
$$
(g,h) \, (g', h') \,=\, (gg', hh'), \quad 	\quad (g,h)^{-1}\,=\, (g^{-1}, h^{-1}), \quad \iota_{\Sscript{(x,\, u)}}\,=\, (\iota_{\Sscript{x}}, \iota_{\Sscript{u}}).
$$
\end{example}

\begin{example}[Action groupoid]\label{exam:action}
Any group $G$ can be considered as a groupoid by taking $G_{\Sscript{1}}=G$ and $G_{\Sscript{0}}=\{*\}$ (a set with one element). Now if $X$ is a right $G$-set with action $\rho: X\times G \to X$, then one can define the so called \emph{the action groupoid $\gG$} whose set of objects is $G_{\Sscript{0}}=X$ and its set of arrows is $G_{\Sscript{1}}=X \times G$; the source and the target are $\Sf{s}=\rho$ and $\Sf{t}=pr_{\Sscript{1}}$, the identity map sends $x \mapsto (x,e)=\iota_{\Sscript{x}}$, where $e$ is the identity element of $G$. The multiplication is given by  $(x,g) (x',g')=(x,gg')$, whenever $xg=x'$, and the inverse is defined by $(x,g)^{-1}=(xg,g^{-1})$. Clearly the pair of maps $(pr_{\Sscript{2}}, *): \cG= (G_{\Sscript{1}}, G_{\Sscript{0}}) \to (G,\{*\})$  defines a morphism of groupoids. For a given $ x \in X$, the isotropy group $\cG^{\Sscript{x}}$ is clearly identified with the stabilizer $\emph{Stab}_{\Sscript{G}}(x)=\{g \in G |\, gx=x\}$ subgroup of $G$.  
\end{example}

\begin{example}[Equivalence relation groupoid]\label{exam:X} 
We expound here several examples ordered by inclusion. 
\begin{enumerate}[(1)] 
\item One can associated to  a given set $X$ the so called \emph{the groupoid  of pairs} (called \emph{fine groupoid} in \cite{Brown:1987} and \emph{simplicial groupoid} in \cite{Higgins:1971}), its set of arrows is defined by $G_{\Sscript{1}}=X \times X$ and the set of objects by $G_{\Sscript{0}}=X$; the source and the target are $\Sf{s}=pr_{\Sscript{2}}$ and $\Sf{t}=pr_{\Sscript{1}}$, the second and the first projections, and the  map of identity arrows $\iota$ is the diagonal map $x \mapsto (x,x)$. The multiplication and the inverse maps are given by 
$$
(x,x') \, (x',x'')\,=\, (x,x''),\quad \text{and} \quad (x,x')^{-1}\,=\, (x',x).
$$

\item Let $\nuup: X \to Y$ be a map.  Consider the fibre product $X\, \due \times {\Sscript{\nuup}} {\; \Sscript{\nuup}} \, X$  as a set of arrows of the groupoid $\xymatrix@C=35pt{ X\, \due \times {\Sscript{\nuup}} {\; \Sscript{\nuup}} \, X \ar@<0.8ex>@{->}|-{\scriptscriptstyle{pr_2}}[r] \ar@<-0.8ex>@{->}|-{\scriptscriptstyle{pr_1}}[r] & \ar@{->}|-{ \scriptscriptstyle{\iota}}[l] X, }$  where as before  $\Sf{s}=pr_{\Sscript{2}}$ and $\Sf{t}=pr_{\Sscript{1}}$,  and the  map of identity arrows $\iota$ is the diagonal map. The multiplication and the inverse are clear.  

\item Assume that $\cR \subseteq X \times X$ is an equivalence relation on the set $X$.  One can construct a groupoid $\xymatrix@C=35pt{\cR \ar@<0.8ex>@{->}|-{\scriptscriptstyle{pr_2}}[r] \ar@<-0.8ex>@{->}|-{\scriptscriptstyle{pr_1}}[r] & \ar@{->}|-{ \scriptscriptstyle{\iota}}[l] X, }$  with structure maps as before. This is  an important class of groupoids  known as \emph{the groupoid of equivalence relation} (or \emph{equivalence relation groupoid}). Obviously $(\cR, X) \hookrightarrow  (X\times X, X)$ is a morphism of groupoid, see for instance \cite[Example 1.4, page 301]{DemGab:GATIGAGGC}.
\end{enumerate}
Notice, that in all these examples each of the isotropy groups is the trivial group. 
\end{example}

\begin{example}[Induced groupoid]\label{exam:induced}
Let $\cG=(\cG_{\Sscript{1}},\cG_{\Sscript{0}})$ be a groupoid and $\varsigma:X \to \cG_{\Sscript{0}}$ a map. Consider the following  pair of sets:
$$
\cG^{\Sscript{\varsigma}}{}_{\Sscript{1}}:= X \,  \due \times {\Sscript{\varsigma}} {\, \Sscript{\Sf{t}}} \, \cG_{\Sscript{1}} \;  \due \times {\Sscript{\Sf{s}}} {\, \Sscript{\varsigma}}  \, X= \Big\{ (x,g,x') \in X\times \cG_{\Sscript{1}}\times X| \;\; \varsigma(x)=\Sf{t}(g), \varsigma(x')=\Sf{s}(g)  \Big\}, \quad \cG^{\Sscript{\varsigma}}{}_{\Sscript{0}}:=X.
$$
Then $\cG^{\Sscript{\varsigma}}{}=(\cG^{\Sscript{\varsigma}}{}_{\Sscript{1}}, \cG^{\Sscript{\varsigma}}{}_{\Sscript{0}})$ is a groupoid, with structure maps: $\Sf{s}= pr_{\Sscript{3}}$, $\Sf{t}= pr_{\Sscript{1}}$, $\iota_{\Sscript{x}}=(\varsigma(x), \iota_{\Sscript{\varsigma(x)}}, \varsigma(x))$, $x \in X$. The multiplication is defined by $(x,g,y) (x',g',y')= ( x,gg',y')$, whenever $y=x'$, and the inverse is given by $(x,g,y)^{-1}=(y,g^{-1},x)$. 
The groupoid $\cG^{\Sscript{\varsigma}}$ is known as \emph{the induced groupoid of $\gG$ by the map $\varsigma$}, (or \emph{ the pull-back groupoid of $\cG$ along $\varsigma$}, see   \cite{Higgins:1971} for dual notion).  Clearly, there  is a canonical morphism $\phi^{\Sscript{\varsigma}}:=(pr_{\Sscript{2}}, \varsigma): \cG^{\Sscript{\varsigma}} \to \cG$ of groupoids. 
A particular instance of an induced groupoid, is the one when $\cG=G$ is a groupoid with one object. Thus for any group $G$ one can consider the Cartesian product $X \times G \times X$ as a groupoid with set of objects $X$.  
\end{example}

\subsection{Groupoid actions and equivariant maps}\label{ssec:Grpd1} The following definition is a natural generalization to the context of groupoids, of the usual notion of group-set, see for instance \cite{Bouc:2010}. It is an abstract formulation of that given in \cite[Definition 1.6.1]{Mackenzie:2005} for Lie groupoids, and essentially the same definition based on the Sets-bundles notion given in  \cite[Definition 1.11]{Renault:1980}.
\begin{definition}\label{def:Gset}
Let  $\mathcal{G}$ be a groupoid and  $\varsigma:X \to \Go$ a map. We say that  $(X,\varsigma)$ is a \emph{right} $\cG$-\emph{set} (with a \emph{structure map} $\varsigma$), if there is a map (\emph{the action}) $\rho: X\, \due \times {\Sscript{\varsigma}} {\, \Sscript{{\Sf{t}}}} \, \Ga \to X$ sending $(x,g) \mapsto xg$ and  satisfying the following conditions:
\begin{enumerate}
\item $\Sf{s}(g)=\varsigma(xg)$, for any $x \in X$ and $g \in \Ga$ with $\varsigma(x)=\Sf{t}(g)$.
\item $x \, \iota_{\varsigma(x)}= x$, for every $x \in X$.
\item $ (xg)h= x(gh)$, for every $x \in X$, $g,h \in \Ga$ with $\varsigma(x)=\Sf{t}(g)$ and $\Sf{t}(h)=\Sf{s}(g)$.
\end{enumerate}
\end{definition}
In order to simplify the notation, the action map of a given right $\cG$-set $(X,\varsigma)$ will be omitted from the notation.  A \emph{left action} is analogously defined by interchanging the source with the target and similar notations will be employed. In general a set $X$ with a (right or left) groupoid action is called \emph{a groupoid-set}, we also employ the terminology: \emph{a set $X$ with a left (or right) $\cG$-action}. 
 
Obviously, any groupoid  $\cG$ acts  over itself on both sides by using the regular action, that is, the multiplication $\Ga \, \due \times {\Sscript{{\Sf{s}}}} {\, \Sscript{\Sf{t}}} \, \Ga \to \Ga$. This means that  $(\Ga, {s})$ is a right $\cG$-set and $(\Ga, {t})$ is a left $\cG$-set with this action.

Given a groupoid \(\mathcal{G}\), let \((X, \varsigma)\) be a right \(\mathcal{G}\)-set with  action map \(\rho\).
Then the pair of sets
\[ 
X \rtimes \mathcal{G} : = \Big( X  {\,{}_{ \scriptscriptstyle{\varsigma} } {\times}}_{ \scriptscriptstyle{\Sf{t}} }\, G_{ \scriptscriptstyle{1} }, X \Big) 
\]
is a groupoid with structure maps \(\mathsf{s}^{\rtimes }= \rho\), \(\mathsf{t}^{\rtimes}= \pr_{ \scriptscriptstyle{1} }\) and \(\iota_{ \scriptscriptstyle{x} }=(x,\iota_{ \scriptscriptstyle{\varsigma(x)} })\) for each \(x \in X\).
The multiplication and the inverse maps are defined as follows:
For each \((x,g), (y,h) \in X  {{}_{ \scriptscriptstyle{\varsigma} } {\times }}_{ \scriptscriptstyle{\mathsf{t}} }\, \mathcal{G}_{ \scriptscriptstyle{1} }\) such that \(\mathsf{t}^{\rtimes}(y,h)=\mathsf{s}^{\rtimes}(x,g)\) the multiplication is given by
\[ 
(x,g)(y,h)=(x,gh).
\]
That is, for any pairs of elements in \( X  {{}_{ \scriptscriptstyle{\varsigma} } {\times }}_{ \scriptscriptstyle{\mathsf{t}} }\, \mathcal{G}_{ \scriptscriptstyle{1} }\) as before, we have
\[ \begin{gathered}
\varsigma(x)=\mathsf{t}(g), \qquad  \mathsf{s}^{\rtimes}(x,g)=\rho(x,g)=xg, \qquad
\mathsf{t}^{\rtimes}(x,g)=\pr_{ \scriptscriptstyle{1} }(x,g)=x,   \\
\xymatrix{ \varsigma(x)   & & \varsigma(xg)  \ar[ll]_{g}} 
\end{gathered}
\]
and
\[ \begin{gathered}
\varsigma(y)=\mathsf{t}(h), \qquad  \mathsf{s}^{\rtimes}(y,h)=\rho(y,h)=yh, \qquad
\mathsf{t}^{\rtimes}(y,h)=\pr_{ \scriptscriptstyle{1} }(y,h)=y, \\
\xymatrix{ \varsigma(y)   & & \varsigma(yh).  \ar[ll]_{h}} 
\end{gathered}
\]
So, the multiplication is explicitly given by
\[ \begin{gathered}
y=\mathsf{t}^{\rtimes}(y,h)=\mathsf{s}^{\rtimes}(x,g)=xg, \\
\xymatrix{ \varsigma(x)   & & \varsigma(xg)= \varsigma(y)  \ar[ll]_{g} && \varsigma(yh)= \varsigma(xgh)  \ar[ll]_{h} \ar@/^2pc/[llll]_{gh} }
\end{gathered}
\]
and schematically can be presented by
\[  \xymatrix{ x   & &  xg=y  \ar[ll]_{(x,g)} && yh=xgh . \ar[ll]_{(y,h)} \ar@/^2pc/[llll]_{(x,\, gh)} }
\]

For each \((x,g)\in X  {{}_{ \scriptscriptstyle{\varsigma} } {\times }}_{ \scriptscriptstyle{\mathsf{t}} }\, \mathcal{G}_{ \scriptscriptstyle{1} }\) the inverse arrow is defined by \((x,g)^{-1}=(xg,g^{-1})\) .
The groupoid \(X \rtimes \mathcal{G}\) is called the \emph{right translation groupoid of \(X\) by \(\mathcal{G}\)}.
Furthermore, there is a canonical morphism of groupoids $\sigmaup: X\rJoin \cG \to \cG$, given by the pair of maps $\sigmaup=(\varsigma, pr_{\Sscript{2}})$. 

The \emph{left translation groupoids} are similarly defined and  denoted by $\cG \ltimes Z$ whenever $(Z, \vartheta)$ is a left $\cG$-set.

A \emph{morphism of  right $\cG$-sets} (or \emph{$\cG$-equivariant map})  $F: (X,\varsigma) \to (X',\varsigma')$ is a map $F:X \to X'$ such that the diagrams 

\begin{equation}\label{Eq:Gequi}
\begin{gathered}
\xymatrix@R=12pt{ & X \ar@{->}_-{\Sscript{\varsigma}}[ld]  \ar@{->}^-{F}[dd] & \\ \mathcal{G}_{\Sscript{0}}& & \\ & X' \ar@{->}^-{\Sscript{\varsigma'}}[lu]  & } \qquad  \qquad \xymatrix@R=12pt{X\, \due \times {\Sscript{\varsigma}} {\, \Sscript{\Sf{t}}} \,  \Ga \ar@{->}^-{}[rr]  \ar@{->}_-{\Sscript{F\, \times \, id}}[dd] & & X  \ar@{->}^-{\Sscript{F}}[dd] \\  & & \\ X'\, \due \times {\Sscript{\varsigma'}} {\, \Sscript{\Sf{t}}} \,  \Ga  \ar@{->}^-{}[rr] & & X'  } 
\end{gathered}
\end{equation}
commute. We denote by $\hom{\Gsets}{X}{X'}$ the set of all $\cG$-equivariant maps from $(X,\varsigma)$ to $(X',\varsigma')$. Clearly any such a $\cG$-equivariant map induces a morphism of groupoids $\Sf{F}: X \rJoin \cG \to X' \rJoin \cG$. A subset $Y \subseteq X$ of a right $\cG$-set $(X,\varsigma)$,  is said to be \emph{$\cG$-invariant} whenever the  inclusion $Y \hookrightarrow X$ is a $\cG$-equivaraint map.

\begin{remark}\label{rem:Core}
A right $\cG$-set can be defined also as a functor $\xX: \cG^{\Sscript{op}} \to \mathit{Sets}$ to the core of the category of sets. Thus to any functor of this kind, we can  attach to it a set $X=\biguplus _{a \in \Go } \xX(a)$ with the canonical map $\varsigma: X \to \Go$ and action $\varrho: X\, \due \times {\Sscript{\varsigma}} {\, \Sscript{{\Sf{t}}}} \, \Ga \to X$ sending $(x,g) \mapsto xg:=\xX(g)(x)$. Notice that, if  none of the fibers $\xX(a)$ is an empty set, then the induced map $\varsigma$ is surjective. However, this is not always the case, see the  example expound in subsection \ref{ssec:bonito}.  

Any natural transformation between functors as above leads to a $\cG$-equivariant map between the associated right $\cG$-sets. This in fact establishes an equivalence of categories between the category of the functors of this form (i.e.,~ functors  $\cG^{\Sscript{op}}$ to the core category of the category of sets) and the category of right $\cG$-sets. The functor, which goes in the opposite direction,  associates to any  right $\cG$-set $(X, \varsigma)$ the functor $\xX: \cG^{\Sscript{op}} \to \mathit{Sets}$ which sends any object $a \in \Go$ to the fibre $\varsigma^{-1}(\{a\})$, and  any arrow $g$ in $\cG^{\Sscript{op}}$ to the bijective map $\varsigma^{-1}(\{\Sf{t}(g)\}) \to 	\varsigma^{-1}(\{\Sf{s}(g)\})$, $x \mapsto xg$.    

Formally there should not be a more advantageous choice between these two definitions. Nevertheless, in our opinion, for technical reasons it is perhaps better to deal with groupoid-sets as given in Definition~\ref{def:Gset}, instead of the aforementioned functorial approach. Specifically,  the latter approach presents an inconvenient, since  one is forced to distinguish,  in certain ``local" proofs, between the cases when the fiber is empty and when it is not.  There is no such  difficulty  using Definition \ref{def:Gset}, as we will see in the sequel. 
\end{remark}

\begin{example}\label{exam: HG}
Let $\phiup: \cH \to \cG$ be a morphism of groupoids. Consider the triple $(\Ho\, \due \times {\Sscript{\phiup_0}} {\, \Sscript{\Sf{t}}} \,  \Ga, pr_{\Sscript{1}}, \varsigma)$, where $\varsigma: \Ho\, \due \times {\Sscript{\phiup_0}} {\, \Sscript{\Sf{t}}} \,  \Ga \to \Go$ sends $(u,g) \mapsto s(g)$, and $pr_{\Sscript{1}}$ is the first projection. Then the following maps 
\begin{equation}\label{Eq:HG}
\begin{gathered}
\xymatrix@R=0pt@C=10pt{ \Big(\Ho\, \due \times {\Sscript{\phiup_0}} {\, \Sscript{\Sf{t}}} \,  \Ga\Big) \, \due \times {\Sscript{\varsigma}} {\, \Sscript{\Sf{t}}} \,  \Ga \ar@{->}^-{}[r] &  \Ho\, \due \times {\Sscript{\phiup_0}} {\, \Sscript{\Sf{t}}} \,  \Ga,  \\ \Big((u,g'),g \Big)\ar@{|->}^-{}[r]  &  (u,g')  \rgaction g:=(u, g'g)}  \quad \xymatrix@R=0pt@C=10pt{ \Ha \, \due \times {\Sscript{\Sf{s}}} {\, \Sscript{pr_1}} \,  \Big(\Ho\, \due \times {\Sscript{\phiup_0}} {\, \Sscript{\Sf{t}}} \,  \Ga\Big) \ar@{->}^-{}[r] &  \Ho\, \due \times {\Sscript{\phiup_0}} {\, \Sscript{\Sf{t}}} \,  \Ga \\ \Big(h, (u,g)\Big) \ar@{|->}^-{}[r] &  h \lhaction (u,g):=(t(h), \phiup(h)g) }
\end{gathered}
\end{equation}
define, respectively, a  structure of  right $\cG$-sets and  that of  left $\cH$-set.  Analogously,  the maps 
\begin{equation}\label{Eq:GH}
\begin{gathered}
\xymatrix@R=0pt@C=10pt{ \Big(\Ga\, \due \times {\Sscript{\Sf{s}}} {\, \Sscript{\phiup_0}} \,  \Ho\Big) \, \due \times {\Sscript{pr_2}} {\, \Sscript{\Sf{t}}} \,  \Ha \ar@{->}^-{}[r] & \Ga\, \due \times {\Sscript{\Sf{s}}} {\, \Sscript{\phiup_0}} \,  \Ho   \\ \Big((g,u),h \Big)\ar@{|->}^-{}[r]  &   (g,u) \rhaction h:= (g\phiup(h), s(h))}  \;\;  \xymatrix@R=0pt@C=10pt{ \Ga \, \due \times {\Sscript{\Sf{s}}} {\, \Sscript{\vartheta}} \,  \Big(\Ga\, \due \times {\Sscript{\Sf{s}}} {\, \Sscript{\phiup_0}} \,  \Ho\Big) \ar@{->}^-{}[r] &  \Ga\, \due \times {\Sscript{\Sf{s}}} {\, \Sscript{\phiup_0}} \,  \Ho  \\ \Big(g, (g',u)\Big) \ar@{|->}^-{}[r] & g \lgaction (g',u):= (gg',u)}
\end{gathered}
\end{equation}
where $\vartheta: \Ga\, \due \times {\Sscript{\Sf{s}}} {\, \Sscript{\phiup_0}} \,  \Ho \to \Go$ sends $(g,u) \mapsto t(g)$, define a left $\cH$-set and right $\cG$-set structures on $\Ga\, \due \times {\Sscript{\Sf{s}}} {\, \Sscript{\phiup_0}} \,  \Ho$, respectively. 
This in particular can be applied to any morphism of groupoids of the form $(X,X)\to (Y \times Y, Y)$, $(x,x') \mapsto \big((f(x),f(x)), \, f(x')\big)$, where $f: X\to Y$ is any map. On the other hand, if $f$ is a $G$-equivariant map, for some a group $G$ acting on both $X$ and $Y$, then the above construction applies to the morphism  $(G\times X, X) \to (G\times Y, Y)$ sending $\big(  (g,x), x' \big) \mapsto \big( (g,f(x)) , f(x')\big)$ of action groupoids, as well. 
\end{example}

The proofs of the following useful Lemmas  are immediate. 

\begin{lemma}\label{lInvEquivaIsEquiva}
Given a groupoid \(\mathcal{G}\), let \((X, \varsigma)\) be a right \(\mathcal{G}\)-set with action \(\rho\) and let be \((X', \varsigma')\) be a right \(\mathcal{G}\)-set with  action \(\rho'\).
Let \(F \colon \left(X, \varsigma\right) \longrightarrow \left(X', \varsigma'\right)\) be a \(\mathcal{G}\)-equivariant map with bijective underlying map.
Then \(F^{-1} \colon \left(X', \varsigma'\right) \longrightarrow \left(X, \varsigma\right)\) is also \(\mathcal{G}\)-equivariant.
\end{lemma}

\begin{lemma}
\label{lActionRestriction}
Given a groupoid \(\mathcal{G}\), let \((X, \varsigma)\) be a right \(\mathcal{G}\)-set with action \(\rho\) and let be \(Y \subseteq X\).
We define
\[ \varsigma'= \left.{\varsigma}\right|_{Y} \colon Y \longrightarrow \mathcal{G}_{ \scriptscriptstyle{0} }
\qquad \text{and} \qquad
\rho' = \left.{\rho}\right|_{Y  {{}_{ \scriptscriptstyle{ \varsigma '} } {\times}}_{ \scriptscriptstyle{\mathsf{t} } }\, \mathcal{G}_{ \scriptscriptstyle{1} }} \colon Y  {{}_{ \scriptscriptstyle{\varsigma'} } {\times}}_{ \scriptscriptstyle{\mathsf{t}} }\, \mathcal{G}_{ \scriptscriptstyle{1} } \longrightarrow X
\]
and let's suppose that for each \((a,g) \in Y  {{}_{ \scriptscriptstyle{ \varsigma '} } {\times}}_{ \scriptscriptstyle{\mathsf{t} } }\, \mathcal{G}_{ \scriptscriptstyle{1} }\) we have \(\rho(a,g) \in Y\).
Then \((Y,\varsigma')\) is a right \(\mathcal{G}\)-set with action map \(\rho'\). In particular $Y$ is a $\cG$-invariant subset of $X$.
\end{lemma}

\subsection{Orbit sets and stabilizers}\label{ssec:Orbits}

Next we recall the notion (see, for instance, \cite[page~11]{Jelenc:2013}) of the orbit set attached to a right groupoid-set.
This notion is a generalization of the orbit set in the context of group-sets.  Here we use the (right) translation groupoid to introduce this set. 
First we recall the notion of the orbit set of a given groupoid.  \emph{The orbit set of a groupoid} $\cG$ is the  quotient set of $\mathcal{G}_{\Sscript{0}}$  by the following equivalence relation: take an object $x \in \cG_{\Sscript{0}}$, define 
\begin{equation}\label{Eq:Ox}
\oO_{\Sscript{x}}:=\Sf{t}\big( \Sf{s}^{-1}(\{x\})\big)\, =\,\Big\{  y \in \cG_{\Sscript{0}}|\; \exists \,g \in \cG_{\Sscript{1}}\, \text{ such that }\, \Sf{s}(g)=x, \, \Sf{t}(g)=y \Big\},
\end{equation}
which is equal to  the set $\Sf{s}\big( \Sf{t}^{-1}(\{x\})\big)$. This is a non empty set, since $x \in \oO_{\Sscript{x}}$.
Two objects $x, x' \in \Go$ are said to be equivalent if and only if $\oO_{\Sscript{x}}\,=\, \oO_{\Sscript{x'}}$, or equivalently, two objects are equivalent if and only if there is an arrow connecting them.  This in fact defines an equivalence relation whose quotient set is denoted by $\Go/\cG$. In others words, this is the set of all connected components of $\cG$, which we denote by $\pi_{\Sscript{0}}(\cG):=\Go/\cG$. 

Given  a right $\cG$-set  $(X,\varsigma)$, the \emph{orbit set}  $X/\cG$ of $(X,\varsigma)$ is the orbit set of the (right) translation groupoid  $X \rJoin \cG$, that is, $X/\cG=\pi_{\Sscript{0}}(X \rJoin \cG)$.  For an element \(x \in X\),  the \emph{equivalence class}  of $x$, called the \emph{the orbit of $x$}, is denoted by 
\[ \Orb_{ \scriptscriptstyle{X \,  \rtimes \, \mathcal{G}} } (x)=\Set{ y \in X | \begin{aligned}
& \exists \, (x,g) \in \left(X \rtimes \mathcal{G}\right)_{ \scriptscriptstyle{1} } \, \text{ such that } \\
& x = \mathsf{t}^{\rtimes }(x,g) \, \text{ and }\, 
 y = \mathsf{s}^{\rtimes }(x,g)=xg
\end{aligned} } = \LR{ xg \in X | \mathsf{t}(g)=\varsigma(x) }
:=[x] \,  \mathcal{G}.
\]

Let us denote by $\rep_{\Sscript{\cG}}(X)$ a \emph{representative set} of the orbit set $X/\cG$. For instance, if $\cG=(X\times G, X)$ is an  action groupoid as in Example \ref{exam:action}, then obviously the orbit set of this groupoid coincides with the classical set of orbits  $X/G$. Of course, the orbit set of an equivalence relation groupoid $(\cR, X)$, see Example \ref{exam:X},  is precisely the quotient set $X/\cR$.  

The left orbits sets for  left groupoids sets are analogously  defined by using the left translation groupoids. We will use the following notation for left orbits sets: Given a left $\cG$-set $(Z,\vartheta)$ its orbit set will be denoted by $\cG \backslash Z$ and the orbit of an element $z \in Z$ by $\Glcoset{z}$.

Let \((X, \varsigma)\) be a right \(\mathcal{G}\)-set with action \(\rho \colon X  {{}_{ \scriptscriptstyle{\varsigma} } {\times }}_{ \scriptscriptstyle{\mathsf{t}} }\, \mathcal{G}_{ \scriptscriptstyle{1} } \longrightarrow X\). 
The \emph{stabilizer} \(\Stabi_{ \scriptscriptstyle{\mathcal{G}}}\left({x}\right) \) of \(x\) in \(\mathcal{G}\) is the groupoid with arrows 
\[ 
\left( \Stabi_{ \scriptscriptstyle{\mathcal{G}}}\left({x}\right)  \right)_{ \scriptscriptstyle{1} }
= \Set{ g \in \mathcal{G}_{ \scriptscriptstyle{1} }  |  \varsigma\left({x}\right)= \mathsf{t}\left({g}\right)  \quad \text{and} \quad x g = x }
\]
and objects 
\[ \left( \Stabi_{ \scriptscriptstyle{\mathcal{G}}}\left({x}\right)  \right)_{ \scriptscriptstyle{0} } 
= \Set{ u \in \mathcal{G}_{ \scriptscriptstyle{0} }  |  \exists g \in \mathcal{G}_{ \scriptscriptstyle{1} }(\varsigma\left({x}\right), u) : xg=x   } \subseteq \mathscr{O}_{ \scriptscriptstyle{\varsigma\left({x}\right)} } .
\]
Note that \(x \iota_{ \scriptscriptstyle{\varsigma\left({x}\right)} }= x\) so \(\varsigma\left({x}\right) \in \left( \Stabi_{ \scriptscriptstyle{\mathcal{G}}}\left({x}\right)  \right)_{ \scriptscriptstyle{0} } \).
Besides that, using the first condition of the right \(\mathcal{G}\)-set, if \(\mathsf{t}\left({g}\right)= \varsigma\left({x}\right)\) and $g \in \left( \Stabi_{ \scriptscriptstyle{\mathcal{G}}}\left({x}\right)  \right)_{ \scriptscriptstyle{1} }$, then \(\mathsf{s}\left({g}\right)= \varsigma\left({xg}\right)=\varsigma\left({x}\right)\).
Therefore
\[ \left( \Stabi_{ \scriptscriptstyle{\mathcal{G}}}\left({x}\right)  \right)_{ \scriptscriptstyle{0} } 
= \Set{\varsigma\left({x}\right)},
\qquad  \Stabi_{\Sscript{\cG}}(x)^{\Sscript{\varsigma(x)}}
 \le \mathcal{G}^{ \scriptscriptstyle{\varsigma\left({x}\right)} }
\]
and as a groupoid with only one object $\varsigma(x)$, the set of arrow is:
\[ 
\left({\Stabi_{ \scriptscriptstyle{\mathcal{G}}}\left({x}\right)  }\right)_{ \scriptscriptstyle{1} }= \LR{ g \in \mathcal{G}_{ \scriptscriptstyle{1} }  | \text{ \(\mathsf{s}\left({g}\right)= \mathsf{t}\left({g}\right) = \varsigma\left({x g}\right) \) and \(x g= x \) }}.
\]
Henceforth, the stabilizer of an element $x \in X$ is the subgroup of the isotropy group $\cG^{\Sscript{\varsigma(x)}}$ consisting of those loops $g$ which satisfy $x g =x$. The following lemma is then an immediate consequence of this observation:

\begin{lemma}\label{lema:Stab}
Let $(X, \varsigma)$ be a right $\cG$-set and consider its associated morphism of groupoids $\sigmaup: X\rJoin \cG \to \cG$, given by the pair of maps $\sigmaup=(\varsigma, pr_{\Sscript{2}})$. Then, for any $x \in X$, the stabilizer $\Stab{x}{G}$ is the image by $\sigmaup$ of the isotropy group $(X\rtimes \cG)^{\Sscript{x}}$.
\end{lemma}

The \emph{left stabilizer} for elements of left groupoids sets are similarly defined and enjoy analogue properties as in Lemma \ref{lema:Stab}.

\section{Groupoid-bisets, translation groupoids, orbits and cosets}\label{sec:II}

In this section, we recall from \cite{ElKaoutit:2017} (with sufficient details)  the notions of groupoid-biset, two sided translation groupoid, coset by a subgroupoid and tensor product of bisets. After that, we will discuss the decomposition of a set, with a groupoid acting over it, into disjoint orbits. Moreover, we will prove the bijective correspondence between groupoid-bisets and left sets over the product of the involved groupoids.

\subsection{Bisets,  two sided translation groupoids and (left) cosets}\label{ssec:biset}

Let $\cG$ and $\cH$ be two groupoids and $(X, \vartheta,\varsigma)$ a triple consisting of a set $X$ and two maps $\varsigma : X \to \Go$, $\vartheta: X \to \Ho$.  The following definitions are abstract formulations of those given in \cite{Jelenc:2013, Moedijk/Mrcun:2005} for topological  and Lie groupoids.
In this regard, see also \cite{ElKaoutit:2017}.

\begin{definition}\label{def:biset}
The triple $(X,\vartheta,\varsigma)$ is said to be an \emph{$(\cH,\cG)$-biset} if there is a left $\cH$-action $\lambda: \Ha \, \due \times {\Sscript{\mathsf{s}}} {\, \Sscript{\vartheta}} \,  X \to X$ and right $\cG$-action $\rho: X\, \due \times {\Sscript{\varsigma}} {\, \Sscript{\mathsf{t}}} \,  \Ga  \to X$ such that
\begin{enumerate} 
\item For any $x \in X$, $h \in \cH_{\Sscript{1}}$, $g \in \cG_{\Sscript{1}}$ with $\vartheta(x)=\mathsf{s}(h)$ and $\varsigma(x)=\mathsf{t}(g)$, we have
$$ \vartheta(xg) =\vartheta(x)\; \text{ and }\; \varsigma(hx)=\varsigma(x).$$
\item For any $ x \in X$, $h \in \cH_{\Sscript{1}}$ and $ g \in \cG_{\Sscript{1}}$ with $\varsigma(x)=\mathsf{t}(g)$, $\vartheta(x)=\mathsf{s}(h)$, we have 
$h(xg)\,=\, (hx)g$.
\end{enumerate}
The triple $(X, \vartheta,\varsigma)$ is referred to as a \emph{groupoid-biset}, whenever the two groupoids $\cH$ and $\cG$ are clear.
\end{definition}

For simplicity the actions maps of a groupoid-bisets are omitted in the notions. 
\emph{The two sided translation groupoid} associated to a given $(\cH, \cG)$-biset $(X, \vartheta, \varsigma)$ is defined to be the groupoid $\cH \lJoin X \rJoin \cG$ whose set of objects is $X$ and set of arrows is 
$$
\cH_{\Sscript{1}}\, \due \times {\Sscript{{s}}} {\, \Sscript{\vartheta}} \, X \, \due \times {\Sscript{\varsigma}}{\, \Sscript{{s}}} \, \cG_{\Sscript{1}}\,=\, \Big\{ (h,x,g) \, \in \,  \cH_{\Sscript{1}}\times X \times \cG_{\Sscript{1}}| \,\, \Sf{s}(h)= \vartheta(x),\, \Sf{s}(g)=\varsigma(x) \Big\}.
$$
The structure maps are:  
$$
\Sf{s}(h,x,g)=x,\quad \Sf{t}(h,x,g)=hxg^{-1}\;\;  \text{ and }\; \iota_{\Sscript{x}}=(\iota_{\Sscript{\vartheta(x)}}, x,  \iota_{\Sscript{\varsigma(x)}}). 
$$
The multiplication and the inverse are given by: 
$$
(h,x,g) (h',x',g')\,=\,(hh',x',gg'),\quad  (h,x,g)^{-1}=(h^{-1}, hxg^{-1}, g^{-1}).
$$

\begin{example}
Given two groups \(G\) and \(H\) and a group-biset \(U\), the category \(\left\langle U \right\rangle\) defined by Bouc in \cite[Notation~\(2.1\)]{BoucBisetsCatTensProd} is the translation groupoid of the \(\left(\mathcal{H}, \mathcal{G}\right)\)-biset \(V\), where \(\mathcal{H}\), respectively \(\mathcal{G}\), is the groupoid with only one object and isotropy group \(H\), respectively \(G\), and \(V\) is exactly \(U\) considered as a groupoid-biset.
\end{example}

The \emph{orbit space of the two translation groupoid}  is the quotient set $X / (\cH,\cG)$ defined using the equivalence relation $x \sim x'$, if and only if, there exist $h \in \Ha$ and $g \in \Ga$ with $\mathsf{s}(h)=\vartheta(x)$ and $\mathsf{t}(g)=\varsigma(x')$,  such that $hx\,=\, x'g$.
We will also employ  the notation $\cH \backslash X / \cG := X/(\cH,\, \cG)$ and denote by $\rep_{\Sscript{(\cH, \, \cG)}}(X)$ one of its \emph{representative sets}.

\begin{example}\label{exam:bisets}
Let $\phiup: \cH \to \cG$ be a morphism of groupoids. Consider, as in Example \ref{exam: HG},  the associated triples $(\Ho\, \due \times {\Sscript{\phiup_0}} {\, \Sscript{\Sf{t}}} \,  \Ga, \varsigma, pr_{\Sscript{1}})$ and  
$(\Ga\, \due \times {\Sscript{\Sf{s}}} {\, \Sscript{\phiup_0}} \,  \Ho, pr_{\Sscript{2}}, \vartheta)$ with actions defined as in eqautions \eqref{Eq:HG} and \eqref{Eq:GH}.  Then these triples are an $(\cH, \cG)$-biset and a $(\cG,\cH)$-biset, respectively. 
\end{example}

\begin{proposition}\label{pBisetRight} 
Let \((X, \vartheta, \varsigma)\) be an \(\left(\mathcal{H}, \mathcal{G}\right)\)-biset with actions $\lambda: \Ha \, \due \times {\Sscript{\mathsf{s}}} {\, \Sscript{\vartheta}} \,  X \to X$ and  $\rho: X\, \due \times {\Sscript{\varsigma}} {\, \Sscript{\mathsf{t}}} \,  \Ga  \to X$. 
Then \(\mathcal{H} \backslash X\) is a right \(\mathcal{G}\)-set with structure map and action as follows:
\[ 
\begin{aligned}{\widehat{\varsigma}}  \colon & { \mathcal{H} \backslash X} \longrightarrow {\mathcal{G}_{ \scriptscriptstyle{0} } } \\ & {\mathcal{H}[x]}  \longrightarrow {\widehat{\varsigma}\left({ \Hlcoset{x} }\right)= \varsigma\left({x}\right) }\end{aligned} 
\qquad \text{and} \qquad
\begin{aligned}{\widehat{\rho}}  \colon & {\left( \mathcal{H} \backslash X \right)  {{}_{ \scriptscriptstyle{\widehat{\varsigma}} } {\times}}_{ \scriptscriptstyle{\mathsf{t} } }\, \mathcal{G}_{ \scriptscriptstyle{1} } } \longrightarrow {\mathcal{H} \backslash X} \\ & { \left( \Hlcoset{x}, g\right)}  \longrightarrow {\widehat{\rho} \left({\Hlcoset{x}, g}\right)= \Hlcoset{xg} = \Hlcoset{\rho\left({x, g}\right)}.}\end{aligned} 
\]
\end{proposition}
\begin{proof}
Let be \(x_1, x_2 \in X\) such that \(\Hlcoset{x_{1}}=  \Hlcoset{x_{2}}\).
Then by definition of orbit there is \(h \in \mathcal{H}_{ \scriptscriptstyle{1} }\) such that \(x_1 = h x_2\) and \(\vartheta\left({x_2}\right)= \mathsf{s}\left({h}\right)\).
One of the biset conditions says that  \(\varsigma\left({x_1}\right) = \varsigma\left({h x_2}\right) = \varsigma\left({x_2}\right)\). This shows that  \(\widehat{\rho}\) is well defined.
Now let be \(g \in \mathcal{G}_{ \scriptscriptstyle{1} }\), \(x_1, x_2 \in X\) such that \(  \Hlcoset{x_{1}}= \Hlcoset{x_{2}} \), \(\widehat{\varsigma}\left(\Hlcoset{x_{1}} \right) = \mathsf{t}\left({g}\right)\) and \(\widehat{\varsigma}\left( \Hlcoset{x_{2}}\right) = \mathsf{t}\left({g}\right)\).
We have
\[ \varsigma\left({x_1 }\right)= \widehat{\varsigma}\left( \Hlcoset{x_{1}} \right) 
= \mathsf{t}\left({g}\right)= \widehat{\varsigma}\left( \Hlcoset{x_{2}}\right)= \varsigma\left({x_2}\right)
\]
and
\[ \rho\left({x_1, g}\right)= x_1 g = \left({ h x_2}\right) g
= h \left({x_2 g}\right) = h \rho \left({x_2, g}\right) 
\]
so \(\Hlcoset{x_{1}g} = \Hlcoset{x_{2}g}\) which shows that \(\widehat{\rho}\) is well defined.

Now we only have to check the axioms of right \(\mathcal{G}\)-set but this is easy and is left to the reader.
As a consequence, we have proved that \(\mathcal{H} \backslash X\) is a right \(\mathcal{G}\)-set as stated.
\end{proof}

The left version of Proposition \ref{pBisetRight} also holds true. Precisely, given  an \(\left({\mathcal{H}, \mathcal{G}}\right)\)-biset  \((X, \varsigma, \vartheta)\);
since \(X\) is obviously a \(\left({\Gop, \Hop}\right)\)-biset, applying  Proposition \ref{pBisetRight},  we obtain that \(\Gop \backslash X\) is right \(\Hop\)-set, that is, \(X / \mathcal{G}\) becomes a left \(\mathcal{H}\)-set.

Next we deal with the \emph{left and right cosets} attached to a morphism of groupoids.  
So let us assume that a morphism of groupoids $\upphi : \cH \to \cG$ is given and denote by  \({}^{\Sscript{\upphi}}\cH(\cG)=\mathcal{H}_{ \scriptscriptstyle{0} }  {{}_{ \scriptscriptstyle{\upphi_0} } {\times}}_{ \scriptscriptstyle{\Sf{t}} }\, \mathcal{G}_{ \scriptscriptstyle{1} }\) the underlying set of the  $(\cH, \cG)$-biset of Example \ref{exam:bisets}. Then  the left translation groupoid is given by 
\[ \mathcal{H} \ltimes  { ^{ \scriptscriptstyle{\upphi} } \mathcal{H}\left(\mathcal{G}\right)} = \mathcal{H} \ltimes \left( \mathcal{H}_{ \scriptscriptstyle{0} }  {{}_{ \scriptscriptstyle{\upphi_0} } {\times}}_{ \scriptscriptstyle{\mathsf{t}} }\, \mathcal{G}_1 \right) 
= \left( \mathcal{H}_{ \scriptscriptstyle{1} }  {{}_{ \scriptscriptstyle{\mathsf{s}} } {\times }}_{ \scriptscriptstyle{\pr_1} }\, \left( \mathcal{H}_{ \scriptscriptstyle{0} }  {{}_{ \scriptscriptstyle{\upphi_0} } {\times}}_{ \scriptscriptstyle{\mathsf{t}} }\, \mathcal{G}_1  \right) , \mathcal{H}_{ \scriptscriptstyle{0} }  {{}_{ \scriptscriptstyle{\upphi_0} } {\times}}_{ \scriptscriptstyle{\mathsf{t}} }\, \mathcal{G}_{ \scriptscriptstyle{1} }  \right)
\]
where the source \(s^{\ltimes}\) is the action $\lhaction$ described in equation \eqref{Eq:HG} and the target \(t^{\ltimes}\)  is the second projection on \(X\). The multiplication of two objects 
\( (h_1, a_1, g_1), (h_2, a_2, g_2) \in \mathcal{H}_{ \scriptscriptstyle{1} }  {{}_{ \scriptscriptstyle{\mathsf{s}} } {\times}}_{ \scriptscriptstyle{\pr_1} }\, \mathcal{H}_{ \scriptscriptstyle{0} }  {{}_{ \scriptscriptstyle{\upphi_0} } {\times}}_{ \scriptscriptstyle{\mathsf{t}} }\, \mathcal{G}_{ \scriptscriptstyle{1} } \)
such that \(\mathsf{s}^{\ltimes}(h_1, a_1, g_1)= \mathsf{t}^{\ltimes}(h_2, a_2, g_2)\), is given as follows: First we have 
\[ \left(t(h_1), \upphi_{ \scriptscriptstyle{1} }(h_1) g_1 \right)= h_1 \lhaction  (a_1, g_1)
= \mathsf{s}^{\ltimes}(h_1, a_1, g_1)= \mathsf{t}^{\ltimes}(h_2, a_2, g_2) = (a_2, g_2)
\]
and \(s(h_2)=\pr_{ \scriptscriptstyle{1} }(a_2,g_2)= a_2= \mathsf{t}(h_1)\) so we can write \(h_2 h_1\).
Second we have \(\mathsf{t}^{\ltimes}(h_1,a_1, g_1)=(a_1,g_1)\) and \(\mathsf{s}^{\ltimes}(h_2, a_2, g_2) = h_2 \lhaction (a_2, g_2)=(\mathsf{t}(h_2), \upphi_{ \scriptscriptstyle{1} }(h_2)a_2)\), and so 
\[ \begin{gathered}
\xymatrix{(a_1,g_1)  &&& \Big( h_1 \lhaction (a_1, g_1) \Big) = \left(\mathsf{t}(h_1), \upphi_{ \scriptscriptstyle{1} }(h_1)g_1 \right) \ar[lll]_{(h_1,a_1,g_1)} &&& \Big(  h_2  \lhaction (a_2, g_2) \Big)  \ar[lll]_{(h_2, a_2, g_2)} } \\
= \left(\mathsf{t}(h_2), \upphi_{ \scriptscriptstyle{1} }(h_2) g_2 \right) 
=\left( \mathsf{t}(h_2 h_1), \upphi_{ \scriptscriptstyle{1} }(h_2) \upphi_{ \scriptscriptstyle{1} }(h_1) g_1 \right)
= \left( \mathsf{t}(h_2 h_1) , \upphi_{ \scriptscriptstyle{1} }(h_2 h_1) g_1 \right)
= \Big( (h_2 h_1)  \lhaction (a_1, g_1) \Big) .
\end{gathered}
\]
Thus,
\[ (h_1, a_1, g_1)(h_2, a_2, g_2)=(h_2 h_1, a_1, g_1).
\]

\begin{definition}
Given a morphism of groupoids \(\upphi \colon \mathcal{H} \longrightarrow \mathcal{G}\), we define
\[ \left( \mathcal{G} / \mathcal{H} \right)^{\Sscript{\mathsf{R}}}_{ \scriptscriptstyle{\upphi} } 
:= \pi_{ \scriptscriptstyle{0} } \Big( \mathcal{H} \ltimes { ^{ \scriptscriptstyle{\upphi} } \mathcal{H}  } \left(\mathcal{G}\right) \Big)
\]
the orbit set $\cH \backslash { ^{ \scriptscriptstyle{\upphi} } \mathcal{H}  } \left(\mathcal{G}\right)$  and, for each \((a,g) \in \mathcal{H}_{ \scriptscriptstyle{0} }  {{}_{ \scriptscriptstyle{\upphi_{ \scriptscriptstyle{0} }} } {\times}}_{ \scriptscriptstyle{\mathsf{t}} }\, \mathcal{G}_{ \scriptscriptstyle{1} }\), we set 
\[  {^{ \scriptscriptstyle{\upphi} } \mathcal{H}[(a,g)]} 
= \Set{ \Big( h \lhaction (a,g) \Big) \in {^{ \scriptscriptstyle{\upphi} } \mathcal{H} }\left(\mathcal{G}\right) | h \in \mathcal{H}_{ \scriptscriptstyle{1} }, \, \mathsf{s}(h)=a }.
\]
\end{definition}

If \(\mathcal{H}\) is a subgroupoid of \(\mathcal{G}\), that is,  \(\upphi:=\tauup \colon \mathcal{H} \hookrightarrow \mathcal{G}\) is the inclusion functor,  we use the notation
\[ \left( \mathcal{G} /\mathcal{H} \right)^{\Sscript{\mathsf{R}}} 
= \mathcal{H} \backslash \left({\mathcal{H}_{ \scriptscriptstyle{0} }  {{}_{ \, \scriptscriptstyle{\tauup_0} } {\times}}_{ \scriptscriptstyle{ \mathsf{t}} }\,  \mathcal{G}_{ \scriptscriptstyle{1} } }\right) 
=  \mathcal{H}\left({\mathcal{G}}\right) = {^{ \scriptscriptstyle{\tauup} }\mathcal{H}} \left({\mathcal{G}}\right) 
\]
and
\begin{equation}\label{Eq:Hag}
\Hlcoset{(a,g)}={^{ \scriptscriptstyle{\tauup} }\mathcal{H}} [ \left({a,g}\right) ]
= \LR{ \Big( h \lhaction (a,g) \Big) \in \mathcal{H}\left({\mathcal{G}}\right) | \; h \in \mathcal{H}_{ \scriptscriptstyle{1} }, \, \mathsf{s}(h)=a },
\end{equation}
where \((a,g) \in \mathcal{H}_{ \scriptscriptstyle{0} }  {{}_{\,  \scriptscriptstyle{\tauup_{ \scriptscriptstyle{0} }} } {\times}}_{ \scriptscriptstyle{\mathsf{t}} }\, \mathcal{G}_{ \scriptscriptstyle{1} }\).
On the other hand, for each \((h,a,g) \in \mathcal{H}_{ \scriptscriptstyle{1} }  {{}_{ \scriptscriptstyle{\mathsf{s}} } {\times }}_{ \scriptscriptstyle{\pr_1} }\,  \mathcal{H}_{ \scriptscriptstyle{0} }  {{}_{ \, \scriptscriptstyle{\tauup_0} } {\times}}_{ \scriptscriptstyle{\mathsf{t}} }\, \mathcal{G}_1\) we have
\[ \Big( h \lhaction (a,g) \Big) = \left(t\left({h}\right), \upphi_{ \scriptscriptstyle{1} }\left({h}\right)g\right)
= \left(t\left({h}\right), hg\right)
\qquad {\text{\normalsize  and}} \qquad
 \mathsf{s}\left({h}\right)= a= \upphi_{ \scriptscriptstyle{0} }\left({a}\right)= \mathsf{t}\left({g}\right) 
\]
so, for each \((a,g) \in \mathcal{H}_{ \scriptscriptstyle{0} }  {{}_{ \, \scriptscriptstyle{\tauup_{ \scriptscriptstyle{0} }} } {\times}}_{ \scriptscriptstyle{\mathsf{t}} }\, \mathcal{G}_{ \scriptscriptstyle{1} }\), we have
\[ \Hlcoset{(a,g)}= \LR{ \left( \mathsf{t}\left({h}\right), hg\right) \in \mathcal{H}\left({\mathcal{G}}\right)  |  h \in \mathcal{H}_{ \scriptscriptstyle{1} }, \, \mathsf{s}\left({h}\right)= \mathsf{t}\left({g}\right)  }.
\]

\begin{definition}\label{def:coset}
Let $\cH$ be a subgroupoid of $\cG$ via the injection $\tau:\cH \hookrightarrow \cG$. The \emph{right coset of $\cG$ by $\cH$} is defined as:
\begin{equation}\label{Eq:R} 
\left({\mathcal{G}/\mathcal{H}}\right)^{ \scriptscriptstyle{\mathsf{R}} } = \Set{ \Hlcoset{(a,g)}  | (a,g) \in \mathcal{H}_{ \scriptscriptstyle{0} }  {{}_{\, \scriptscriptstyle{\tauup_{ \scriptscriptstyle{0} }} } {\times}}_{ \scriptscriptstyle{\mathsf{t}} }\, \mathcal{G}_{ \scriptscriptstyle{1} } ,  }
\end{equation}
where each class $\cH[(a,g)]$ is as in equation \eqref{Eq:Hag}.
\end{definition}

Keeping the notation of the previous definition we can state:

\begin{lemma}
\label{lEqCosets}
Given \((a_1, g_1), (a_2, g_2) \in \mathcal{H}_{ \scriptscriptstyle{0} }  {{}_{ \scriptscriptstyle{\tauup_{ \scriptscriptstyle{0} }} } {\times}}_{ \scriptscriptstyle{\mathsf{t}} }\, \mathcal{G}_{ \scriptscriptstyle{1} }\), we have \(\Hlcoset{(a_1, g_1)} = \Hlcoset{(a_2, g_2)}\) if and only if there is \(h \in {\mathcal{H}}\left(a_2,a_1\right)\) such that \(h=g_1 g_2^{-1}\).
\end{lemma}
\begin{proof}
We have \(\Hlcoset{(a_1, g_1)} = \Hlcoset{(a_2, g_2)}\) if and only if \((a_1, g_1) \in \Hlcoset{(a_2, g_2)}\), if and only if there is \(h \in \mathcal{H}_{ \scriptscriptstyle{1} }\) such that \(\mathsf{s}(h)=\mathsf{t}(g_2)\) and \((a_1, g_1)= (\mathsf{t}(h), hg_2)\), if and only if there is \(h \in \mathcal{H}_{ \scriptscriptstyle{1} }\) such that
\[ \left\lbrace  \begin{aligned}
& \mathsf{t}(g_1)=a_1=\mathsf{t}(h)   \\
& g_1 = h g_2,
\end{aligned}  \right.
\]
if and only if there is \(h \in \mathcal{H}_{ \scriptscriptstyle{1} }\) such that \(\mathsf{s}(h)=\mathsf{t}(g_2)=a_2\), \(\mathsf{t}(h)=\mathsf{t}(g_1)=a_1\), \(h= g_1 g_2^{-1}\), if and only if there is \(h \in {\mathcal{H}}\left(a_1,a_2\right)\) such that \(h=g_1 g_2^{-1}\).
\end{proof}

The \emph{left coset of $\cG$ by $\cH$} is defined using the $(\cG, \cH)$-biset $\cH(\cG)^{\Sscript{\tauup}}:= \Ga\, \due \times {\Sscript{\Sf{s}}} {\, \Sscript{\tauup_0}} \,  \Ho$  of Example \ref{exam:bisets}  with actions maps as in  equation \eqref{Eq:GH}. 
If \(\tauup \colon \mathcal{H} \longrightarrow \mathcal{G}\) is the inclusion functor, then we use the notation
\begin{equation}\label{Eq:L}
\left({\mathcal{G}/\mathcal{H}}\right)^{\Sscript{\mathsf{L}} } = \cH(\cG)^{\Sscript{\tauup}} / \mathcal{H} 
= \LR{ \Hrcoset{(g,u)}   |   \; \left({g,u}\right) \in \cH(\cG)^{\Sscript{\tauup}} }
\end{equation}
and, for each \(\left({g,u}\right) \in \cH(\cG)^{\Sscript{\tauup}}\),
\[ 
\Hrcoset{(g,u)}  = \LR{ \left({gh, \mathsf{s}\left({h}\right) }\right) \in \cH(\cG)^{\Sscript{\tauup}}:=\Ga\, \due \times {\Sscript{\Sf{s}}} {\, \Sscript{\tauup_0}} \,  \Ho |  \; h \in \mathcal{H}_{ \scriptscriptstyle{1} }  , \mathsf{s}(g)=\mathsf{t}(h)}.
\]
The following is an  analogue of Lemma \ref{lEqCosets} 
\begin{lemma}
\label{lEquLatLeft}
If \(\tauup \colon \mathcal{H} \longrightarrow \mathcal{G}\) is an inclusion functor, then for each \(\left({g_1, u_1}\right), \left({g_2, u_2}\right) \in  \mathcal{G}_{ \scriptscriptstyle{1} }  {{}_{ \scriptscriptstyle{\mathsf{s} } } {\times}}_{ \scriptscriptstyle{\tauup_0} }\, \mathcal{H}_{ \scriptscriptstyle{0} }\) we have \(  \Hrcoset{(g_1, u_1)} = \Hrcoset{(g_2, u_2)} \) if and only if \(g_2^{-1} g_1 \in \mathcal{H}_{ \scriptscriptstyle{1} }\).
\end{lemma}
\begin{proof}
Is similar to that of  Lemma \ref{lEqCosets}.
\end{proof}

As a corollary of Proposition \ref{pBisetRight}, we obtain:
\begin{corollary}
\label{cRightCosetsRightSet}
Let $\cH$ be a subgroupoid of $\cG$ via the injection $\tau:\cH \hookrightarrow \cG$. Then
\(\left({\mathcal{G}/\mathcal{H}}\right)^{\Sscript{\mathsf{R}}}\) becomes a right \(\mathcal{G}\)-set with structure map and action given as follows:
\begin{equation}\label{Eq:PalomaI}
\begin{aligned}{{\varsigma}  }  \colon & {\left({\mathcal{G}/\mathcal{H} }\right)^{\Sscript{\mathsf{R}}} } \longrightarrow {\mathcal{G}_{ \scriptscriptstyle{0} } } \\ &  \Hlcoset{(a, g)}  \longrightarrow {\mathsf{s}\left({g}\right) }\end{aligned} 
\qquad \text{and} \qquad
\begin{aligned}{{\rho} }  \colon & {\left({\mathcal{G}/\mathcal{H} }\right)^{\Sscript{\mathsf{R}} }  {{}_{ \scriptscriptstyle{{\varsigma} } } {\times}}_{ \scriptscriptstyle{\mathsf{t} } }\, \mathcal{G}_{ \scriptscriptstyle{1} } } \longrightarrow {\left({\mathcal{G}/\mathcal{H} }\right)^{\Sscript{\mathsf{R}} } } \\ & \big( \Hlcoset{(a,g_1)}, g_2\big)   \longrightarrow \Hlcoset{(a, g_1 g_2)}. 
\end{aligned} 
\end{equation}
\end{corollary}

In a similar way  the left coset $(\cG/\cH)^{\Sscript{\Sf{L}}}$ becomes a left $\cG$-set with a structure and action maps given by: 

\begin{equation}\label{Eq:Paloma}
 \begin{aligned}{{\vartheta}  }  \colon & {\left({\mathcal{G}/\mathcal{H} }\right)^{\Sscript{\mathsf{L}}} } \longrightarrow {\mathcal{G}_{ \scriptscriptstyle{0} } } \\ &  \Hrcoset{(g, u)}  \longrightarrow {\mathsf{t}\left({g}\right) }\end{aligned} 
\qquad \text{and} \qquad
\begin{aligned}{{\lambda} }  \colon & {\mathcal{G}_{ \scriptscriptstyle{1} }  {{}_{  \scriptscriptstyle{\mathsf{s} } } {\times}}_{    \scriptscriptstyle{{\vartheta} }  }\, \left({\mathcal{G}/\mathcal{H} }\right)^{\Sscript{\mathsf{L}} }  } \longrightarrow {\left({\mathcal{G}/\mathcal{H} }\right)^{\Sscript{\mathsf{L}} } } \\ & \big(g_1, \Hrcoset{(g_2,u)}\big)   \longrightarrow \Hrcoset{(g_1g_2, u)}. \end{aligned} 
\end{equation}

The following crucial proposition characterizes, as in the classical case of groups,  the right cosets by the stabilizer subgroupoid.  
\begin{proposition}
\label{pIsomStabOrb}
Let \((X, \varsigma)\) be a right \(\mathcal{G}\)-set with  action map $\rho \colon X  {{}_{ \scriptscriptstyle{\varsigma} } {\times}}_{ \scriptscriptstyle{ \mathsf{t} } }\, \mathcal{G}_{ \scriptscriptstyle{1} } \longrightarrow X$. 
Given \(x \in X\), let us consider \(\mathcal{H}= \Stabi_{ \scriptscriptstyle{\mathcal{G}}}\left({x}\right) \) its stabilizer as a subgroupoid of $\cG$ (see Subsection \ref{ssec:Orbits}). 
Then  the following map 
$$
\xymatrix@R=0pt{  \varphi:  \;  {\left({\mathcal{G}/\mathcal{H} }\right)^{\Sscript{\mathsf{R}} } }  \ar@{->}[rr] & &  {[x] \mathcal{G} } \\ \Hlcoset{(a,g)}  \ar@{|->}[rr]  & &  xg  }
$$
establishes an  isomorphism of  right \(\mathcal{G}\)-sets.
Likewise, a similar statement is true for left cosets. That is, for a given left $\cG$-set $(Y, \vartheta)$ with action $\lambda$, we have, for every $y \in Y$, an isomorphism of left $\cG$-set 
$$
\xymatrix@R=0pt{  \psi:  \;  {\left({\mathcal{G}/\mathcal{H} }\right)^{\Sscript{\mathsf{L}} } }  \ar@{->}[rr] & &  { \mathcal{G}[y] } \\ \Hrcoset{(u,a)}  \ar@{|->}[rr]  & &  ay,  }
$$
where $\cH=\Stabi_{ \scriptscriptstyle{\mathcal{G}}}\left({y}\right)$ is the stabilizer of $y$.
\end{proposition}
\begin{proof}
We only show the right side of the statement. Given \( \Hlcoset{(a,g)} \in \left({\mathcal{G}/\mathcal{H}}\right)^{\Sscript{\mathsf{R}} }\), we have \(a \in \mathcal{H}_{ \scriptscriptstyle{0} }= \Set{\varsigma\left({x}\right) }\) and, since \(\left({a,g}\right) \in \mathcal{H}_{ \scriptscriptstyle{0} }  {{}_{ \, \scriptscriptstyle{\tauup_0} } {\times }}_{ \scriptscriptstyle{\mathsf{t} } }\, \mathcal{G}_{ \scriptscriptstyle{1} }\), where \(\tauup \colon \mathcal{H} \longrightarrow \mathcal{G}\) is the inclusion functor, we have \(\mathsf{t}\left({g}\right)= a = \varsigma\left({x}\right)\) and we can write \(x g\).
Now consider \(\left({a_1, g_1}\right), \left({a_2, g_2}\right) \in \mathcal{H}_{ \scriptscriptstyle{0} }   {{}_{ \, \scriptscriptstyle{\tauup_0} } {\times}}_{ \scriptscriptstyle{\mathsf{t } } }\, \mathcal{G}_{ \scriptscriptstyle{1} }\) (that is, \(a_1= \mathsf{t}\left({g_1}\right)\) and \(a_2=\mathsf{t}\left({g_2}\right)\)) such that \(  \Hlcoset{(a_{1},g_{1})} =  \Hlcoset{(a_{2},g_{2})} \).
Then, by Lemma~\(\ref{lEqCosets}\), there exists \(h \in \mathcal{H}_{ \scriptscriptstyle{1} }\) such that \(\mathsf{s}\left({h}\right)= a_2\), \(\mathsf{t}\left({h}\right)=a_1\) and \(h=g_1 g_2^{-1}\).
Since \(\mathcal{H}= \Stabi_{ \scriptscriptstyle{\mathcal{G}}}\left({x}\right) \) we have \(a_1=a_2= \varsigma\left({x}\right)\) and \(x h= x\) so \(x = x h= x g_1 g_2^{-1}\), whence \(x g_2 = x g_1\).
This shows that \(\varphi\) is well defined.
Now let \(   \Hlcoset{(a_{1},g_{1})} , \Hlcoset{(a_{2},g_{2})}   \in \left({\mathcal{G}/\mathcal{H}}\right)^{\Sscript{\mathsf{R}}}\) such that \(\varphi\big( \Hlcoset{(a_{1},g_{1})}  \big)=  \varphi\big( \Hlcoset{(a_{2},g_{2})} \big)\). Then 
we have 
\[ 
x g_1 = \varphi\big( \Hlcoset{(a_{1},g_{1})}  \big)=  \varphi\big( \Hlcoset{(a_{2},g_{2})} \big)
= x g_2,
\]
so \(x g_1 g_2^{-1} = x\), which means that  \(g_1 g_2^{-1} \in \Stabi_{ \scriptscriptstyle{\mathcal{G}}}\left({x}\right) =\mathcal{H}\).
Therefore \( \Hlcoset{(a_{1},g_{1})} =  \Hlcoset{(a_{2},g_{2})} \) and \(\varphi\) is  then injective.
Now consider an element \(x g \in \Grcoset{x} \), by definition we have \( \varphi\left( \Hlcoset{(\varsigma(x), g)} \right)= x g\), so \(\varphi\) is  a surjective map. Therefore, $\varphi$ is bijective. 

By Corollary \ref{cRightCosetsRightSet} and Lemma \ref{lActionRestriction} it follows that \(\left({\mathcal{G}/\mathcal{H}}\right)^{\Sscript{\mathsf{R}}}\) and \( \Grcoset{x}\) are right \(\mathcal{G}\)-sets. We denote by $\varsigma_{\Sscript{x}}$ and $\rho_{\Sscript{x}}$ the structure and the action maps of $\Grcoset{x}$, respectively.
To prove that \(\varphi\) is a morphism of right \(\mathcal{G}\)-set we have to prove that the following two diagrams are commutative:
\[ \xymatrix{
\left({\mathcal{G}/\mathcal{H}}\right)^{\Sscript{\mathsf{R}}}  {{}_{ \scriptscriptstyle{\widehat{\varsigma} } } {\times}}_{ \scriptscriptstyle{\mathsf{t}} }\, \mathcal{G}_{ \scriptscriptstyle{1} }  \ar[rr]^{\widehat{\rho} } \ar[d]_{\varphi \, \times \,  \Id_{\mathcal{G}_{ \scriptscriptstyle{1} } } } & &  \left({\mathcal{G}/\mathcal{H} }\right)^{\Sscript{\mathsf{R}} }  \ar[d]^{\varphi}  \\
\left( \Grcoset{x} \right)  {{}_{ \scriptscriptstyle{\varsigma_{\Sscript{x}}} } {\times}}_{ \scriptscriptstyle{\mathsf{t}} }\, \mathcal{G}_{ \scriptscriptstyle{1} }   \ar[rr]^{\rho_{\Sscript{x}}} & &  \Grcoset{x}   
}
\qquad \text{and} \qquad
\xymatrix{
\left({\mathcal{G}/\mathcal{H} }\right)^{\Sscript{\mathsf{R}} }  \ar[d]_{\varphi}  \ar[rr]^{\widehat{\varsigma}} &  & \mathcal{G}_{ \scriptscriptstyle{0} }   \\
\Grcoset{x},   \ar[rru]_{\varsigma_{\Sscript{x}}}   && }
\]
where we used the notation $\widehat{\varsigma}$ and $\widehat{\rho}$ for the structural and the action maps given in \eqref{Eq:PalomaI}. 
Let us check the commutativity of the triangle. So take  \( \Hlcoset{(a,g)} \in \left({\mathcal{G}/\mathcal{H}}\right)^{\Sscript{\mathsf{R}} }\),  using the definition of right action,  we have 
\[ \varsigma_{\Sscript{x}} \varphi \left( \Hlcoset{(a,g)}  \right) 
= \varsigma_{\Sscript{x}} \left({x g}\right) = \varsigma\left({x g}\right)
= \mathsf{s}\left({g}\right) =  \widehat{\varsigma} \left(  \Hlcoset{(a,g)}  \right) 
\]
since \(\mathsf{t}\left({g}\right)= a = \varsigma\left({x}\right)\). As for the rectangle, take an arbitrary element
\(\Big(\Hlcoset{(a,g)}, g_1\Big) \in \left({\mathcal{G}/\mathcal{H}}\right)^{\Sscript{\mathsf{R}}}  {{}_{ \scriptscriptstyle{\widehat{\varsigma} } } {\times}}_{ \scriptscriptstyle{\mathsf{t}} }\, \mathcal{G}_{ \scriptscriptstyle{1} }\), we have
\[ \begin{gathered}
\rho_{\Sscript{x}} \left(\varphi \times \Id_{\mathcal{G}_{ \scriptscriptstyle{1} } } \right) \Big( \Hlcoset{(a,g)}, g_1 \Big)
= \rho_{\Sscript{x}} \left( \varphi\left( \Hlcoset{(a,g)} \right), g_1 \right) \\
= \rho_{\Sscript{x}} \left({x g, g_1}\right) = \rho\left({x g, g_1}\right) 
= \left({x g}\right) g_1 
= x \left({g g_1}\right) 
= \varphi\left( \Hlcoset{(a,gg_{1})}  \right) 
= \varphi \widehat{\rho}\big( \Hlcoset{(a,g)}  , g_1\big).
\end{gathered}
\]
Therefore, $\varphi$ is compatible with the action and by using Lemma \ref{lInvEquivaIsEquiva}, we conclude that \(\varphi\) is an isomorphism of right \(\mathcal{G}\)-set as desired. 
\end{proof}

\begin{corollary}\label{lema:LAquila}
Let $(X, \varsigma)$ be a right $\cG$-set. Then there is an isomorphism of right $\cG$-sets:
$$
X \, \cong \, \biguplus_{x\, \in \,   \rep_{\scriptscriptstyle{\mathcal{G}}}(X) } \left(\mathcal{G}/ \Stabi_{ \scriptscriptstyle{\mathcal{G}}}\left({x}\right) \right)^{\Sscript{\mathsf{R}} }.
$$
where the right hand side is the coproduct in the category of right $\cG$-sets and with \(\rep_{\scriptscriptstyle{\mathcal{G}}}(X)\) we indicate a set of representatives of the orbits of the right \(\mathcal{G}\)-set \(X\).
\end{corollary}
\begin{proof}
Immediate from Proposition~\(\ref{pIsomStabOrb}\), considering the fact that a right \(\mathcal{G}\)-set is a disjoint union of its orbits.
\end{proof}

\subsection{Groupoid-bisets versus (left) sets}\label{ssec:Bisets} 
In this subsection we give the complete proof of the fact that the category of groupoids $(\cH,\cG)$-bisets is isomorphic to the category of left groupoids $(\cH \times \cGop)$-sets  (equivalently right $(\cHop \times \cG)$-sets). Here the structure groupoid of the (cartesian) product of two groupoids is the one given by the product of the underlying categories as described in Example \ref{exam:trivial}. 

\begin{proposition}
\label{pBisetsLeftSets}
Given a set \(X\) and two groupoids \(\mathcal{H}\) and \(\mathcal{G}\). Then there is a bijective correspondence between structures of \(\left({\mathcal{H},\mathcal{G}}\right)\)-bisets on \(X\) and structures of left \(\left({\mathcal{H} \times \cGop}\right)\)-sets on \(X\).
\end{proposition}
\begin{proof}
Let \(X\) be an \(\left({\mathcal{H}, \mathcal{G}}\right)\)-biset with actions and structures map as follows:
\[ \begin{aligned}
& \vartheta \colon X \longrightarrow \mathcal{H}_{ \scriptscriptstyle{0} } \\
& \lambda \colon \mathcal{H}_{ \scriptscriptstyle{1} }  {{}_{ \scriptscriptstyle{\mathsf{s} } } {\times }}_{ \scriptscriptstyle{\vartheta} }\, X \longrightarrow X
\end{aligned}
\qquad \text{and} \qquad
\begin{aligned}
& \varsigma \colon X \longrightarrow \mathcal{G}_{ \scriptscriptstyle{0} } \\
& \rho \colon X  {{}_{ \scriptscriptstyle{\varsigma} } {\times }}_{ \scriptscriptstyle{\mathsf{t} } }\, \mathcal{G}_{ \scriptscriptstyle{1} } \longrightarrow X.
\end{aligned}
\]
We define the structure map and action as follows
\[ \begin{aligned}{\alpha}  \colon & {X} \longrightarrow {\left({\mathcal{H} \times \Gop}\right)_{ \scriptscriptstyle{0} }  } \\ & {x}  \longmapsto {\left({\vartheta\left({x}\right), \varsigma\left({x}\right) }\right) }\end{aligned} 
\qquad \text{and} \qquad
\begin{aligned}{\beta}  \colon & {\left({\mathcal{H} \times \Gop}\right)_{ \scriptscriptstyle{1} }  {{}_{ \scriptscriptstyle{ \mathsf{s} } } {\times}}_{ \scriptscriptstyle{\alpha} }\, X} \longrightarrow {X} \\ & {\left({(h,g),x}\right) }  \longmapsto {h\left({xg}\right) .}\end{aligned} 
\]
Now the verification that \(\left(X, \alpha\right)\) is a left \(\left({\mathcal{H} \times \Gop}\right)\)-set is obvious.

Conversely,  let \(X\) be an \(\left({\mathcal{H} \times \Gop}\right)\)-left set with
\[ \alpha \colon X \longrightarrow \left({\mathcal{H} \times \Gop}\right)_{ \scriptscriptstyle{0} }  
\qquad \text{and} \qquad
\beta \colon \left({\mathcal{H} \times \Gop}\right)_{ \scriptscriptstyle{1} }  {{}_{ \scriptscriptstyle{ \mathsf{s} } } {\times}}_{ \scriptscriptstyle{\alpha} }\, X \longrightarrow X
\]
as structure map and action.
Let be \(p_1\) and \(p_2\) the canonical projections
\[ \begin{aligned}{p_1}  \colon & { \left({\mathcal{H} \times \Gop}\right)_{ \scriptscriptstyle{0} } } \longrightarrow {\mathcal{H}_{ \scriptscriptstyle{0} } } \\ & {\left({h,g}\right) }  \longmapsto {h}\end{aligned} 
\qquad \text{and} \qquad
 \begin{aligned}{p_2}  \colon & { \left({\mathcal{H} \times \Gop}\right)_{ \scriptscriptstyle{0} } } \longrightarrow {\mathcal{G}_{ \scriptscriptstyle{0} } } \\ & {\left({h,g}\right) }   \longmapsto {g.}\end{aligned} 
\]
We define a structure of left \(\mathcal{H}\)-set as follows
\[ \begin{aligned}{\vartheta}  \colon & {X} \longrightarrow {\mathcal{H}_{ \scriptscriptstyle{0} } } \\ & {x}   \longmapsto {p_1 \alpha\left({x}\right) }\end{aligned} 
\qquad \text{and} \qquad
\begin{aligned}{\lambda}  \colon & {\mathcal{H}_{ \scriptscriptstyle{1} }  {{}_{ \scriptscriptstyle{\mathsf{s} } } {\times}}_{ \scriptscriptstyle{\vartheta} }\, X } \longrightarrow {X} \\ & {\left({x,g}\right) }   \longmapsto   {\beta\left({(h, \iota_{ \scriptscriptstyle{\varsigma \left({x}\right) } }) , x }\right)}\end{aligned} 
\]
and a structure of right \(\mathcal{G}\)-set as follows:
\[ \begin{aligned}{\varsigma}  \colon & {X} \longrightarrow {\mathcal{G}_{ \scriptscriptstyle{0} } } \\ & {x}   \longmapsto {p_2 \alpha\left({x}\right) }\end{aligned} 
\qquad \text{and} \qquad
\begin{aligned}{\rho}  \colon & {X  {{}_{ \scriptscriptstyle{\varsigma} } {\times}}_{ \scriptscriptstyle{\mathsf{t} } }\,\mathcal{G}_{ \scriptscriptstyle{1} } } \longrightarrow {X} \\ & {\left({x,g}\right) }   \longmapsto {\beta\left({(\iota_{ \scriptscriptstyle{\vartheta\left({x}\right) } } , g), x }\right).}\end{aligned} 
\]
For each \(\left({h,x}\right) \in \mathcal{H}_{ \scriptscriptstyle{1} }  {{}_{ \scriptscriptstyle{\mathsf{s} } } {\times}}_{ \scriptscriptstyle{\vartheta} }\, X\) we have
\[ \alpha\left({x}\right) = \left({\vartheta\left({x}\right), \varsigma\left({x}\right) }\right)
= \left({\mathsf{s}\left({h}\right), \mathsf{t}\left({\iota_{ \scriptscriptstyle{\varsigma\left({x }\right) } } }\right) }\right)
 = \mathsf{s}\left({h, \iota_{ \scriptscriptstyle{\varsigma\left({x}\right) } } }\right) 
\]
so \(\lambda\) is well defined.
Now the verification that \(\left(X, \vartheta\right)\) is a left \(\mathcal{H}\)-set is obvious.
As for the  right action, for each \(\left({x,g}\right) \in X  {{}_{ \scriptscriptstyle{\varsigma} } {\times}}_{ \scriptscriptstyle{\mathsf{t} } }\, \mathcal{G}_{ \scriptscriptstyle{1} }\), we have 
\[ \alpha\left({x}\right) = \left({\vartheta\left({x}\right), \varsigma\left({x}\right) }\right) 
= \left({\mathsf{s}\left({\iota_{ \scriptscriptstyle{\vartheta\left({x}\right) } } }\right), \mathsf{t}\left({g}\right) }\right) 
= \mathsf{s}\left({\iota_{ \scriptscriptstyle{\vartheta\left({x}\right) } }, g}\right) 
\]
so \(\rho\) is well defined.
The verification that \(\left(X, \varsigma\right)\) is a right \(\mathcal{G}\)-set is also obvious.
Now we only have to verify the properties of a biset but this is easy and it is left to the reader.
As a consequence, $X$ is an $(\cH,\cG)$-biset. 
Lastly, it is clear that these two constructions are mutually  inverse  and this completes the proof.
\end{proof}

A similar  proof to that of Proposition \ref{pBisetsLeftSets}, works to show that there is a one-to-one correspondence between right $(\cHop \times \cG)$-set structure and $(\cH, \cG)$-set structure. Furthermore, any $(\cH, \cG)$-equivariant map (i.e., any morphism of $(\cH, \cG)$-bisets) is canonically transformed, under this correspondence,  to a left $(\cH \times \cGop)$-equivariant map. In this way, we have  the following corollary. 

\begin{corollary}\label{coro:iso}
Let $\cH$ and $\cG$ be two groupoids. Then there are canonical isomorphisms of categories between the category of $(\cH, \cG)$-bisets, left $(\cH \times \cGop)$-sets and  right $(\cHop \times \cG)$-sets. 
\end{corollary}

\subsection{Orbits and stabilizers of bisets  and double cosets}\label{ssec:StabOrbBiset}

We will use the notations of the proof of Proposition~\(\ref{pBisetsLeftSets}\).
Let \(X\) be an \(\left({\mathcal{H}, \mathcal{G}}\right)\)-biset: it becomes an \(\left({\mathcal{H} \times \cGop }\right)\)-left set, so we have
\[ \left(\Stabi_{ \scriptscriptstyle{\left({\mathcal{H}, \, \mathcal{G} }\right) }}\left({x}\right)  \right)_{ \scriptscriptstyle{0} } 
= \left( \Stabi_{ \scriptscriptstyle{\left({\mathcal{H} \, \times \,  \cGop}\right) }}\left({x}\right)  \right)_{ \scriptscriptstyle{0} } 
= \Set{\left({\vartheta\left({x}\right), \varsigma\left({x}\right) }\right) } 
\]
and
\[ \begin{aligned}
\left(\Stabi_{ \scriptscriptstyle{\left({\mathcal{H}, \, \mathcal{G} }\right) }}\left({x}\right)  \right)_{ \scriptscriptstyle{1} } 
&= \left( \Stabi_{ \scriptscriptstyle{\left({\mathcal{H} \, \times \, \cGop }\right) }}\left({x}\right)  \right)_{ \scriptscriptstyle{1} }  
= \Set{\left({h, g}\right) \in \mathcal{H}_{ \scriptscriptstyle{1} } \times \Gop_{ \scriptscriptstyle{1} } | \begin{gathered}
\mathsf{s}\left({h,g}\right) = \mathsf{t}\left({h,g}\right) = \alpha\left({x}\right) \\
\left({h,g}\right) x = x
\end{gathered} }  \\
& = \Set{\left({h, g}\right) \in \mathcal{H}_{ \scriptscriptstyle{1} } \times \Gop_{ \scriptscriptstyle{1} } | \begin{gathered}
\mathsf{s}\left({h}\right) = \mathsf{t}\left({h}\right) = \vartheta\left({x}\right) \\
\mathsf{s}\left({g}\right) = \mathsf{t}\left({g}\right) = \varsigma\left({x}\right) \\
h xg= x
\end{gathered} } .
\end{aligned}
\]

For a given elements $x \in X$, the orbit set of $x$ is given by  
\[ 
\Orb_{\Sscript{\left({\mathcal{H}, \,  \mathcal{G} }\right) }}(x) = \mathcal{H}[ x ] \mathcal{G} = \LR{ hxg = \left({h,g}\right)x \in X | \;\; \mathsf{s}\left({h}\right)= \vartheta\left({x}\right), \; \varsigma\left({x}\right)= \mathsf{t}\left({g}\right)  }.
\]

\begin{proposition}\label{prop:Doublecoset}
Given a groupoid \(\mathcal{H}\), let \(\mathcal{A}\) and \(\mathcal{B}\) be subgroupoid of \(\mathcal{H}\).
We define
\[ X= \mathcal{A}_{ \scriptscriptstyle{0} }  {{}_{ \scriptscriptstyle{ } } {\times}}_{ \scriptscriptstyle{\mathsf{t} } }\, \mathcal{H}  {{}_{ \scriptscriptstyle{ \mathsf{s} } } {\times}}_{ \scriptscriptstyle{ } }\, \mathcal{B}_{ \scriptscriptstyle{0} } = \LR{ \left({a,h,b}\right) \in \mathcal{A}_{ \scriptscriptstyle{0} } \times \mathcal{H}_{ \scriptscriptstyle{1} } \times \mathcal{B}_{ \scriptscriptstyle{0} }  | \; a = \mathsf{t}\left({h}\right) , \mathsf{s}\left({h}\right)= b }.
\]
Then \(X\) is an \(\left({\mathcal{A}, \mathcal{B}}\right)\)-biset with structure maps
\[ \begin{aligned}{\vartheta}  \colon & {X} \longrightarrow {\mathcal{A}_{ \scriptscriptstyle{0} } } \\ & {\left({a,h,b}\right) }   \longmapsto {a}\end{aligned} 
\qquad \text{and} \qquad
\begin{aligned}{\varsigma}  \colon & {X} \longrightarrow {\mathcal{B}_{ \scriptscriptstyle{0} } } \\ & {\left({a,h,b}\right) }   \longmapsto {b}\end{aligned} 
\]
and action maps
\[ \begin{aligned}{\lambda}  \colon & {\mathcal{A}_{ \scriptscriptstyle{1} }  {{}_{ \scriptscriptstyle{\mathsf{s} } } {\times}}_{ \scriptscriptstyle{\vartheta} }\, X} \longrightarrow {X} \\ & {\left({r, \left({a,h,b}\right) }\right) }  \longmapsto {\left({\mathsf{t}\left({r}\right), r h, b}\right) }\end{aligned}
\qquad \text{and} \qquad
 \begin{aligned}{\rho}  \colon & {X  {{}_{ \scriptscriptstyle{\varsigma} } {\times}}_{ \scriptscriptstyle{\mathsf{t} } }\, \mathcal{B}_{ \scriptscriptstyle{1} } } \longrightarrow {X} \\ & {\left({\left({a, h, b}\right), q}\right) }  \longmapsto {\left({a, hq, \mathsf{s}\left({q}\right) }\right) .}\end{aligned}  
\]
\end{proposition}
\begin{proof}
We have to check the properties of a groupoid right action.
\begin{enumerate}
\item For each \(\left({a,h,b}\right) \in X\) and \(q \in \mathcal{B}_{ \scriptscriptstyle{1} }\) such that \(\varsigma\left({a,h,b}\right)= \mathsf{t}\left({q}\right)\) we have
\[ \varsigma\left({\left({a,h,b}\right)d}\right)= \varsigma\left({a, hq, \mathsf{s}\left({q}\right) }\right) 
= \mathsf{s}\left({q}\right).
\]
\item For each \(\left({a,h,b}\right) \in X\) we have
\[ \left({a,h,b}\right) \iota_{ \scriptscriptstyle{\varsigma\left({a,h,b}\right) } } = \left({a,h,b}\right) \iota_{ \scriptscriptstyle{b} } 
= \left({a, h \iota_{ \scriptscriptstyle{b} }, \mathsf{s}\left({\iota_{ \scriptscriptstyle{b} } }\right) }\right) 
= \left({a, h, b}\right).
\]
\item For each \(\left({a,h,b}\right) \in X\) anq \(q,q' \in \mathcal{B}_{ \scriptscriptstyle{1} }\) such that \(\varsigma\left({a,h,b}\right)= \mathsf{t}\left({q}\right)\) anq \(\mathsf{s}\left({q}\right)= \mathsf{t}\left({q'}\right)\) we have
\[ \left({\left({a,h,b}\right) q }\right) q' 
= \left({a, hq, \mathsf{s}\left({q}\right) }\right) q'
= \left({a, hq q', \mathsf{s}\left({q'}\right) }\right) 
= \left({a, hq q', \mathsf{s}\left({q q'}\right) }\right) 
= \left({a, h , b}\right)\left({q q'}\right).
\]
\end{enumerate}
The properties of the left action are similarly proved.
Now, we have to check the compatibility conditions  of a biset.
For each \(\left({a, h, b}\right) \in X\), \(r \in \mathcal{A}_{ \scriptscriptstyle{1} }\) and \(q \in \mathcal{B}_{ \scriptscriptstyle{1} }\) such that \(\vartheta\left({a, h, b}\right) = \mathsf{s}\left({a}\right)\) and \(\varsigma\left({a, h, b}\right) = \mathsf{t}\left({q}\right)\) we have
\[ \begin{gathered}
\vartheta\left({ \left({a, h, b}\right) q}\right) = \vartheta\left({a , hq, \mathsf{s}\left({q}\right) }\right) 
= a = \vartheta\left({a, h, b}\right),  \\
\varsigma\left({r \left({a, h, b}\right) }\right) = \varsigma\left({ \mathsf{t}\left({r}\right), r h, b}\right) 
= b = \varsigma\left({a, h, b}\right) 
\end{gathered}
\]
and
\[ r\left({\left({a, h, b}\right) q}\right) = r \left({a, hq, \mathsf{s}\left({q }\right) }\right) 
= \left({\mathsf{t}\left({r}\right), r h q, \mathsf{s}\left({q}\right) }\right) 
= \left({ \mathsf{t}\left({r}\right), r h, b}\right) q
= \left({ r \left({a, h, b}\right) }\right) q,
\]
and this finishes the proof.  
\end{proof}

\subsection{The tensor product of groupoid-bisets}\label{ssec:tensor}

Next, from \cite[page~161]{Moedijk/Mrcun:2005}, \cite[Chap.~III, §4, 3.1]{DemGab:GATIGAGGC} and \cite[Définition~1.3.1, page~114]{GiraCohoNAb}, we recall the definition of the tensor product of two groupoid-bisets and show its universal property. Fix three groupoids $\cG$, $\cH$ and $\cK$. Given $(Y, \varkappa,\varrho)$ a $(\cG,\cK)$-biset  and $(X, \vartheta,\varsigma)$ an $(\cH, \cG)$-biset.  
Considering the triple $\big(  X \, \due \times {\scriptscriptstyle{\varsigma}} {\, \scriptscriptstyle{\sf{t}}} \Ga \, \due \times {\scriptscriptstyle{\sf{s}}} {\, \scriptscriptstyle{\varkappa}}  Y , \bara{\vartheta}, \bara{\varrho} \big)$, where 
$$
\bara{\vartheta}: X \, \due \times {\scriptscriptstyle{\varsigma}} {\, \scriptscriptstyle{\sf{t}}} \Ga \, \due \times {\scriptscriptstyle{\sf{s}}} {\, \scriptscriptstyle{\varkappa}}  Y  \longrightarrow  \Ho, \; \Big( (x,g,y) \longmapsto \vartheta(x) \Big) ;\quad  \bara{\varrho}: X \, \due \times {\scriptscriptstyle{\varsigma}} {\, \scriptscriptstyle{\sf{t}}} \Ga \, \due \times {\scriptscriptstyle{\sf{s}}} {\, \scriptscriptstyle{\varkappa}}  Y \longrightarrow \Ko, \;\Big( (x,  g, y)  \longmapsto \varrho(y) \Big), 
$$
we  have that $\big(  X \, \due \times {\scriptscriptstyle{\varsigma}} {\, \scriptscriptstyle{\sf{t}}} \Ga \, \due \times {\scriptscriptstyle{\sf{s}}} {\, \scriptscriptstyle{\varkappa}}  Y , \bara{\vartheta}, \bara{\varrho} \big)$ is an $(\cH, \cK)$-biset with actions:
\begin{equation*}
\begin{split}
\Big(   X \, \due \times {\scriptscriptstyle{\varsigma}} {\, \scriptscriptstyle{\sf{t}}} \Ga \, \due \times {\scriptscriptstyle{\sf{s}}} {\, \scriptscriptstyle{\varkappa}}  Y   \Big)\, \due \times {\Sscript{\bara{\varrho}}}{\,\Sscript{\Sf{t}} } \, \Ka \longrightarrow X \, \due \times {\scriptscriptstyle{\varsigma}} {\, \scriptscriptstyle{\sf{t}}} \Ga \, \due \times {\scriptscriptstyle{\sf{s}}} {\, \scriptscriptstyle{\varkappa}}  Y , \quad \Big(  \big( (x, g, y), k\big) \longmapsto (x, g, yk)  \Big)  \\ 
\Ha\, \due \times {\Sscript{\Sf{s}}}{\,\Sscript{\bara{\vartheta}} } \, \Big(  X \, \due \times {\scriptscriptstyle{\varsigma}} {\, \scriptscriptstyle{\sf{t}}} \Ga \, \due \times {\scriptscriptstyle{\sf{s}}} {\, \scriptscriptstyle{\varkappa}}  Y  \Big) \longrightarrow X \, \due \times {\scriptscriptstyle{\varsigma}} {\, \scriptscriptstyle{\sf{t}}} \Ga \, \due \times {\scriptscriptstyle{\sf{s}}} {\, \scriptscriptstyle{\varkappa}}  Y, \quad \Big(  \big(h,  (x,g,y)\big) \longmapsto (hx,g, y) \Big).
\end{split}
\end{equation*}

On the other hand, consider the map  $\omegaup: X \, \due \times {\Sscript{\varsigma}}{\, \Sscript{\varkappa}} \, Y \to \Go$ sending $(x,y) \mapsto \varkappa(y)=\varsigma(x)$. Then  the pair $\big(X \, \due \times {\Sscript{\varsigma}}{\, \Sscript{\varkappa}} \, Y , \omegaup \big)$ admits a structure of right $\cG$-set with action 
$$
\big(X \, \due \times {\Sscript{\varsigma}}{\, \Sscript{\varkappa}} \, Y\big)\, \due \times {\Sscript{\omegaup}}{\,\Sscript{\Sf{t}}} \, \Ga \longrightarrow \big(X \, \due \times {\Sscript{\varsigma}}{\, \Sscript{\varkappa}} \, Y\big),\qquad  \Big( \big((x,y), g\big) \longmapsto (xg,g^{-1}y) \Big).
$$
Following the notation and the terminology of \cite[Remark 2.12]{ElKaoutit:2017}, we denote by $X \tensor{\cG} Y := \big(X \, \due \times {\Sscript{\varsigma}}{\, \Sscript{\varkappa}} \, Y\big)  / \cG$ the orbit set of the right $\cG$-set $\big(X \, \due \times {\Sscript{\varsigma}}{\, \Sscript{\varkappa}} \, Y, \omegaup\big)$. We refer to $X \tensor{\cG} Y$ as \emph{the tensor product over $\cG$} of $X$ and $Y$.  It turns out that $X \tensor{\cG} Y$ admits a structure of $(\cH,\cK)$-biset whose structure maps are given as follows. First, denote by $x\tensor{\cG}y$ the equivalence class of an element $(x,y) \in  X \, \due \times {\Sscript{\varsigma}}{\, \Sscript{\varkappa}} \, Y$. That is, we have 
$$
xg\tensor{\cG}y=x\tensor{\cG}gy, \;\; \text{for every } \; g \in \Ga \; \text{ with }\; \varkappa(y)=\Sf{t}(g)=\varsigma(x).
$$
Second, one can easily check that the maps 
$$
\td{\varrho}: X \tensor{\cG} Y \to \Ko, \;\; \Big( x\tensor{\cG}y \longmapsto \varrho(y) \Big) ;\quad  \td{\vartheta}: X \tensor{\cG} Y \to \Ho, \;\; \Big( x\tensor{\cG} y \longmapsto \vartheta(x) \Big)
$$ 
are well defined, in such a way that the following maps
\begin{equation*}
\begin{split}
\big(X \tensor{\cG} Y\big)\, \due \times {\Sscript{\td{\varrho}}}{\,\Sscript{\Sf{t}} } \, \Ka \longrightarrow X \tensor{\cG} Y, \quad \Big(  \big( x\tensor{\cG}y, k\big) \longmapsto x\tensor{\cG}yk \Big)  \\ 
\Ha\, \due \times {\Sscript{\Sf{s}}}{\,\Sscript{\td{\varsigma}} } \, \big(X \tensor{\cG} Y\big) \longrightarrow X \tensor{\cG} Y, \quad \Big(  \big(h,  x\tensor{\cG}y \big)  \longmapsto hx\tensor{\cG}y \Big)
\end{split}
\end{equation*}
induce a structure of $(\cH,\cK)$-biset on $X\tensor{\cG}Y$.

The following describes the universal property of the tensor product between groupoid-bisets, see also \cite[Remark 2.2]{Kaoutit/Kowalzig:14}.

\begin{lemma}\label{lema:TP}
Let $\cG, \cH$ and   $\cK$ be three groupoids and $(X, \vartheta,  \varsigma)$, $(Y, \varkappa, \varrho)$ the above groupoid-bisets.  Then the following diagram 
\begin{equation}\label{Eq:CoEq}
\xymatrix{ X \, \due \times {\scriptscriptstyle{\varsigma}} {\, \scriptscriptstyle{\sf{t}}} \Ga \, \due \times {\scriptscriptstyle{\sf{s}}} {\, \scriptscriptstyle{\varkappa}}  Y  \ar@<0.70ex>@{->}^-{\rho\,  \times\,  1_{\Sscript{X}}}[rr] \ar@<-0.70ex>@{->}_-{1_{\Sscript{Y}} \, \times\,  \lambda }[rr]  & & X \, \due \times {\scriptscriptstyle{\varsigma}} {\, \scriptscriptstyle{\varkappa}} Y      \ar@{->>}^-{}[rr]  &&  X\tensor{\cG} Y    }
\end{equation}
is the co-equalizer, in the category of $(\cH, \cK)$-bisets, of the pair of morphisms $\big( \rho \times 1_{\Sscript{X}}, 1_{\Sscript{Y}} \times \lambda  \big)$. 
\end{lemma}
\begin{proof}
Straightforward. 
\end{proof}

\section{Mackey Formula for Groupoids}\label{sec:MF}

Now we will explain and prove the main result of this work.
Before doing this, however, we have to introduce some particular groupoid-bisets and prove their properties.
We will also have to define a specific kind of product of two subgroupoids.

\subsection{Orbits of products and cosets}\label{ssec:product}

Let $\cG$ and $\cH$ be two groupoids and $\cL $ a subgroupoid of the product $\cH \times \cG$. Consider the set of equivalence classes  \(\left({\frac{\mathcal{H} \times \mathcal{G}}{\mathcal{L}} }\right)^{\Sscript{\mathsf{L}}} \) as in equation \eqref{Eq:L}. An element  in this set is an equivalence clase of a  fourfold element $(h,g,u,v) \in \cL (\cH \times \cG)^{\Sscript{\tauup}} = \Big(  \Ha \times \Ga \Big)\, \due \times {\Sscript{\Sf{s}}} {\, \Sscript{\tauup_0}} \,  \Lo$ where $\tauup: \cL \hookrightarrow \cH \times \cG$ is the inclusion functor, that is, 
$$
[(h,g, u,v)] \cL=\LR{ \Big({ hh_{1}, gg_{1},  \mathsf{s}(h_{1}),  \mathsf{s}(g_{1}) }\Big) \in \cL(\cH \times \cG)^{\Sscript{\tauup}} | \; (h_{1}, g_{1}) \in \cL_{ \scriptscriptstyle{1} } \text{ with } \, \Sf{t}(h_{1})=\Sf{s}(h), \Sf{t}(g_{1})=\Sf{s}(g) }.
$$

\begin{lemma}
\label{lSubGrpdQuotL}
Given two groupoids \(\mathcal{G}\) and \(\mathcal{H}\), let \(\mathcal{L}\) be a subgroupoid of \(\mathcal{H} \times \mathcal{G}\).
Then the left $\cL$-coset
\[ X= \left({\frac{\mathcal{H} \times \mathcal{G}}{\mathcal{L}} }\right)^{\Sscript{\mathsf{L}}} 
\]
is an \(\left({\mathcal{H}, \mathcal{G}}\right)\)-biset with structure maps
\[ \begin{aligned}{\vartheta}  \colon & {X} \longrightarrow {\mathcal{H}_{ \scriptscriptstyle{0} } } \\ & {[\left({h,g,u,v}\right)]\mathcal{L} }  \longmapsto {\mathsf{t}\left({h}\right) }\end{aligned} 
\qquad \text{and} \qquad
\begin{aligned}{\varsigma}  \colon & {X} \longrightarrow {\mathcal{G}_{ \scriptscriptstyle{0} } } \\ & {[\left({h,g,u,v}\right)]\mathcal{L} }  \longmapsto {\mathsf{t}\left({g}\right) }\end{aligned} 
\]
and actions
\[ \begin{aligned}{\lambda}  \colon & {\mathcal{H}_{ \scriptscriptstyle{1} }  { \, {}_{ \scriptscriptstyle{\mathsf{s} } } {\times}}_{ \scriptscriptstyle{\vartheta} }\, X} \longrightarrow {X} \\ & {\left({h_{1}, [\left({h,g,u,v}\right)]\mathcal{L} }\right) }  \longmapsto {[\left({h_{1}h,g,u,v}\right)]\mathcal{L} }\end{aligned} 
\qquad \text{and} \qquad
\begin{aligned}{\rho}  \colon & {X  {{}_{ \scriptscriptstyle{\varsigma} } {\times}}_{ \scriptscriptstyle{\mathsf{t} } }\, \mathcal{G}_{ \scriptscriptstyle{1} } } \longrightarrow {X} \\ & { \Big({[\left({h,g,u,v}\right)] \mathcal{L}, g_{1}}\Big) } \longmapsto  \Big[{\left({h,g_{1}^{-1} g,u, v}\right)\Big] \mathcal{L} .}\end{aligned} 
\]
\end{lemma}
\begin{proof}
Let us first check that  $\vartheta$ and $\varsigma$ are well defined maps. So given two representatives of the same equivalence class $[\left({h,g,u,v}\right)]\mathcal{L} = [\left({h',g',u',v'}\right)]\mathcal{L}$,  by Lemma \ref{lEquLatLeft}, we know that  $(h',g')^{-1} (h,g) \in\La$. From which one obtains  $t(h')=t(h)$ and  $t(g')=t(g)$. As for \(\lambda\) and \(\rho\),  
take \(h_1 \in \Ha \) and \(g_1 \in\Ga\) such that $\Sf{s}(h_{1}) =\vartheta\big( [(h,g,u,v)]\cL \big) = \Sf{t}(h)$,  $\Sf{t}(g_{1}) =\varsigma\big([(h,g,u,v)]\cL\big) = \Sf{t}(g)$ and  $[\left({h,g,u,v}\right)]\mathcal{L} = [\left({h',g',u',v'}\right)]\mathcal{L}$. Then, as before, we have 
$$
\begin{aligned}
 \mathcal{L}_{ \scriptscriptstyle{1} }  & \ni \left({h', g'}\right)^{-1} \left({h, g}\right) 
= \left({h'{}^{-1} h, g'{}^{-1} g}\right) 
= \left({h'{}^{-1} h_{1}^{-1} h_{1} h, g'{}^{-1} g_{1} g_{1}^{-1} g}\right)  
= \left({h_{1} h', g_{1}^{-1} g'}\right)^{-1} \left({h_{1} h, g_{1}^{-1} g}\right), 
\end{aligned}
$$
which also shows that 
$$
\Big[\left({h_{1} h, g_{1}^{-1} g, u, v}\right)\Big] \mathcal{L} \, =\, \Big[\left({h_{1}h', g_{1}^{-1} g', u', v'}\right)\Big] \mathcal{L}.
$$
Therefore,  \(\lambda\) and \(\rho\) are well defined.
The verification that \(\left(X, \vartheta, \varsigma\right)\) is a biset is now easy and is left to the reader.
\end{proof}

\begin{lemma}
\label{lStabSubGrpd}
Given two groupoids \(\mathcal{H}\) and \(\mathcal{G}\), let \(\left({X, \vartheta, \varsigma}\right)\) be an \(\left({\mathcal{H},\mathcal{G}}\right)\)-biset and take \(x \in X\).
We define
\begin{equation}
\label{eStabMor}
\left({\mathcal{L}_{ \scriptscriptstyle{x} }}\right)_{ \scriptscriptstyle{1} } = \LR{ \left({h,g}\right) \in \mathcal{H} \times \mathcal{G} |\; \mathsf{s}\left({h}\right)= \vartheta\left({x}\right), \; \mathsf{t}\left({g}\right)= \varsigma\left({x}\right), \; hx = x g  } 
\end{equation} 
and
\begin{equation}
\label{eStabObj}
\left({\mathcal{L}_{ \scriptscriptstyle{x} }}\right)_{ \scriptscriptstyle{0} } = \LR{\left( \vartheta\left({x}\right), \varsigma\left({x}\right) \right) }.
\end{equation} 
Then \(\mathcal{L}_{ \scriptscriptstyle{x} }\) is a subgroupoid of the groupoid \(\mathcal{H}\times \mathcal{G}\).
\end{lemma}
\begin{proof}
It is immediate, since by using the first axiom of a biset, we know that, for every \(\left({h, g}\right) \in \left({\mathcal{L}_{ \scriptscriptstyle{x} }}\right)_{ \scriptscriptstyle{1} }\), we have
\[ 
\vartheta\left({x}\right)= \vartheta\left({xg}\right)
= \vartheta\left({hx}\right)= \mathsf{t}\left({h}\right)
\qquad \text{and} \qquad
\varsigma\left({x}\right) = \varsigma\left({hx}\right)
= \varsigma\left({xg}\right) 
= \mathsf{s}\left({g}\right) .
\]
Thus, $\cL_{\Sscript{x}}$ is a subgroup of the isotropy group $(\cH \times \cG)^{\Sscript{(\vartheta(x),\,  \varsigma(x))}}$ and then a subgroupoid with only one object $\{ (\vartheta(x),\,  \varsigma(x)) \}$.
\end{proof}

\begin{proposition}
\label{pIsomDueStab}
Given two groupoids \(\mathcal{H}\) and \(\mathcal{G}\), let \(X\) be an \(\left({\mathcal{H}, \mathcal{G}}\right)\)-biset and  take  \(x \in X\).
We define:
\[ \begin{aligned}
\left({\mathcal{L}_{ \scriptscriptstyle{x} }}\right)_{ \scriptscriptstyle{0} } & = \left({\mathcal{K}_{ \scriptscriptstyle{x} }}\right)_{ \scriptscriptstyle{0} } = \LR{\left({\vartheta\left({x}\right), \varsigma\left({x}\right)}\right)}, \\
\left({\mathcal{L}_{ \scriptscriptstyle{x} }}\right)_{ \scriptscriptstyle{1} } &= \LR{ \left({h,g}\right) \in \mathcal{H} \times \mathcal{G} | \; hx = xg, \quad \vartheta\left({x}\right)= \mathsf{s}\left({h}\right), \quad \mathsf{t}\left({g}\right)= \varsigma\left({x}\right)  } , \\
\left({\mathcal{K}_{ \scriptscriptstyle{x} }}\right)_{ \scriptscriptstyle{1} } &= \LR{ \left({h,g}\right) \in \mathcal{H} \times \mathcal{G} | \; hxg = x, \quad \vartheta\left({x}\right)= \mathsf{s}\left({h}\right), \quad \mathsf{t}\left({g}\right)= \varsigma\left({x}\right)  } , \\
\end{aligned}
\]
that is, \(\mathcal{K}=\Stabi_{ \scriptscriptstyle{\left({\mathcal{H}, \, \mathcal{G} }\right) }}\left({x}\right)\). 
Then
$$
\phiup_{ \scriptscriptstyle{x}} :   {\left({\frac{\mathcal{H} \times \mathcal{G}^{\Sscript{\text{op}}}}{\mathcal{K}_{ \scriptscriptstyle{x} } } }\right)^{\Sscript{\mathsf{L} }} } \longrightarrow { \left({ \frac{\mathcal{H} \times \mathcal{G}}{ \mathcal{L}_{ \scriptscriptstyle{x} } }  }\right)^{\Sscript{\mathsf{L}} } }, \qquad  \Big(  [(h,g,\vartheta(x), \varsigma(x))]\cK_{\scriptscriptstyle{x}}  \longmapsto    [(h,g^{-1},\vartheta(x), \varsigma(x))]\cL_{\scriptscriptstyle{x}}  \Big)
$$
is a well-defined isomorphism of \(\left({\mathcal{H}, \mathcal{G}}\right)\)-bisets with structure given as in Lemma \ref{lSubGrpdQuotL}.
\end{proposition}
\begin{proof}
For each \(h',h \in \mathcal{H}_{ \scriptscriptstyle{1} }\) and \(g', g \in \mathcal{G}_{ \scriptscriptstyle{1} }\) with 
$\Sf{s}(h)=\vartheta(x)=\Sf{s}(h')$ and $\Sf{s}(g)=\varsigma(x)=\Sf{s}(g')$,
 we have
\[ 
[\left({h',g', \vartheta\left({x}\right), \varsigma\left({x}\right) }\right)] \mathcal{K}_{ \scriptscriptstyle{x} } = [\left({h, g^{-1}, \vartheta\left({x}\right), \varsigma\left({x}\right) }\right) ]\mathcal{K}_{ \scriptscriptstyle{x} } 
\]
if and only if
\[ \left({h^{-1} h', g' g^{-1}}\right) = \left({h^{-1} h' , g^{-1} \overset{\Sscript{\text{op}}}{\cdot} g' }\right) 
= \left({h, g}\right)^{-1} \left({h', g'}\right) \in \left({\mathcal{K}_{ \scriptscriptstyle{x} } }\right)_{ \scriptscriptstyle{1} } ,
\]
if only if \(h^{-1} bxag^{-1} = x\), if and only if \(h^{-1} h' x = x g^{-1}\), if and only if
\[ \left({h, g^{-1}}\right)^{-1} \left({h', g'{}^{-1}}\right) = \left({h^{-1} h', gg'{}^{-1}}\right) \in \mathcal{L}_{ \scriptscriptstyle{x} } ,
\]
if and only if
\[ [\left({h',g'{}^{-1}, \vartheta\left({x}\right), \varsigma\left({x}\right) }\right)] \mathcal{K}_{ \scriptscriptstyle{x} } = [\left({h, g^{-1}, \vartheta\left({x}\right), \varsigma\left({x}\right) }\right)] \mathcal{K}_{ \scriptscriptstyle{x} } .
\]
Therefore \(\phiup_{ \scriptscriptstyle{x} }\) is well defined and injective.
For each \(h' \in \mathcal{H}_{ \scriptscriptstyle{1} }\) and \(g' \in \mathcal{G}_{ \scriptscriptstyle{1} }\) we have
\[ \phiup_{ \scriptscriptstyle{x} } \left({[\left({h', g'{}^{-1}, \vartheta\left({x}\right), \varsigma\left({x}\right) }\right)] \mathcal{K}_{ \scriptscriptstyle{x} } }\right) = [\left({h', g', \vartheta\left({x}\right), \varsigma\left({x}\right) }\right)] \mathcal{L}_{ \scriptscriptstyle{x} } 
\]
hence \(\phiup_{ \scriptscriptstyle{x} } \) is surjective.

Now given \(y =[\left({h',g', \vartheta\left({x}\right), \varsigma\left({x}\right) }\right)] \mathcal{K}_{ \scriptscriptstyle{x} } \in  \left({ \frac{\mathcal{H} \, \times \,  \mathcal{G}^{\Sscript{\text{op}}}}{\mathcal{K}_{ \scriptscriptstyle{x} } } }\right)^{\Sscript{\mathsf{L}} }\), \(h \in \mathcal{H}_{ \scriptscriptstyle{1} }\) and \(g \in \mathcal{G}_{ \scriptscriptstyle{1} }\) such that \(\mathsf{s}\left({h}\right)= \vartheta\left({y}\right)\) and \(\varsigma\left({y}\right)= \mathsf{t}\left({g}\right)\) we have
\[ \begin{aligned}
\phiup_{ \scriptscriptstyle{x} }\left({hyg}\right) &= \varphi_{ \scriptscriptstyle{x} }\left({[\left({hh', g \overset{\Sscript{\text{op}}}{\cdot} g', \vartheta\left({x}\right), \varsigma\left({x}\right)] }\right) \mathcal{K}_{ \scriptscriptstyle{x} } }\right) 
= \phiup_{ \scriptscriptstyle{x} }\left({[\left({hh', g'g , \vartheta\left({x}\right), \varsigma\left({x}\right) ]}\right) \mathcal{K}_{ \scriptscriptstyle{x} } }\right) \\
&= [\left({hh', g^{-1} g'{}^{-1} , \vartheta\left({x}\right), \varsigma\left({x}\right) }\right)]\mathcal{L}_{ \scriptscriptstyle{x} } 
= h\left({[\left({h', g'{}^{-1}, \vartheta\left({x}\right), \varsigma\left({x}\right) }\right)]\mathcal{L}_{ \scriptscriptstyle{x} } }\right)g
= h \phiup_{ \scriptscriptstyle{x} }\left({y}\right) g
\end{aligned}
\]
thus \(\phiup_{ \scriptscriptstyle{x} }\) is an isomorphism of \(\left({\mathcal{H},\mathcal{G}}\right)\)-bisets as stated. 
\end{proof}

\begin{definition}
\label{dGrpdStarGrpd}
Given groupoids \(\mathcal{G}\), \(\mathcal{H}\) and \(\mathcal{K}\), let \(\mathcal{L}\) be a subgroupoid of \(\mathcal{H} \times \mathcal{G}\) and \(\mathcal{M}\) be a subgroupoid of \(\mathcal{K} \times \mathcal{H}\).
We define
\[ \left({\mathcal{M} \ast \mathcal{L} }\right)_{ \scriptscriptstyle{1} } = \LR{ \left({k,g}\right) \in \mathcal{K}_{ \scriptscriptstyle{1} } \times \mathcal{G}_{ \scriptscriptstyle{1} }  | \; \exists \, h \in \mathcal{H}_{ \scriptscriptstyle{1} } \; \text{ such that }\;  \left({k,h}\right) \in \mathcal{M}_{ \scriptscriptstyle{1} } \; \text{ and }\;  \left({h,g}\right) \in \mathcal{L}_{ \scriptscriptstyle{1} }  } 
\]
and
\[ \left({\mathcal{M} \ast \mathcal{L} }\right)_{ \scriptscriptstyle{0} } = \LR{ \left({v,a}\right) \in \mathcal{K}_{ \scriptscriptstyle{0} } \times \mathcal{G}_{ \scriptscriptstyle{0} }  | \;  \exists  \, u \in \mathcal{H}_{ \scriptscriptstyle{0} } \; \text{ such that }\; \left({v,u}\right) \in \mathcal{M}_{ \scriptscriptstyle{0} }  \; \text{ and }\;  \left({u,a}\right) \in \mathcal{L}_{ \scriptscriptstyle{0} }  } .
\]
Notice that, if $pr_{\Sscript{2}}(\cM) \cap pr_{\Sscript{1}}(\cL) \,=\,  \emptyset $, where $pr_{\Sscript{1}}$ and $pr_{\Sscript{2}}$ are the first and second projections, then $(\cM * \cL)_{\Sscript{0}}$ is obviously an  empty set. 
\end{definition}

\begin{lemma}
\label{lGrpdStarGrpd}
Given groupoids \(\mathcal{G}\), \(\mathcal{H}\) and \(\mathcal{K}\) such that \(\mathcal{H}\) has only one object, let \(\mathcal{L}\) be a subgroupoid of \(\mathcal{H} \times \mathcal{G}\) and \(\mathcal{M}\) be a subgroupoid of \(\mathcal{K} \times \mathcal{H}\).
Then \(\mathcal{M} \ast \mathcal{L}\), as defined in Definition~\(\ref{dGrpdStarGrpd}\), is a subgroupoid of \(\mathcal{K} \times \mathcal{G}\).
\end{lemma}
\begin{proof}
Given  \(\left({k, g }\right) \in \left({\mathcal{M} \ast \mathcal{L}}\right)_{ \scriptscriptstyle{1} }\), then there is \(h \in \mathcal{H}_{ \scriptscriptstyle{1} }\) such that \(\left({k, h}\right) \in \mathcal{M}_{ \scriptscriptstyle{1} }\) and \(\left({h,g}\right) \in \mathcal{L}_{ \scriptscriptstyle{1} }\) so \(\left({k^{-1}, h^{-1}}\right) \in \mathcal{M}_{ \scriptscriptstyle{1} }\) and \(\left({h^{-1}, g^{-1}}\right) \in \mathcal{L}_{ \scriptscriptstyle{1} }\) thus \(\left({k, g}\right)^{-1} \in \left({\mathcal{M} \ast \mathcal{L}}\right)_{ \scriptscriptstyle{1} }\).
Now let be \(\left({k_1, g_1}\right), \left({k_2, g_2}\right) \in \left({\mathcal{M} \ast \mathcal{L}}\right)_{ \scriptscriptstyle{1} }\) such that \(\mathsf{s}\left({k_1, g_1}\right) = \mathsf{t}\left({k_2, g_2}\right)\).
There are \(h_1, h_2 \in \mathcal{H}_{ \scriptscriptstyle{1} }\) such that \(\left({k_i, h_i}\right) \in \mathcal{M}_{ \scriptscriptstyle{1} }\) and \(\left({h_i, g_i}\right) \in \mathcal{L}_{ \scriptscriptstyle{1} }\) for each \({i}\in \Set{{1},{2}}\).
Since \(\mathcal{H}\) has only one object we have \(\mathsf{s}\left({h_1}\right)= \mathsf{t}\left({h_2}\right)\) thus we can write \(h_1 h_2\) and we have
\[ \left({k_1, h_1}\right) \left({k_2, h_2}\right) = \left({k_1 k_2, h_1 h_2}\right) \in \mathcal{M}_{ \scriptscriptstyle{1} },
 \quad
\left({h_1, g_1}\right) \left({h_2, g_2}\right) = \left({h_1 h_2, g_1 g_2}\right) \in \mathcal{L}_{ \scriptscriptstyle{1} } .
\]
Therefore
\( \left({k_1, g_1}\right) \left({k_2, g_2}\right) = \left({k_1 k_2, g_1 g_2}\right) \in \left({\mathcal{M} \ast \mathcal{L} }\right)_{ \scriptscriptstyle{1} }
\), and this completes the proof.
\end{proof}

\begin{example}
Given groupoids \(\mathcal{K}\), \(\mathcal{H}\) and \(\mathcal{G}\) such that \(\mathcal{H}\) has only one object \(\omega\), we consider  subgroupoids \(\mathcal{D} \le \mathcal{K}\), \(\mathcal{C} \le \mathcal{H}\), \(\mathcal{B} \le \mathcal{H}\) and \(\mathcal{A}\le \mathcal{G}\) where \(\mathcal{C}\) and \(\mathcal{B}\) are not empty, that is, have exactly the object \(\omega\).
Then \(\mathcal{M}= \mathcal{D} \times \mathcal{C}\) is a subgroupoid of \(\mathcal{K} \times \mathcal{H}\) and \(\mathcal{L}= \mathcal{B} \times \mathcal{A}\) is a subgroupoid of \(\mathcal{H} \times \mathcal{G}\).
For each \(d_0 \in \mathcal{D}_{ \scriptscriptstyle{0} }\) and \(a_0 \in \mathcal{A}_{ \scriptscriptstyle{0} }\) we have \(\left({d_0, \omega}\right) \in \mathcal{M}_{ \scriptscriptstyle{0} }\) and \(\left({\omega, a_0}\right) \in \mathcal{L}_{ \scriptscriptstyle{0} }\) thus \(\left({d_0, a_0}\right) \in \left({\mathcal{M} \ast \mathcal{L}}\right)_{ \scriptscriptstyle{0} }\).
We have \(\iota_{ \scriptscriptstyle{\omega} } \in \mathcal{C}_{ \scriptscriptstyle{1} } \cap \mathcal{B}_{ \scriptscriptstyle{1} }\) so for each \(d_1 \in \mathcal{D}_{ \scriptscriptstyle{1} }\) and \(a_1 \in \mathcal{A}_{ \scriptscriptstyle{1} }\) we have \(\left({d_1, \iota_{ \scriptscriptstyle{\omega} } }\right) \in \mathcal{M}_{ \scriptscriptstyle{1} }\) and \(\left({\iota_{ \scriptscriptstyle{\omega} }, a_1}\right) \in \mathcal{L}_{ \scriptscriptstyle{1} }\) therefore \(\left({d_1, a_1}\right) \in \left({\mathcal{M} \ast \mathcal{L} }\right)_{ \scriptscriptstyle{1} }\).
As a consequence we have \(\mathcal{D}_{ \scriptscriptstyle{i} } \times \mathcal{A}_{ \scriptscriptstyle{i} } \subseteq \left({\mathcal{M} \ast \mathcal{L} }\right)_{ \scriptscriptstyle{i} }\) for \(i=0\) and \(i=1\).
For each \({i}\in \Set{{0},{1}}\) and for each \(\left({k_i, g_i}\right) \in \left({\mathcal{M} \ast \mathcal{L} }\right)_{ \scriptscriptstyle{i} }\) there is \(h_i \in \mathcal{H}_{ \scriptscriptstyle{i} }\) such that \(\left({k_i, h_i}\right) \in \mathcal{M}_{ \scriptscriptstyle{i} } = \mathcal{D}_{ \scriptscriptstyle{i} } \times \mathcal{C}_{ \scriptscriptstyle{i} }\) and \(\left({h_i, g_i}\right) \in \mathcal{L}_{ \scriptscriptstyle{i} } = \mathcal{B}_{ \scriptscriptstyle{i} } \times \mathcal{A}_{ \scriptscriptstyle{i} }\), thus \(k_i \in \mathcal{D}_{ \scriptscriptstyle{i} }\) and \(g_i \in \mathcal{A}_{ \scriptscriptstyle{i} }\), therefore \(\left({k_i, g_i}\right) \in \mathcal{D}_{ \scriptscriptstyle{i} } \times \mathcal{A}_{ \scriptscriptstyle{i} }\) and \(\left({\mathcal{M} \ast \mathcal{L}}\right)_{ \scriptscriptstyle{i} } \subseteq \mathcal{D}_{ \scriptscriptstyle{i} } \times \mathcal{A}_{ \scriptscriptstyle{i} }\).
This shows that \(\mathcal{M} \ast \mathcal{L} = \mathcal{D} \times \mathcal{A}\)  is not an empty groupoid if both \(\mathcal{D}\) and \(\mathcal{A}\) are not so.
\end{example}

\begin{proposition}
\label{pModifDoubleCosets}
Given groupoids \(\mathcal{K}\), \(\mathcal{H}\) and \(\mathcal{G}\), let \(\mathcal{M}\) be a subgroupoid of \(\mathcal{K} \times \mathcal{H}\) and \(\mathcal{L}\) be a subgroupoid of \(\mathcal{H} \times \mathcal{G}\).
Let be
\begin{equation}\label{Eq:Xx} 
X= \Set{ \left({w, u, h, v, a}\right) \in \mathcal{K}_{ \scriptscriptstyle{0} } \times \mathcal{H}_{ \scriptscriptstyle{0} } \times \mathcal{H}_{ \scriptscriptstyle{1} } \times \mathcal{H}_{ \scriptscriptstyle{0} } \times \mathcal{G}_{ \scriptscriptstyle{0} }  |  \begin{gathered}
\left({w, u}\right) \in \mathcal{M}_{ \scriptscriptstyle{0} }, \left({v, a}\right) \in \mathcal{L}_{ \scriptscriptstyle{0} }, \\
 u = \mathsf{t}\left({h}\right), v= \mathsf{s}\left({h}\right) 
\end{gathered} }.
\end{equation}
Then \(X\) is a \(\left({\mathcal{M}, \mathcal{L}}\right)\)-biset with structure maps
\[ \begin{aligned}{\vartheta}  \colon & {X} \longrightarrow {\mathcal{M}_{ \scriptscriptstyle{0} } } \\ & {\left({w, u, h, v, a}\right) }  \longmapsto {\left({w, u}\right) }\end{aligned} 
\qquad \text{and} \qquad
\begin{aligned}{\varsigma}  \colon & {X} \longrightarrow {\mathcal{L}_{ \scriptscriptstyle{0} } } \\ & {\left({w, u, h, v, a}\right) }  \longmapsto {\left({v, a}\right), }\end{aligned} 
\]
left action
\[ \begin{aligned}{\lambda}  \colon & {\mathcal{M}_{ \scriptscriptstyle{1} }  {\, {}_{ \scriptscriptstyle{\mathsf{s} } } {\times}}_{ \scriptscriptstyle{\vartheta} }\, X} \longrightarrow {X} \\ & {\Big( ({k, h'), \left({w, u, h, v, a}\right) }  \Big) }  \longmapsto {\left({\mathsf{t}\left({k}\right), \mathsf{t}\left({h'}\right), h' h, v, a}\right) }\end{aligned} 
\]
and right action
\[ \begin{aligned}{\rho}  \colon & {X  {{}_{ \scriptscriptstyle{\varsigma} } {\times}}_{ \scriptscriptstyle{\mathsf{t} } }\, \mathcal{L}_{ \scriptscriptstyle{1} } } \longrightarrow {X} \\ &  \Big( { {\left({w, u, h, v, a}\right), (h'', g)} } \Big)  \longmapsto  {\left({w, u, h h'', \mathsf{s}\left({h''}\right), \mathsf{s}\left({g}\right) }\right).}\end{aligned} 
\]
\end{proposition}
\begin{proof}
We only check the properties of a right action, since a similar proof shows the left action properties.
\begin{enumerate}
\item For each \(y = \left({w, u, h, v, a}\right) \in X\) and \(\left({h'', g}\right) \in \mathcal{L}_{ \scriptscriptstyle{1} } \) such that \(\varsigma\left({y}\right)= \mathsf{t}\left({h'', g}\right)\) we have
\[ \varsigma\left({y \left({h'', g}\right) }\right) = \varsigma\left({w, u, h h'', \mathsf{s}\left({h''}\right), \mathsf{s}\left({g}\right) }\right) 
= \left({\mathsf{s}\left({h''}\right), \mathsf{s}\left({g}\right) }\right) 
= \mathsf{s}\left({h'', g}\right) 
\]
\item For each \(y = \left({w, u, h, v, a}\right) \in X\) we have
\[ y \iota_{ \scriptscriptstyle{\varsigma\left({y}\right) } } = y \iota_{ \scriptscriptstyle{\left({v, a}\right) } } 
=  y \left({\iota_{ \scriptscriptstyle{v} }, \iota_{ \scriptscriptstyle{a} } }\right) 
= \left({w, u, h\,  \iota_{ \scriptscriptstyle{v} }, \mathsf{s}({\iota_{ \scriptscriptstyle{v} } }), \mathsf{s}({\iota_{ \scriptscriptstyle{a} } }) }\right) 
= y.
\]
\item For each \(y = \left({w, u, h, v, a}\right) \in X\) and \(\left({h_1, g}\right), \left({h_2, g'}\right) \in \mathcal{L}_{ \scriptscriptstyle{1} }\) such that \(\varsigma\left({y}\right)= \mathsf{t}\left({h_1, g}\right)\) and \(\mathsf{s}\left({h_1, g}\right) = \mathsf{t}\left({h_2, g'}\right)\) we have
\[ \begin{gathered}
\left({y \left({h_1, g}\right) }\right) \left({h_2, g'}\right) = \left({w, u, h h_1, \mathsf{s}\left({h_1}\right), \mathsf{s}\left({g}\right) }\right) \left({h_2, g'}\right)  
= \left({w, u, h h_1 h_2, \mathsf{s}\left({h_2}\right), \mathsf{s}\left({g'}\right) }\right) \\  
= \left({w, u, h h_1 h_2, \mathsf{s}\left({h_1 h_2}\right), \mathsf{s}\left({g g'}\right) }\right) 
= y \left({h_1 h_2, g, g'}\right)
= y \left({\left({h_1, g}\right). \left({h_2, g'}\right) }\right). 
\end{gathered}
\]
\end{enumerate}

Now we have to check the properties of a biset on $X$, that is, condition (2) in Definition \ref{def:biset} for the stated actions  $\lambda$ and $\rho$.
So, for each \(y = \left({w, u, h, v, a}\right) \in X\), \(\left({k, h'}\right) \in \mathcal{M}_{ \scriptscriptstyle{1} }\) and \(\left({h'', g}\right) \in \mathcal{L}_{ \scriptscriptstyle{1} }\) such that \(\mathsf{s}\left({k, h'}\right)= \vartheta\left({y}\right)\) and \(\varsigma\left({y}\right)= \mathsf{t}\left({h'', g}\right)\) we have
\[ \begin{gathered}
\vartheta\left({ y \left({h'', g}\right) }\right) = \vartheta\left({w, u, h h'', \mathsf{s}\left({h''}\right), \mathsf{s}\left({g}\right) }\right)
= \left({w, u}\right) 
= \vartheta\left({y}\right) , \\
\varsigma\left({\left({k, h'}\right) y}\right) = \varsigma\left({\mathsf{t}\left({k}\right), \mathsf{t}\left({h'}\right), h' h, v, a}\right)
= \left({v, a}\right) 
= \varsigma\left({y}\right)
\end{gathered}
\]
and
\[ \begin{gathered}
\left({k, h'}\right) \left({y \left({h'', g}\right) }\right) = \left({k, h'}\right) \left({w, u, h h'', \mathsf{s}\left({h''}\right), \mathsf{s}\left({g}\right) }\right)  \\
= \left({\mathsf{t}\left({k}\right), \mathsf{t}\left({h'}\right), h' h h'', \mathsf{s}\left({h''}\right), \mathsf{s}\left({g}\right) }\right) 
= \left({ \mathsf{t}\left({k}\right), \mathsf{t}\left({h'}\right), h' h, v, a}\right) \left({h'', g}\right)  \\
= \left({\left({k, h'}\right) y}\right) \left({h'', g}\right), 
\end{gathered}
\] 
and this gives the desired properties, which finishes the proof.
\end{proof}

\subsection{The main result}\label{ssec:Mackey}

Let's keep the notations of Proposition \ref{pModifDoubleCosets} and let's assume we are given \(w \in \mathcal{K}_{ \scriptscriptstyle{0} }\), \(u \in \mathcal{H}_{ \scriptscriptstyle{0} }\) and \(a \in \mathcal{G}_{ \scriptscriptstyle{0} }\) such that \(\left({w, u}\right) \in \mathcal{M}_{ \scriptscriptstyle{0} }\) and \(\left({u, a}\right) \in \mathcal{L}_{ \scriptscriptstyle{0} }\).
Under these assumptions,  we can apply Lemma \ref{lGrpdStarGrpd} to the groupoids \(\mathcal{K}^{\Sscript{\left({w, \,  w}\right)}}\), \(\mathcal{H}^{\Sscript{\left({u,\,  u}\right)}}\) and \(\mathcal{G}^{\Sscript{\left({a, \, a}\right)}}\) by taking the  subgroupoids \(\mathcal{M}^{\Sscript{\left({w,\,  u}\right)}}\) of \(\mathcal{K}^{\Sscript{\left({w, \, w}\right)}} \times \mathcal{H}^{\Sscript{\left({u,\,  u}\right)}}\) and \(\mathcal{L}^{\Sscript{\left({u, \, a}\right)}}\) of \(\mathcal{H}^{\Sscript{\left({u, \, u}\right)}} \times \mathcal{G}^{\Sscript{\left({a, \, a}\right)}}\). Of course, here  we are identifying the isotropy groups \(\mathcal{M}^{\Sscript{( w, \, u)}}\)  and \(\mathcal{L}^{\Sscript{( u, \, a)}}\) with groupoids having only one object \(\left({w, u}\right)\) and \(\left({u, a}\right)\), respectively. In this way, we obtain that
\begin{equation}\label{Eq:ML}
\mathcal{M}^{\Sscript{\left( w, \, u\right)}} \ast \left({ {}^{\left({h,\,  \iota_{ \scriptscriptstyle{a} } }\right) } \mathcal{L}^{\Sscript{(u,\, a)}}  }\right) 
\end{equation}
is a subgroupoid of \(\mathcal{K}^{\Sscript{(w,\,  w)}} \times \mathcal{G}^{\Sscript{(a,\,  a)}}\),  for every  $(h, \iota_{\Sscript{a}}) \in \La$ with $\Sf{s}(h)=\Sf{t}(h)=u$, and   where we have used the notation \({ {}^{g} } G = gGg^{-1}\) to denote the conjugation class of a given element $g$ in a group $G$. 
Since we know that  \(\mathcal{K}^{\Sscript{(w,\,  w)}} \times \mathcal{G}^{\Sscript{(a,\,  a)}}\)  is a subgroupoid of \(\mathcal{K} \times \mathcal{G}\), we have  that  $
\mathcal{M}^{\Sscript{\left( w, \, u\right)}} \ast \left({ {}^{\left({h,\,  \iota_{ \scriptscriptstyle{a} } }\right) } \mathcal{L}^{\Sscript{(u,\, a)}}  }\right) 
$
is a subgroupoid of \(\mathcal{K} \times \mathcal{G}\). This will be used implicitly in the sequel. 

Our main result is the following. 

\begin{theorem}[Mackey formula for bisets]\label{thm:Main}
Let $\cK$, $\cH$, $\cG$, $\cM$ and $\cL$ be as in  Proposition \ref{pModifDoubleCosets}. Consider the biset $X$ defined in equation \eqref{Eq:Xx} and the subgroupoids $
\mathcal{M}^{\Sscript{\left( w, \, u\right)}} \ast \left({ {}^{\left({h,\,  \iota_{ \scriptscriptstyle{a} } }\right) } \mathcal{L}^{\Sscript{(u,\, a)}}  }\right) 
$ of equation \eqref{Eq:ML}. 
Assume that \(\mathcal{M}_{ \scriptscriptstyle{0} }= \mathcal{K}_{ \scriptscriptstyle{0} } \times \mathcal{H}_{ \scriptscriptstyle{0} }\) and \(\mathcal{L}_{ \scriptscriptstyle{0} } = \mathcal{H}_{ \scriptscriptstyle{0} } \times \mathcal{G}_{ \scriptscriptstyle{0} }\), then there is  a (non canonical) isomorphism of bisets
\begin{equation}
\label{eMackeyFormulaBisets}
  \begin{gathered}
\left({\frac{\mathcal{K} \times \mathcal{H} }{\mathcal{M} } }\right)^{\Sscript{\mathsf{L} }} \otimes_{ \scriptscriptstyle{\mathcal{H} } } \left({\frac{\mathcal{H} \times \mathcal{G} }{\mathcal{L} } }\right)^{\Sscript{\mathsf{L}} } 
\cong   \biguplus_{\left({w,\, u,\, h, \,v,\, a}\right) \,  \in  \,  \rep_{\scriptscriptstyle{\left(\mathcal{M}, \,  \mathcal{L}\right)}}(X)    } \left({\frac{\mathcal{K} \times \mathcal{G} }{\mathcal{M}^{\Sscript{( w, \, u)}}  \ast \left({ ^{\left({h, \, \iota_{ \scriptscriptstyle{a} } }\right) } \mathcal{L}^{\Sscript{( u, \, a)}} }\right)  } }\right)^{\Sscript{\mathsf{L}} },
\end{gathered}
\end{equation} 
where 
 \(\rep_{ \scriptscriptstyle{\left(\mathcal{M},\,  \mathcal{L}\right)}}(X)\) is a set of representatives of the orbits of $X$ as \(\left({\mathcal{M}, \mathcal{L}}\right)\)-biset.
\end{theorem}
\begin{proof}
Notice that under assumptions the denominator in the right hand side of Formula \eqref{eMackeyFormulaBisets} is a well defined subgroupoid of $\cK \times \cG$ and thus the right hand side of this formula is well defined as well. 
For simplicity let us denote
\[ \mathcal{V}:= \left({\frac{\mathcal{K} \times \mathcal{H} }{\mathcal{M} } }\right)^{\Sscript{\mathsf{L}} }
\qquad \text{and} \qquad
\mathcal{U} := \left({\frac{\mathcal{H} \times \mathcal{G} }{\mathcal{L} } }\right)^{\Sscript{\mathsf{L}} }.
\]
As expounded in Lemma \ref{lSubGrpdQuotL}, \(\mathcal{V}\) is a \(\left({\mathcal{K}, \mathcal{H}}\right)\)-biset with structure maps
\[ \begin{aligned}{\Theta}  \colon & {\mathcal{V}} \longrightarrow {\mathcal{K}_{ \scriptscriptstyle{0} } } \\ & {[\left({k, h, w, u}\right)] \mathcal{M} }  \longmapsto {\mathsf{t}\left({k}\right) }\end{aligned} 
\qquad \text{and} \qquad
\begin{aligned}{\Xi}  \colon & {\mathcal{V} } \longrightarrow {\mathcal{H}_{ \scriptscriptstyle{0} } } \\ & {[\left({k, h, w, u}\right)] \mathcal{M}}  \longmapsto {\mathsf{t}\left({h}\right) }\end{aligned} 
\]
and \(\mathcal{U}\) is an \(\left({\mathcal{H}, \mathcal{G}}\right)\)-biset with structure maps
\[ \begin{aligned}{\Upsilon}  \colon & {\mathcal{U}} \longrightarrow {\mathcal{H}_{ \scriptscriptstyle{0} } } \\ & {[\left({h, g, v, a}\right)] \mathcal{L} }  \longmapsto {\mathsf{t}\left({h}\right) }\end{aligned} 
\qquad \text{and} \qquad
\begin{aligned}{\Lambda}  \colon & {\mathcal{U} } \longrightarrow {\mathcal{G}_{ \scriptscriptstyle{0} } } \\ & {[\left({h, g, v, a}\right)] \mathcal{L}}  \longmapsto {\mathsf{t}\left({g}\right). }\end{aligned} 
\]
Therefore, following subsection \ref{ssec:tensor}, the tensor product  $\cV \tensor{\cH}\cU$ in the left hand side of  \eqref{eMackeyFormulaBisets} makes sense and it is a $(\cK, \cG)$-biset by Lemma \ref{lema:TP}.  The orbit of a given element 
$[(k, h, w, u)] \mathcal{M} \tensor{\cH}  [(h', g, v, a)] \mathcal{L}$  in $\cV \tensor{\cH}\cU$ will be  denoted by 
$ \mathcal{K}\Big[\Big([(k,h,w,u)] \mathcal{M} \tensor{\cH} [(h', g, v, a)] \mathcal{L}  \Big)\Big]\mathcal{G} $. If  \(y \in \mathcal{K} \backslash \left({ \mathcal{V} \otimes_{ \scriptscriptstyle{\mathcal{H} } }  \mathcal{U} }\right) / \mathcal{G}\) is an element in the set of orbits of $\cV \tensor{\cH}\cU$, then we will use the following notations, similar to the ones already used in Section \ref{ssec:StabOrbBiset} and in Proposition \ref{pIsomDueStab}:
\[ 
\begin{aligned}
\left({\Stabi_{ \scriptscriptstyle{\left({\mathcal{K},\,  \mathcal{G}}\right) }}\left({y}\right) }\right)_{ \scriptscriptstyle{0} } & = \left({ \left({\mathcal{K}, \mathcal{G} }\right)_{ \scriptscriptstyle{y} }  }\right)_{ \scriptscriptstyle{0} } = \Set{\left({\Theta\left({y}\right), \Lambda\left({y}\right)}\right)}, \\
\left({\left({\mathcal{K}, \mathcal{G} }\right)_{ \scriptscriptstyle{y} }  }\right)_{ \scriptscriptstyle{1} } &= \LR{ \left({k,g}\right) \in \mathcal{K}_{ \scriptscriptstyle{1} } \times \mathcal{G}_{ \scriptscriptstyle{1} } | \; ky = yg, \quad \Theta\left({y}\right)= \mathsf{s}\left({k}\right), \quad \mathsf{t}\left({g}\right)= \Lambda\left({y}\right)  } , \\
\left({\Stabi_{ \scriptscriptstyle{\left({\mathcal{K},\,  \mathcal{G}}\right) }}\left({y}\right)  }\right)_{ \scriptscriptstyle{1} } &= \LR{ \left({k,g}\right) \in \mathcal{K}_{ \scriptscriptstyle{1} } \times \mathcal{G}_{ \scriptscriptstyle{1} } | \; kyg = y, \quad \Theta\left({y}\right)= \mathsf{s}\left({k}\right), \quad \mathsf{t}\left({g}\right)= \Lambda\left({y}\right)  } . \\
\end{aligned}
\]
Since, by Lemma \ref{lema:LAquila}  and Proposition \ref{pBisetsLeftSets}, every biset is the disjoint union of its orbits. By  Proposition~\(\ref{pIsomDueStab}\) we obtain the following isomorphisms of \(\left({\mathcal{K}, \mathcal{G}}\right)\)-bisets:
\[  
\mathcal{V} \otimes_{ \scriptscriptstyle{\mathcal{H} } }  \mathcal{U} 
\cong \biguplus_{  y\, \in\,   \rep_{ \scriptscriptstyle{\left(\mathcal{K},\,  \mathcal{G}\right)}}\left(\mathcal{V} \otimes_{ \scriptscriptstyle{\mathcal{H} } }  \mathcal{U} \right)  } \left({\frac{\mathcal{K} \times \mathcal{G} }{\Stabi_{ \scriptscriptstyle{\left({\mathcal{K}, \, \mathcal{G} }\right) }}\left({  y  }\right)  } }\right)^{\Sscript{\mathsf{L}} }  
\cong \biguplus_{y\, \in\,  \rep_{ \scriptscriptstyle{\left(\mathcal{K},\,  \mathcal{G}\right)}}\left(\mathcal{V} \otimes_{ \scriptscriptstyle{\mathcal{H} } }  \mathcal{U} \right) } \left({\frac{\mathcal{K} \times \mathcal{G} }{ \left({\mathcal{K}, \mathcal{G} }\right)_{ \scriptscriptstyle{y} }  } }\right)^{\Sscript{\mathsf{L}} }  .
\]
Consider the map
\[ \begin{aligned}{\varphi}  \colon & {\mathcal{K} \setminus \left({\mathcal{V} \otimes_{ \scriptscriptstyle{\mathcal{H} } } \mathcal{U} }\right) / \mathcal{G} } \longrightarrow { \mathcal{M} \setminus X / \mathcal{L} } \\ & {\mathcal{K}\Big[\Big({[\left({k, h, w, u}\right)] \mathcal{M} \, \otimes_{ \scriptscriptstyle{\mathcal{H}} }\,  [\left({h', g, v, a}\right)] \mathcal{L}  }\Big)\Big]\mathcal{G} }  \longmapsto {\mathcal{M}\Big[({w, u, h^{-1} h', v, a}) \Big]\mathcal{L} .}
\end{aligned} 
\]
We have to check that \(\varphi\) is well defined. So let us choose two representatives for the same orbit, that is, let us assume that we have an equality of the form:
\[ \mathcal{K}\Big[{[\left({k, h, w, u}\right) ]\mathcal{M} \otimes_{ \scriptscriptstyle{\mathcal{H}} } [\left({h', g, v, a}\right)] \mathcal{L}  }\Big]\mathcal{G} = \mathcal{K}\Big[{[\left({l, e, r, n}\right) ]\mathcal{M} \otimes_{ \scriptscriptstyle{\mathcal{H}} } [\left({e', f, m, b}\right)] \mathcal{L}  }\Big]\mathcal{G} \; \in \; \cK\backslash (\cV \tensor{\cH} \cU)  / \cG
\]
the orbits set of $\cV \tensor{\cH} \cU$, where 
$[(k, h, w, u) ]\mathcal{M},  [(l, e, r, n ]\mathcal{M} \in \cV$ and $[(h', g, v, a)] \mathcal{L},  [(e', f, m, b)] \mathcal{L} \in \cU$. By definition this equality means that 
there are \(k_1 \in \mathcal{K}_{ \scriptscriptstyle{1} }\) and \(g_1 \in \mathcal{G}_{ \scriptscriptstyle{1} }\) such that \(\mathsf{s}\left({k_1}\right)= \mathsf{t}\left({l}\right)\),  \(\mathsf{t}\left({l}\right)= \mathsf{t}\left({g_1}\right)\) and 
$$
\begin{gathered}
\big[(k,h, w, u)\big]
\cM\tensor{\cH}[(h', g, v, a)]\cL = k_1\,\Big( [(l, e, r, n)]\cM \tensor{\cH}[(e', f, m, b)]\cL\Big) \, g_1 \\ = [(k_1l,e,r,n)]\cM\tensor{\cH}[(e',g_1^{-1}f,m,b)]\cL
\end{gathered}
$$
Thus there is \(h_1 \in \mathcal{H}_{ \scriptscriptstyle{1} }\) such that \(\mathsf{t}\left({e}\right)= \mathsf{t}\left({h_1}\right)\), \(\mathsf{t}\left({h_1}\right)= \mathsf{t}\left({e'}\right)\) and
\[ 
\Big({ [\left({k, h, w, u}\right)] \mathcal{M}, [\left({h', g, v, a}\right)]\mathcal{L} }\Big) = \Big({ [\left({k_1 l, h_1^{-1} e, r, n}\right)]\mathcal{M}, [\left({h_1^{-1} e', g_1^{-1} f, m, b}\right)] \mathcal{L} }\Big)  \; \in \; \cV \times \cU.
\]
This means that 
\[ \left\lbrace  
\begin{aligned}
& [\left({k, h, w, u}\right) ]\mathcal{M} = [\left({k_1 l, h_1^{-1} e, r, n}\right)]\mathcal{M}  \\
& [\left({h', g, v, a}\right)]\mathcal{L} = [\left({h_1^{-1} e', g_1^{-1} f, m, b}\right)] \mathcal{L}.
\end{aligned}  \right.
\]

As a consequence, from one hand, there is  \(\left({k_2, h_2}\right) \in \mathcal{M}_{ \scriptscriptstyle{1} }\) such that \(\mathsf{s}\left({k_2, h_2}\right)= \left({w, u}\right)\), \(\mathsf{t}\left({k_2, h_2}\right)= \left({r, n}\right)\) and
\begin{equation}\label{Eq:h1} 
\left({k, h, w, u}\right) = \left({k_1 l, h_1^{-1} e, r, n}\right) \left({k_2, h_2}\right) 
= \left({k_1 l k_2, h_1^{-1} e h_2, \mathsf{s}\left({k_2}\right), \mathsf{s}\left({h_2}\right) }\right). 
\end{equation}
On the other hand,  there is \(\left({h_3, g_2}\right) \in \mathcal{L}_{ \scriptscriptstyle{1} }\) such that \(\mathsf{s}\left({h_3, g_2}\right)= \left({v, a}\right)\), \(\mathsf{t}\left({h_3, g_2}\right)= \left({m, b}\right)\) and
\begin{equation}\label{Eq:k1} 
\left({h', g, v, a}\right) = \left({h_1^{-1} e', g_1 f, m, b}\right) \left({h_3, g_2}\right)
= \left({h_1^{-1} e' h_3, g_1 f g_2, \mathsf{s}\left({h_3}\right), \mathsf{s}\left({g_2}\right) }\right).
\end{equation}
Therefore we obtain the following equalities  from equations \eqref{Eq:h1} and \eqref{Eq:k1}  

\[
k_2 = l^{-1} k_1^{-1} k, \quad h_2 = e^{-1} h_1 h, \quad \text{ and }\quad  h_3 = e'^{-1} h_1 h', \quad  g_2 = f^{-1} g_1^{-1} g.
\]
Thus
\[ \begin{gathered}
\left({k_2, h_2}\right) \left({w, u, h^{-1} h', v, a}\right) \left({h_3, g_2}\right)^{-1}
= \left({\mathsf{s}\left({k_2}\right), \mathsf{t}\left({h_2}\right), h_2 h^{-1} h' h_3^{-1}, \mathsf{t}\left({h_3}\right), \mathsf{t}\left({g_2}\right) }\right) \\
= \left({w, n, e^{-1} h_1 h h^{-1} h' h'^{-1} h_1^{-1} e', m, b}\right)
= \left({w, n, e^{-1} e', m, b}\right), 
\end{gathered}
\]
which shows that $\cM[(w, u, h^{-1} h', v, a)]\cL = \cM[(w, n, e^{-1} e', m, b)] \cL$ in the orbits set $\cM \backslash X / \cL$. Henceforth,  \(\varphi\) is a well defined map.

In other direction, we have a well defined map given by 
\[ \begin{aligned}{\psi}  \colon & {\mathcal{M} \backslash X / \mathcal{L}} \longrightarrow {\mathcal{K}\backslash  \left({\mathcal{V} \otimes_{ \scriptscriptstyle{\mathcal{H} } } \mathcal{U} }\right) / \mathcal{G} } \\ & {\mathcal{M} [\left({w, u, h, v, a}\right)] \mathcal{L} }  \longrightarrow {\mathcal{K} \Big[{[\left({\iota_{ \scriptscriptstyle{w} }, \iota_{ \scriptscriptstyle{u} }, w, u}\right)] \mathcal{M} \otimes_{ \scriptscriptstyle{\mathcal{H} } } [\left({h, \iota_{ \scriptscriptstyle{a} }, v, a}\right)] \mathcal{L} }\Big] \mathcal{G}.}
\end{aligned} 
\]

Let us check  that \(\varphi\) and \(\psi\) are one the inverse of the other.
So, for each orbit
\[ 
\mathcal{K}\Big[{[\left({k, h, w, u}\right)] \mathcal{M} \otimes_{ \scriptscriptstyle{\mathcal{H}} } [\left({h', g, v, a}\right)] \mathcal{L}  }\Big]\mathcal{G} \; \in \; \mathcal{K} \backslash \left({\mathcal{V} \otimes_{ \scriptscriptstyle{\mathcal{H} } } \mathcal{U} }\right) / \mathcal{G}
\]
we have
\begin{eqnarray*}
\psi  \circ \varphi \Big({ \mathcal{K}\Big[{[\left({k, h, w, u}\right)] \mathcal{M} \otimes_{ \scriptscriptstyle{\mathcal{H}} } [\left({h', g, v, a}\right)] \mathcal{L}  }\Big]\mathcal{G} }\Big)  &=& 
 \psi \left({ \mathcal{M}\Big[\left({w, u, h^{-1} h', v, a}\right)\Big] \mathcal{L} }\right)  \\
&=& \mathcal{K}\Big[{[\left({\iota_{ \scriptscriptstyle{w} }, \iota_{ \scriptscriptstyle{u} }, w, u}\right)] \mathcal{M} \otimes_{ \scriptscriptstyle{H} } [\left({h^{-1} h', \iota_{ \scriptscriptstyle{a} }, v, a}\right)] \mathcal{L} }\Big] \mathcal{G} \\
&=&  \mathcal{K} \Big[\Big( k\big([(\iota_{\Sscript{w}}, \iota_{\Sscript{u}}, w, u)]\cM\big)h^{-1} \tensor{\cH} \big( [(h', \iota_{\Sscript{a}},v,a )] \cL \big) g^{-1}  \Big)\Big] \cG \\
&=& \mathcal{K}\left[{[\left({k, h, w, u}\right)] \mathcal{M} \otimes_{ \scriptscriptstyle{\mathcal{H}} } [\left({h', g, v, a}\right)] \mathcal{L}  }\right]\mathcal{G},
\end{eqnarray*}
which shows that $\psi \circ \varphi =id$. 

Conversely, for each element $
\mathcal{M} [\left({w, u, h, v, a}\right)] \mathcal{L} \;  \in \; \mathcal{M} \backslash X / \mathcal{L}$,
we have
\begin{eqnarray*}
\varphi \circ  \psi \left(\cM [(w, u, h, v, a)] \cL \right)  &=& \varphi \Big( \mathcal{K} \Big[{[\left({\iota_{ \scriptscriptstyle{w} }, \iota_{ \scriptscriptstyle{u} }, w, u}\right)] \mathcal{M} \otimes_{ \scriptscriptstyle{\mathcal{H} } } [\left({h, \iota_{ \scriptscriptstyle{a} }, v, a}\right)] \mathcal{L} }\Big] \mathcal{G} \Big) \\
&=& \mathcal{M} \left[\big({w, u, \iota_{ \scriptscriptstyle{u} }^{-1} h, u, a}\big)\right] \mathcal{L}  \\
&=& \mathcal{M}[ \left({w, u, h, u, a}\right)] \mathcal{L},
\end{eqnarray*}
whence $\varphi \circ \psi = id$. 
Therefore \(\varphi\) is bijective with inverse \(\psi\).

Let us check that, for every  element of the form
\[ 
y= [\left({\iota_{ \scriptscriptstyle{w} }, \iota_{ \scriptscriptstyle{u} }, w, u}\right)] \mathcal{M} \otimes_{ \scriptscriptstyle{\mathcal{H}} } [\left({h, \iota_{ \scriptscriptstyle{a} }, v, a}\right)] \mathcal{L}  \;  \in \; \left[ \mathcal{K} \backslash \left({\mathcal{V} \otimes_{ \scriptscriptstyle{\mathcal{H} } } \mathcal{U} }\right) / \mathcal{G} \right],
\]
there is the following equality of subgroupoids
$$
\left({\mathcal{K}, \mathcal{G}}\right)_{ \scriptscriptstyle{y} }  \, =\, \cM^{\Sscript{(w,\, u)}} * {}^{\Sscript{(h,\, \iota_a)}} \cL^{\Sscript{(u,\, a)}}.
$$

So, taking \(\left({k_3, g_3}\right) \in \mathcal{K}_{ \scriptscriptstyle{1} } \times \mathcal{G}_{ \scriptscriptstyle{1} }\) such that \(\mathsf{s}\left({k_3}\right)= \mathsf{t}\left({k_3}\right)= w\) and \(\mathsf{s}\left({g_3}\right)= \mathsf{t}\left({g_3}\right)= a\) we have \(k_3 y = y g_3\) if and only if
\[ y = [\left({k_3, \iota_{ \scriptscriptstyle{u} }, w, u}\right)] \mathcal{M} \otimes_{ \scriptscriptstyle{\mathcal{H} } } [\left({h, g_3, u, a}\right)] \mathcal{L},
\]
if and only if there exists \(h_4 \in \mathcal{H}_{ \scriptscriptstyle{1} }\) such that \(\mathsf{s}\left({h_4}\right)= \mathsf{t}\left({h_4}\right)= u\) and
\begin{eqnarray*}
\Big({[\left({\iota_{ \scriptscriptstyle{w} }, \iota_{ \scriptscriptstyle{u} }, w, u}\right)] \mathcal{M}, [\left({h, \iota_{ \scriptscriptstyle{a} }, u, a}\right)] \mathcal{L} }\Big)  &=& \Big(\big({[\left({k_3, \iota_{ \scriptscriptstyle{u} }, w, u}\right)] \mathcal{M}\big) h_4, h_4^{-1}\big( [\left({h, g_3, u, a}\right)] \mathcal{L}\big) }\Big)   \\
&=& \Big({[\left({k_3, \iota_{ \scriptscriptstyle{u} }h_4^{-1}, w, u}\right)] \mathcal{M} ,  [\left({h_4^{-1}h, g_3, u, a}\right)] \mathcal{L} }\Big) \\ 
&=& \Big({[\left({k_3, h_4^{-1}, w, u}\right)] \mathcal{M} ,  [\left({h_4^{-1}h, g_3, u, a}\right)] \mathcal{L} }\Big) \quad \in \cV \times \cU.
\end{eqnarray*}
This holds true, if and only if, there exists  \(h_4 \in \cH^{\Sscript{(u,\,  u)}}\) such that
\[ \left\lbrace  \begin{aligned}
& [\left({\iota_{ \scriptscriptstyle{w} }, \iota_{ \scriptscriptstyle{u} }, w, u}\right)] \mathcal{M} = [\left({k_3, h_4^{-1}, w, v}\right)] \mathcal{M} \; \in \; \cV,    \\
& [\left({h, \iota_{ \scriptscriptstyle{a} }, u, a}\right)] \mathcal{L} =  [\left({h_4^{-1} h, g_3, u, a}\right)] \mathcal{L} \; \in \; \cU,
\end{aligned}  \right.
\]
if and only if there exists \(h_4  \in \cH^{\Sscript{(u,\,  u)}}\) such that
\[ \left\lbrace  \begin{aligned}
& \left({k_3, h_4^{-1}}\right) \in \mathcal{M}^{\Sscript{\left({w, u}\right)} } \\
& \left({h, \iota_{ \scriptscriptstyle{a} } }\right)^{-1}  \left({h_4^{-1}, g_3}\right) \left({h, \iota_{ \scriptscriptstyle{a} } }\right) = \left({h, \iota_{ \scriptscriptstyle{a} } }\right)^{-1} \left({h_4^{-1} h, g_3}\right) \in \mathcal{L}^{\Sscript{{\left({v, a}\right)}}},
\end{aligned}  \right.
\]
if and only if there exists \(h_4  \in \cH^{\Sscript{(u,\,  u)}}\) such that
\[ \left\lbrace  \begin{aligned}
& \left({k_3, h_4^{-1}}\right) \in \mathcal{M}^{\Sscript{\left({w,\,  u}\right)} }  \\
& \left({h_4^{-1}, g_3}\right)  \in
{}^{\Sscript{(h,\, \iota_a)}}\mathcal{L}^{\Sscript{{\left({u,\,  a}\right)}}},
\end{aligned}  \right.
\]
if and only if
\[ \left({k_3, g_3}\right) \, \in \,  \mathcal{M}^{\Sscript{\left({w,\,  u}\right)} }    \ast {}^{\Sscript{(h,\, \iota_a)}}\mathcal{L}^{\Sscript{{\left({u,\,  a}\right)}}}.
\]

As a consequence we get the following isomorphisms of \(\left({\mathcal{K}, \mathcal{G}}\right)\)-bisets:
\[ \begin{gathered}
 \mathcal{V} \otimes_{ \scriptscriptstyle{\mathcal{H} } }  \mathcal{U}  
\cong \biguplus_{y \, \in \,  \rep_{ \scriptscriptstyle{\left(\mathcal{K}, \mathcal{G}\right)}}\left(\mathcal{V} \otimes_{ \scriptscriptstyle{\mathcal{H} } }  \mathcal{U} \right) } \left({\frac{\mathcal{K} \times \mathcal{G} }{ \left({\mathcal{K}, \mathcal{G} }\right)_{ \scriptscriptstyle{y} }  } }\right)^{\Sscript{\mathsf{L}} }   
\cong \biguplus_{  \left({w,\,  u,\,  h,\, v,\, a}\right) \, \in \,  \rep_{\scriptscriptstyle{\left(\mathcal{M}, \mathcal{L}\right)}}(X)  } \left({\frac{\mathcal{K} \times \mathcal{G} }{ \mathcal{M}^{\Sscript{\left({w,\,  u}\right)} }    \ast {}^{\Sscript{(h,\, \iota_a)}}\mathcal{L}^{\Sscript{{\left({u,\,  a}\right)}}}} }\right)^{\Sscript{\mathsf{L}} } ,  
\end{gathered}
\]
which depends on the choice of a representatives set of the orbits of the biset $X$, and this finishes the proof.
\end{proof}

As a final remark of this section, note that, as it has been done in \cite[Remark 4.1.6]{Bouc:2010} the Mackey formula could be used to characterize the admissible subcategories of the biset category of finite groupoids, once this has been opportunely defined. We do not go into the details because it would be very  technical and it will be published elsewhere.

\section{Examples using the equivalence relation groupoid}\label{sec:EquiRelation}
In this section we will give a simple application of Theorem \ref{thm:Main}, using the case of groupoids of equivalence relations as subgroupoids of groupoids of pairs, see Example \ref{exam:X}. In other words, we want to test this formula for this case. As we will see at the end of subsection \ref{ssec:EqMF}, this result is not surprising, although not immediate to  decipher. 
  
\subsection{Subgroupoids and equivalence relations}\label{ssec:SER}

Given a set \(H\), consider as in Example \ref{exam:X} the groupoid of pairs \(\mathcal{H}= \left({H \times H, H}\right)\) and let \(R\) be an equivalence relation on \(H\).
Given the groupoid of equivalence relation \(\mathcal{R}= \left({R,H}\right)\), we can consider the inclusion of groupoids
\begin{equation}\label{Eq:Rtau} 
\mathcal{R}= \left({R,H}\right) \hookrightarrow \mathcal{H}= \left({H \times H, H}\right).
\end{equation}
Following equation \eqref{Eq:L}, we know that
\[ 
\left({\frac{\mathcal{H} }{\mathcal{R} } }\right)^{ \scriptscriptstyle{L} } = \Set{ \left[{\left({a,u}\right) }\right] \mathcal{R}  | \left({a,u}\right) \in \mathcal{R}\left({\mathcal{H} }\right)^{ \scriptscriptstyle{\uptau} } = \mathcal{H}_{ \scriptscriptstyle{1} }  {{}_{ \scriptscriptstyle{\mathsf{s} } } {\times}}_{ \scriptscriptstyle{\uptau_0} }\, \mathcal{R}_{ \scriptscriptstyle{0} }  } 
\]
where \(a \in \mathcal{H}_{ \scriptscriptstyle{1} } \), \(\mathsf{s}\left({a}\right)= u \in \mathcal{R}_{ \scriptscriptstyle{0} } = H\) and where
\[ 
\left[{\left({a,u}\right) }\right]\mathcal{R} = \Set{ \left({ar, \mathsf{s}\left({r}\right) }\right)\in \mathcal{R}\left({\mathcal{H}}\right)^{ \scriptscriptstyle{\uptau} }   |  r \in \mathcal{R}_{ \scriptscriptstyle{1} }, \mathsf{s}\left({a}\right) = \mathsf{t}\left({r}\right)  } 
\]
with \(r \in \mathcal{R}_{ \scriptscriptstyle{1} }\); so there are \(h_3, h_4 \in H\) (in the same equivalence class) such that \(r= \left({h_3, h_4}\right)\), and  we also have  \(a= \left({h_1, h_2}\right)\in \mathcal{H}_{ \scriptscriptstyle{1} }\) with  \(h_1, h_2 \in H\).
In particular, we have  \(h_2 = \mathsf{s}\left({a}\right) = \mathsf{t}\left({r}\right) = h_3\), \(h_4 = \mathsf{s}\left({r}\right) \in \mathcal{R}_{ \scriptscriptstyle{0} } = H\), \(h_2 = \mathsf{s}\left({a}\right) = u \in \mathcal{R}_{ \scriptscriptstyle{0} }\) and
\[ ar= \left({h_1, h_2}\right) \left({h_3, h_4}\right) = \left({h_1, h_4}\right).
\]
As a consequence
\[ \left[{\left({a,u}\right) }\right] \mathcal{R} = \left[{\left({ \left({h_1, h_2}\right), h_2}\right)}\right] \mathcal{R} 
=\Set{\left({ \left({h_1, h_4}\right), h_4}\right) \in \mathcal{H}_{ \scriptscriptstyle{1} }  {{}_{ \scriptscriptstyle{\mathsf{s} } } {\times}}_{ \scriptscriptstyle{\uptau_0} }\, \mathcal{R}_{ \scriptscriptstyle{0} } = \left({H \times H}\right) \times H }
\]
and
\[ \left({\frac{\mathcal{H} }{\mathcal{R} } }\right)^{ \scriptscriptstyle{\mathsf{L} } } = \Big\{ \left[{\left({\left({h_1, h_2}\right), h_2}\right) }\right] \mathcal{R}  | \; h_1, h_2 \in H\Big\}.
\]

\begin{lemma}\label{lema:Velero}
\label{lLeftSetRel}
Let \(H/R\) be the quotient set of \(H\) by the equivalence relation \(R\).
Then \(H \times \left({H/R}\right)\) becomes a left \(\mathcal{H}\)-set with structure map and action given by
\[ \begin{aligned}{\varsigma}  \colon & {H \times \frac{H}{R} } \longrightarrow {\mathcal{H}_{ \scriptscriptstyle{0} } = H} \\ & {\left({h_1, \overline{h_2}}\right) }  \longrightarrow {h_1}\end{aligned} 
\qquad \text{and} \qquad
\begin{aligned}
&\mathcal{H}_{ \scriptscriptstyle{1} }  {{}_{ \scriptscriptstyle{\mathsf{s}} } {\times}}_{ \scriptscriptstyle{\varsigma} }\, \left({H \times \frac{H}{R}}\right)  \longrightarrow H \times \frac{X}{R}  \\
&\left({\left({h_3,h_1}\right), \left({h_1,\overline{h_2} }\right) }\right)  \longrightarrow \left({h_3, \overline{h_2} }\right),
\end{aligned}
\]
where for every $h \in H$, $\bara{h}$ denotes the equivalence class of $h$ modulo the relation $R$. 
\end{lemma}
\begin{proof}
It is immediate. 
\end{proof}
The following lemma is also straightforward. 
\begin{lemma}
\label{lExGrpdEqRelPairs}
Given \(h_1, h_2, h_3, h_1', h_2', h_3' \in H\) we have that \(\left({\left({h_1, h_2}\right) h_2}\right) \sim \left({\left({h_1', h_2'}\right), h_2'}\right)\), as representative elements of the coset \(\left[{\left({\left({h_1, h_2}\right), h_2}\right) }\right]\mathcal{R}\), if and only if \(h_1 = h_1'\) and \(\overline{h_2} = \overline{h_2'}\) as equivalence classes in \(H/R\).
\end{lemma}

Now we have the following isomorphism of groupoid-sets:
\begin{proposition}\label{pIsomLCosetsRel}
Using the morphism  of groupoids \eqref{Eq:Rtau}, we consider \(\left({\mathcal{H}/\mathcal{R} }\right)^{ \scriptscriptstyle{\mathsf{L}} }\)  as a left \(\mathcal{H}\)-set with structure map and action given explicitly in equation \eqref{Eq:Paloma}.  Then, 
there is an isomorphism of left \(\mathcal{H}\)-set
\[ \begin{aligned}{\psi}  \colon & {\left({\frac{\mathcal{H} }{\mathcal{R} } }\right)^{ \scriptscriptstyle{\mathsf{L} } } } \longrightarrow {H \times \frac{H}{R} } \\ & {\left[{\left({\left({h_1, h_2}\right), h_2}\right) }\right] \mathcal{R} }  \longrightarrow {\left({h_1, \overline{h_2} }\right) .}
\end{aligned} 
\]
\end{proposition}
\begin{proof}
The map \(\psi\) is well defined  and injective thanks to  Lemma \ref{lExGrpdEqRelPairs}, and the surjectivity is obvious.
Therefore, we only have to check that \(\psi\) is a homomorphism of left \(\mathcal{H}\)-set.
The condition on the structure maps is trivial so we only have to check the condition on the actions. Henceforth,  take  \(\left({y_1, h_1}\right) \in \mathcal{H}_{ \scriptscriptstyle{1} }\) and \(\left[{\left({\left({h_1, h_2}\right), h_2}\right) }\right] \mathcal{R} \in \left({\mathcal{H}/\mathcal{R}}\right)^{ \scriptscriptstyle{\mathsf{L}} }\):
using the action of equation \eqref{Eq:Paloma}, we compute
\[ \left({y_1, h_1}\right) \cdot \left[{\left({\left({h_1, h_2}\right), h_2}\right) }\right] \mathcal{R} = \left[{\left({\left({y_1, h_2}\right), h_2}\right) }\right] \mathcal{R}
\]
and we apply \(\psi\) to obtain \(\left({y_1, \overline{h_2}}\right)\).
Now we apply again  \(\psi\) to \(\left[{\left({\left({h_1, h_2}\right), h_2}\right) }\right] \mathcal{R}\), we obtain \(\left({h_1, \overline{h_2}}\right)\) and finish by applying the action once again to obtain \(\left({y_1, h_1}\right) \cdot \left({h_1, \overline{h_2} }\right) = \left({y, \overline{h_2} }\right)\).
\end{proof}

\subsection{Equivalence relations and Mackey formula}\label{ssec:EqMF}

Given sets \(H\), \(K\) and \(G\), let us consider the groupoids of pairs \(\mathcal{H}= \left({H \times H, H}\right)\), \(\mathcal{K}= \left({K \times K, K}\right)\) and \(\mathcal{G}= \left({G \times G, G}\right)\).
We have the isomorphisms of groupoids (they're just a switch)
\[ \mathcal{K} \times \mathcal{H} = \left({K \times K \times H \times H, K \times H}\right) \cong \left({K \times H \times K \times H, K \times H}\right) = : \mathcal{A}
\]
and
\[ \mathcal{H} \times \mathcal{G} = \left({H \times H \times G \times G, H \times G}\right) \cong \left({H \times G \times H \times G, H \times G}\right) = : \mathcal{B}
\]
that we call
\[ \gamma_1 \colon \mathcal{K} \times \mathcal{H} \longrightarrow \mathcal{A} 
\qquad \text{and} \qquad
\gamma_2 \colon \mathcal{H} \times \mathcal{G} \longrightarrow \mathcal{B},
\]
respectively. Note that \(\mathcal{A}\) is the groupoid of pairs with respect to the set \(K \times H\) and \(\mathcal{B}\) is the groupoid of pairs with respect to  the set \(H \times G\).
Let \(R\) be an equivalence relation on \(K \times H\) and \(Q\) be an equivalence relation on \(H \times G\): we have \(R \subseteq \mathcal{A}_{ \scriptscriptstyle{1} }\) and \(Q \subseteq \mathcal{B}_{ \scriptscriptstyle{1} }\) and we can consider the inclusion of groupoids
\[ \mathcal{R} : = \left({R, K \times H}\right) \hookrightarrow \mathcal{A} 
\qquad \text{and} \qquad
\mathcal{Q}= \left({Q, H \times G}\right) \hookrightarrow \mathcal{B}.
\]

Defined the groupoids \(\mathcal{M} = \gamma_1^{-1}  \left({\mathcal{R}}\right)\) and \(\mathcal{L} = \gamma_2^{-1} \left({\mathcal{Q}}\right)\), we clearly  have the isomorphisms 
$$
\left({\frac{\mathcal{K} \times \mathcal{H}}{\mathcal{M}}}\right)^{ \scriptscriptstyle{\mathsf{L}} }  \cong \left({\frac{\mathcal{A}}{\mathcal{R}} }\right)^{ \scriptscriptstyle{\mathsf{L}} } \quad \text{and} \quad   \left({\frac{\mathcal{H} \times \mathcal{G}}{\mathcal{L}} }\right)^{ \scriptscriptstyle{\mathsf{L}} } \cong  \left({\frac{\mathcal{B}}{\mathcal{Q}}}\right)^{ \scriptscriptstyle{\mathsf{L}} }
$$
of \(\left({\mathcal{K}, \mathcal{H}}\right)\)-bisets and 
of \(\left({\mathcal{H}, \mathcal{G}}\right)\)-bisets, respectively. 
In this way we get an other isomorphism of \(\left({\mathcal{K}, \mathcal{G}}\right)\)-bisets
\[ 
\left({\frac{\mathcal{K} \times \mathcal{H}}{\mathcal{M}}}\right)^{ \scriptscriptstyle{\mathsf{L}} } \otimes_{\mathcal{H}} \left({\frac{\mathcal{H} \times \mathcal{G}}{\mathcal{L}} }\right)^{ \scriptscriptstyle{\mathsf{L}} } \cong \left({\frac{\mathcal{A}}{\mathcal{R}} }\right)^{ \scriptscriptstyle{\mathsf{L}} } \otimes_{\mathcal{H}} \left({\frac{\mathcal{B}}{\mathcal{Q}}}\right)^{ \scriptscriptstyle{\mathsf{L}} }
\]
where, in the right hand term, the structure map and the action are opportunely defined.

On the other hand, we know that 
\[ \left({\frac{\mathcal{A}}{\mathcal{R}}}\right)^{ \scriptscriptstyle{\mathsf{L}} } = \Big\{ \left[{\left({\left({k_1, h_1, k_2, h_2}\right), \left({k_2, h_2}\right)}\right) }\right] \mathcal{R} | h_1, h_2 \in H, k_1, k_2 \in K \Big\}
\]
and
\[ \left({\frac{\mathcal{B}}{\mathcal{Q}}}\right)^{ \scriptscriptstyle{\mathsf{L}} } = \Big\{ \left[{\left({\left({h_3, g_1, h_4, g_2}\right), \left({h_4, g_2}\right)}\right) }\right] \mathcal{Q} | h_3, h_4 \in H, g_1, g_2 \in G \Big\}.
\]

\begin{lemma}
The following Cartesian product of sets
\[ \left({K \times H}\right) \times \left({\frac{K \times H}{R}}\right) 
\]
admits a structure of  a \(\left({\mathcal{K}, \mathcal{H}}\right)\)-biset.
\end{lemma}
\begin{proof}
The fact that it is a left \(\mathcal{K}\)-set can be proved like in Lemma~\(\ref{lLeftSetRel}\). Therefore we only have to show that it is a right \(\mathcal{H}\)-set.
The structure map is given by
\[ \begin{aligned}{ \vartheta}  \colon & {\left({K \times H}\right) \times \left({\frac{K \times H}{R}}\right) } \longrightarrow {\mathcal{H}} \\ & {\left({\left({k_1, h_1}\right), \overline{\left({k_2, h_2}\right) } }\right) }  \longrightarrow {h_1}\end{aligned} 
\]
and action is defined by 
\[ \begin{aligned}
& \left({K \times H}\right) \times \left({\frac{K \times H}{R}}\right)   {{}_{ \scriptscriptstyle{\vartheta} } {\times}}_{ \scriptscriptstyle{\mathsf{s}} }\, \mathcal{H}_{ \scriptscriptstyle{1} } \longrightarrow \left({K \times H}\right) \times \left({\frac{K \times H}{R}}\right) \\
& \Big( \left({\left({k_1, h_1}\right), \overline{k_2, h_2} }\right), \left({h_1, h_3}\right) \Big) \longrightarrow  \left({\left({k_1, h_3}\right), \overline{\left({k_2, h_2}\right) } .}\right)
\end{aligned}
\]
Now we have to prove the action conditions.
The neutral element conditions and the associativity are trivial.
Regarding the other condition, let be
\[
 \Big( \left({\left({k_1, h_1}\right), \overline{\left({ k_2, h_2}\right)} }\right), \left({h_1, h_3}\right) \Big)  \in  \Big(  \left({K \times H}\right) \times \left({\frac{K \times H}{R}}\right)  \Big) \,  {{}_{ \scriptscriptstyle{\vartheta} } {\times}}_{ \scriptscriptstyle{\mathsf{t}} }\, \mathcal{H}_{ \scriptscriptstyle{1} }:
\]
we have
\[ \vartheta\left({\left({k_1, h_3}\right), \overline{\left({k_2, h_2}\right) } }\right) = h_3 = \mathsf{s}\left({h_1, h_3}\right).
\]
Lastly, the compatibility conditions of the left and right actions are immediate to verify.
\end{proof}

\begin{proposition}
Keeping the above notations, we have the following isomorphism of \(\left({\mathcal{K}, \mathcal{H}}\right)\)-bisets
\[ \begin{aligned}{\varphi}  \colon & {\left({\frac{\mathcal{A}}{\mathcal{R}}}\right)^{ \scriptscriptstyle{\mathsf{L}} }} \longrightarrow {\left({K \times H}\right) \times \frac{K \times H}{R}} \\ & {\left[{\left({\left({k_1, h_1, k_2, h_2}\right), \left({k_2, h_2}\right)}\right) }\right] \mathcal{R}}  \longrightarrow {\Big(\left({k_1, h_1}\right), \overline{\left({k_2, h_2}\right) } \Big).}\end{aligned}
\]
\end{proposition}
\begin{proof}
As a map of left $\cK$-sets,  \(\varphi\) is an isomorphism thanks to Proposition \ref{pIsomLCosetsRel} applied to the set \(K \times H\).
So we just have to prove that it is an homomorphism of right \(\mathcal{H}\)-sets.
The condition on the structure map is obvious.
Regarding the condition on the action maps, given \(x=\left[{\left({\left({k_1, h_1, k_2, h_2}\right), \left({k_2, h_2}\right)}\right) }\right] \mathcal{R} \in \left({\mathcal{A}/ \mathcal{R}}\right)^{ \scriptscriptstyle{\mathsf{L}} }\) and \(\left({h_1, h_3}\right) \in \mathcal{H}_{ \scriptscriptstyle{1} }\),  we have
\[ x \cdot \left({h_1, h_3}\right) = \left[{\left({\left({k_1, h_3, k_2, h_2}\right), \left({k_2, h_2}\right)}\right) }\right] \mathcal{R} \xymatrix{   \ar@{|->}[r]^{\varphi} &    }  \left({\left({k_1, h_3}\right), \overline{\left({k_2, h_2}\right) } }\right) .
\]
On the other hand, we have 
$$
\varphi(x) \cdot  \left({h_1, h_3}\right) \,=\, \left({\left({k_1, h_1, \overline{\left({k_2, h_2}\right) } }\right) }\right) \cdot \left({h_1, h_3}\right) \,=\, \left({\left({k_1, h_3}\right), \overline{\left({k_2, h_2}\right) } }\right),
$$
and this finishes the proof. 
\end{proof}

Since a similar proposition is true also for \(\left({\mathcal{B}/ \mathcal{Q}}\right)^{ \scriptscriptstyle{\mathsf{L}} }\), we have the following isomorphism of \(\left({\mathcal{K}, \mathcal{G}}\right)\)-bisets:
\begin{equation}
\label{eBisetTensRelIsom}
\left({\frac{\mathcal{A} }{\mathcal{R}} }\right)^{ \scriptscriptstyle{\mathsf{L}} } \otimes_{\mathcal{H} } \left({\frac{\mathcal{B} }{\mathcal{Q}} }\right)^{ \scriptscriptstyle{\mathsf{L}} } \longrightarrow \Big( \left({K \times H}\right) \times \left({\frac{K \times H}{R} }\right) \Big) \otimes_{\mathcal{H} } \Big( \left({H \times G}\right) \times \left({\frac{H \times G}{Q} }\right) \Big).
\end{equation} 
The typical element of the right hand side of formula~\(\eqref{eBisetTensRelIsom}\) is
\begin{equation}
\label{eBisetTensRelEle}
y= \left({\left({k_1, h_3}\right), \overline{\left({k_2, h_2}\right) } }\right) \otimes_{\mathcal{H} }  \left({\left({h_3, g_1}\right), \overline{\left({h_4, g_2}\right) }  }\right) 
\end{equation} 
for \(k_1, k_2 \in K\), \(h_2, h_3, h_4 \in H\) and \(g_1, g_2 \in G\).
Notice that, by the definition of the tensor product over \(\mathcal{H}\), the element \(h_3\) should appear in both factors of the tensor product.   Therefore, there is the following isomorphism of \(\left({\mathcal{K}, \mathcal{G}}\right)\)-bisets given explicitly by:
\[ \begin{aligned}
& \Big( \left({K \times H}\right) \times \left({\frac{K \times H}{R} }\right) \Big) \otimes_{\mathcal{H} } \Big( \left({H \times G}\right) \times \left({\frac{H \times G}{Q} }\right) \Big)  \longrightarrow K \times \frac{K \times H}{R} \times \frac{H \times G}{Q} \times G \\
&\left({\left({k_1, h_3}\right), \overline{\left({k_2, h_2}\right) } }\right) \otimes_{\mathcal{H} }  \left({\left({h_3, g_1}\right), \overline{\left({h_4, g_2}\right) }  }\right)  \longrightarrow  \left({k_1, \overline{\left({k_2, h_2}\right) }, \overline{\left({h_4, g_2}\right) }, g_1}\right).
\end{aligned}
\]
This gives us the left hand side of the Mackey formula \eqref{eMackeyFormulaBisets} in the situation under consideration.

Let us pass to the right hand side of the formula  \eqref{eMackeyFormulaBisets}.  Consider  the biset $X$  given  in equation \eqref{Eq:Xx}, and fix a representative set  \(\rep_{(\mathcal{M}, \,  \mathcal{L})}\left({X}\right)\). Using  the notations of Proposition \ref{pModifDoubleCosets}, we have 
\begin{equation*}
X= \Set{ \left({w, u, h, v, a}\right) \in \mathcal{K}_{ \scriptscriptstyle{0} } \times \mathcal{H}_{ \scriptscriptstyle{0} } \times \mathcal{H}_{ \scriptscriptstyle{1} } \times \mathcal{H}_{ \scriptscriptstyle{0} } \times \mathcal{G}_{ \scriptscriptstyle{0} }  |  \begin{gathered}
\left({w, u}\right) \in \mathcal{M}_{ \scriptscriptstyle{0} }, \left({v, a}\right) \in \mathcal{L}_{ \scriptscriptstyle{0} }, \\
 u = \mathsf{t}\left({h}\right), v= \mathsf{s}\left({h}\right) 
\end{gathered} }.
\end{equation*}

In our case, it is clear that $X$ is identified, as a $(\cM, \cL)$-biset with 
\[ 
Y = \Set{ \left({k,  \left({h_1, h_2}\right), g}\right) \in K \times H  \times H \times G  }= K  \times \mathcal{H}_{ \scriptscriptstyle{1} } \times G.
\]
Therefore, one can choose a bijection  between their representative sets
$ \rep_{\left({\mathcal{M}, \, \mathcal{L}}\right)}\left({X}\right)$ and
$\rep_{\left({\mathcal{M},\,  \mathcal{L}}\right)}\left({Y}\right)$. As a consequence, we obtain the following isomorphism of \(\left({\mathcal{K}, \mathcal{G}}\right)\)-bisets:
\begin{equation}\label{eIsomBisetsMackeyRel}
 \biguplus_{\left({k, h_1, \left({h_1, h_2}\right), h_2, g}\right) \, \in \, \rep_{(\mathcal{M},\, \mathcal{L})}\left({X}\right)} \left( \frac{\mathcal{K} \times \mathcal{H} }{ \mathcal{M}^{ \scriptscriptstyle{\left({k, h_1}\right) } } \ast \left({ {^{ \scriptscriptstyle{\left({\left({h_1, h_2}\right), \,\iota_{ \scriptscriptstyle{g} } }\right) } } } \mathcal{L}^{ \scriptscriptstyle{\left({h_2, \,g}\right) } } }\right) }  \right)^{ \scriptscriptstyle{\mathsf{L} } }   \cong  \biguplus_{\left({k, \left({h_1, h_2}\right), g}\right) \, \in \, \rep_{(\mathcal{M},\, \mathcal{L})}\left({Y}\right)} \left( \frac{\mathcal{K} \times \mathcal{H} }{ \mathcal{M}^{ \scriptscriptstyle{\left({k, h_1}\right) } } \ast \left({ {^{ \scriptscriptstyle{\left({\left({h_1, h_2}\right),\, \iota_{ \scriptscriptstyle{g} } }\right) } } } \mathcal{L}^{ \scriptscriptstyle{\left({h_2, \,g}\right) } } }\right) }  \right)^{ \scriptscriptstyle{\mathsf{L} }. }  
\end{equation} 
On the other hand, given \(k, k' \in K\), \(h_1, h_2, h_1', h_2' \in H\) and \(g, g' \in G\),  we have
\[ \left({k, \left({h_1, h_2}\right), g}\right) \sim \left({k' , \left({h_1', h_2'}\right), g}\right),
\]
as elements of the \(\left({\mathcal{M}, \mathcal{L}}\right)\)-biset \(Y\), if and only if there are \(m = \left({k, k', h_1, h_1'}\right)\in \mathcal{M}_{ \scriptscriptstyle{1} }\) and \(l= \left({h_2, h_2', g, g'}\right)  \in \mathcal{L}_{ \scriptscriptstyle{1} }\) such that
\[ \left({k, \left({h_1, h_2}\right), g}\right) = m \left({k' , \left({h_1', h_2'}\right), g'}\right) l ,
\]
if and only if \(\left({k, k', h_1, h_1'}\right) \in \mathcal{M}_{ \scriptscriptstyle{1} }\) and \(\left({h_2, h_2', g, g'}\right) \in \mathcal{L}_{ \scriptscriptstyle{1} }\), if and only if \(\left({k, h_1}\right) R \left({k', h_1'}\right)\) and \(\left({h_2, g}\right) Q \left({h_2', g'}\right)\).
Therefore, we have a bijection
\begin{equation}\label{Eq:Yiso} 
\rep_{(\mathcal{M},\, \mathcal{L})}\left({Y}\right)
\, \simeq \, 
\frac{K \times H}{R} \times \frac{H \times G}{Q}.
\end{equation}

Let us now decipher the denominator parts of the right hand-side of equation \eqref{eIsomBisetsMackeyRel}.  For any element   \(\left({k, \left({h_1, h_2}\right), g}\right) \in Y\), we have
\[ 
\mathcal{M}^{ \scriptscriptstyle{\left({k,\,  h_1}\right) } }  = \Big\{\left({k, k, h_1, h_1}\right) \Big\} ,
\qquad \mathcal{L}^{ \scriptscriptstyle{\left({h_2,\,g }\right) } }  = \Big\{\left({h_2, h_2, g, g}\right) \Big\} 
\]
and, using the multiplication of groupoids of pairs we get
\[ \begin{gathered}
{ ^{ \scriptscriptstyle{\left({\left({h_1,\,  h_2}\right), \, \iota_{ \scriptscriptstyle{g} } }\right)} }   }  \mathcal{L}^{ \scriptscriptstyle{\left({h_2, g}\right) } } \, =\,  { ^{ \scriptscriptstyle{\left({\left({h_1, h_2}\right),\; \left({g, g}\right) }\right)} }   }  \mathcal{L}^{ \scriptscriptstyle{\left({h_2, g}\right) } } \,
= \,  \Set{\Big( \left({h_1, h_2}\right), \left({g,g}\right) \Big) \Big(\left({h_2, h_2}\right), \left({g,g}\right) \Big) \Big(\left({h_1, h_2}\right), \left({g,g}\right) \Big)^{-1}} \\
= \Set{\Big( \left({h_1, h_2}\right) \left({h_2, h_2}\right) \left({h_2, h_1}\right), \left({g,g}\right) \left({g,g}\right) \left({g,g}\right) \Big)} \,
=\, \Set{\Big( \left({h_1, h_1}\right), \left({g,g}\right) \Big)} 
= \Big\{\left({h_1, h_1, g,g}\right) \Big\}.
\end{gathered}
\]
As a consequence, we have
\[  \mathcal{M}^{ \scriptscriptstyle{\left({k,\, h_1}\right) } } \ast \left({ {^{ \scriptscriptstyle{\left({\left({h_1,\, h_2}\right),\,  \iota_{ \scriptscriptstyle{g} } }\right) } } } \mathcal{L}^{ \scriptscriptstyle{\left({h_2, \,g}\right) } } }\right)  = \big\{\left({k,k,g,g}\right) \} 
= \left({\mathcal{K} \times \mathcal{G} }\right)^{ \scriptscriptstyle{\left({k,\,g}\right) } },
\]
which leads to the following  isomorphism of \(\left({\mathcal{K}, \mathcal{G}}\right)\)-biset 
$$
 \left( \frac{\mathcal{K} \times \mathcal{H} }{ \mathcal{M}^{ \scriptscriptstyle{\left({k, h_1}\right) } } \ast \left({ {^{ \scriptscriptstyle{\left({\left({h_1, h_2}\right), \iota_{ \scriptscriptstyle{g} } }\right) } } } \mathcal{L}^{ \scriptscriptstyle{\left({h_2, g}\right) } } }\right) }  \right)^{ \scriptscriptstyle{\mathsf{L} } } \,\cong \,  \Big\{ \left({k_1, k, g_1, g}\right) \in K \times K \times G \times G |  \left({k_1, g_1}\right) \in K \times G \}.
$$

Combining with formula \eqref{eIsomBisetsMackeyRel}, we arrive to the isomorphism of \(\left({\mathcal{K}, \mathcal{G}}\right)\)-bisets
\[ \begin{gathered}
 \biguplus_{\left({k, h_1, \left({h_1,\, h_2}\right), h_2, g}\right) \, \in \, \rep_{(\mathcal{M},\, \mathcal{L})}\left({X}\right)} \left( \frac{\mathcal{K} \times \mathcal{H} }{ \mathcal{M}^{ \scriptscriptstyle{\left({k,\, h_1}\right) } } \ast \left({ {^{ \scriptscriptstyle{\left({\left({h_1, h_2}\right), \,\iota_{ \scriptscriptstyle{g} } }\right) } } } \mathcal{L}^{ \scriptscriptstyle{\left({h_2,  \,g}\right) } } }\right) }  \right)^{ \scriptscriptstyle{\mathsf{L} } }   \,  \cong \, \biguplus_{\left({k, \left({h_1,\, h_2}\right), g}\right) \, \in \, \rep_{(\mathcal{M},\, \mathcal{L})}\left({Y}\right)} \left({K \times G}\right) 
\end{gathered}
\]

Therefore, using the bijection of equation \eqref{Eq:Yiso},   in this case the Mackey formula can be read as the following isomorphism of \(\left({\mathcal{K}, \mathcal{G}}\right)\)-bisets
\[ K \times \frac{K \times H}{R} \times \frac{H \times G}{Q} \times G \cong \biguplus_{ \frac{K \times H}{R} \times \frac{H \times G}{Q} } \left({K \times G}\right)
\]
where the right hand side is the coproduct of \(\frac{K \times H}{R} \times \frac{H \times G}{Q}\) times \(K \times G\) in the category of \(\left({\mathcal{K}, \mathcal{G}}\right)\)-bisets.

\subsection{A right groupoid-set with a non surjective structure map}\label{ssec:bonito}
We will give an example of a right $\cG$-set whose structure map is not surjective, completing by this the observations made in Remark \ref{rem:Core}. 
Given a set \(S\), let be \(\emptyset \neq S' \subsetneqq S\) and let \(R\) be an equivalence relation on \(S\).
Assume that there is \(x_0 \in S \setminus S'\) such that for every \(y \in S'\), \(\left({x_0,y}\right) \notin R\).
Let's define \(R_{S'} = R \cap \left({S' \times S'}\right)\), that is, the restriction of $R$ to $S'$.  It is clear that \(R_{S'}\) is an equivalence relation on \(S'\).

We have the following two groupoids of equivalence relation: \(\mathcal{G}= \left({R, S}\right)\) and \(\mathcal{H}= \left({R_{S'}, S'}\right)\) with the inclusion 
 \(\uptau \colon \mathcal{H}  \lhook\joinrel\longrightarrow \mathcal{G}\) of groupoids. Let us define \(X:= \mathcal{H}_{ \scriptscriptstyle{0} } \,  {{}_{ \scriptscriptstyle{\uptau} } {\times}}_{ \scriptscriptstyle{\mathsf{s}} }\, \mathcal{G}_{ \scriptscriptstyle{1} } = S' \,  {{}_{ \scriptscriptstyle{\uptau} } {\times}}_{ \scriptscriptstyle{\mathsf{t}} }\, R\).
Note that we have to choose \(S' \neq \emptyset\) otherwise \(X\) would be empty.
The set \(X\) becomes a right \(\mathcal{G}\)-set with structure map
\[ 
\begin{aligned}{\varsigma}  \colon & {X} \longrightarrow {\mathcal{G}_{ \scriptscriptstyle{0} }} \\ & {\left({s', (a,b)}\right) }  \longrightarrow {\mathsf{s}(a,b)= b}\end{aligned}
\]
and action
\[ \begin{aligned}
& X \,  {{}_{ \scriptscriptstyle{\varsigma} } {\times }}_{ \scriptscriptstyle{\mathsf{s} } }\, \mathcal{G}_{ \scriptscriptstyle{1} }  \longrightarrow X \\
& \Big( \left({s', (a,b)}\right) , \left({b', d}\right) \Big) \longrightarrow \left({s', (a,b)\left({b',d}\right)}\right) = \left({s', (a,d).}\right), 
\end{aligned}
\]
where
\[ b= \varsigma\left({s', (a,b)}\right) = \mathsf{t}\left({b',d}\right)= b'.
\]
The axioms of right \(\mathcal{G}\)-set are not difficult to  verify.

Now, we want to prove that \(\varsigma\) is not surjective.
By contradiction, let us assume that \(\varsigma\) is surjective. Therefore, there is \(\left({s', (a,b)}\right) \in X = S' \,  {{}_{ \scriptscriptstyle{\uptau} } {\times }}_{ \scriptscriptstyle{\mathsf{t} } }\, R\) such that \(x_0 = \varsigma\left({s', (a,b)}\right) = b\): in particular \(a = \mathsf{t}(a,b) = \uptau_{ \scriptscriptstyle{0} }\left({s'}\right) = s'\), hence  \(\left({s', x_0}\right)= (a,b) \in R\), with \(s' \in S'\), which contradicts our assumption. As a consequence \(x_0\) doesn't belong to the image of \(\varsigma\) in \(\mathcal{G}_{ \scriptscriptstyle{0} }\) and \(\varsigma\) is not surjective.
As a particular implementation of this example, we can choose \(S= \mathbb{R}\), \(S' = \mathbb{Q}\), $x_0=\sqrt{2}$ and for every \(s_1, s_2 \in S\) we can define \(s_1 R s_2\) if and only if \(s_1- s_2 \in  \mathbb{Z}\).

\bigskip
\noindent{}
\textbf{Acknowledgement.} The second author wants to thank the departement of algebra of the University of Granada, Campus of Ceuta, for the marvellous experience, the friendly and welcoming ambience, as well as for the partial support of the grant MTM2016-77033-P, which enabled him to undertake the long journey there.


\begin{thebibliography}{}

\bibitem{BonMackeyTypeA}
C.~Bonnaf\'e,
\emph{Mackey formula in type A},
Proc. London Math. Soc.~(3) 80 (2000), no.~3, pp. 545–574.

\bibitem{BonCorMackeyTypeA}
C.~Bonnaf\'e,
\emph{Corrigenda: ``Mackey formula in type A''} [Proc. London Math. Soc.~(3) 80 (2000), no.~3, pp. 545--574],
Proc. London Math. Soc.~(3) 86 (2003), no.~2, pp. 435--442.


\bibitem{Bouc:2010}
S.~Bouc, \emph{Biset Functors for finite Groups}, LNM, vol.~1999. Springer-Verlag Berlin Heidelberg 2010.


\bibitem{BoucBisetsCatTensProd}
S.~Bouc,
\emph{Bisets as categories and tensor product of induced bimodules},
Appl. Categ. Structures~18 (2010), no.~5, pp. 517–521.


\bibitem{Brown:1987}
R.~Brown,
\emph{From groups to groupoids: A brief survey.},
Bull. London Math. Soc.~19 (1987), no.~2, pp. 113–134.



\bibitem{Cartier:2008}
P.~Cartier,
\emph{Groupo\"ides de Lie et leurs Algébro\"ides},
S\'eminaire Bourbaki $60^e$ ann\'ee, 2007-2008, num.~987, pp. 165--196.


\bibitem{Connes:1994}
A.~Connes,
\emph{Noncommutative Geometry},
Academic Press, Inc., San Diego, CA, 1994.

\bibitem{CurtReiRepThFinGrAsAlg}
C.~W.~Curtis and I.~Reiner,
\emph{Representation Theory of Finite Groups and Associative Algebras},
Pure and Applied Mathematics, Vol.~XI Interscience Publishers, a division of John Wiley \& Sons, New York-London 1962 xiv+685~pp.



\bibitem{DigMiRepFinGrLieT}
F.~Digne and J.~Michel,
\emph{Representations of Finite Groups of Lie Type},
London Math. Soc. Student Texts~21, Cambridge Univ. Press, Cambridge, 1991, iv+159~pp.

\bibitem{DeliLusRepRedGrFinFi}
P.~Deligne and G.~Lusztig,
\emph{Representations of reductive groups over finite fields},
Ann. of Math.~(2) 103 (1976), no.~1, pp.~103–161.

\bibitem{DeliLusDuaRepRedGrFinFiII}
P.~Deligne and G.~Lusztig,
\emph{Duality for representations of a reductive group over a finite field, II},
J.~Algebra~81 (1983), no.~2, pp.~540–545.


\bibitem{DemGab:GATIGAGGC}
M.~Demazure and P.~Gabriel, \emph{Groupes alg\'ebriques. {T}ome {I}:
{G}\'eom\'etrie alg\'ebrique, g\'en\'eralit\'es, groupes commutatifs}, Masson
\& Cie, \'Editeur, Paris; North-Holland Publishing Co., Amsterdam, 1970, Avec
un appendice {\em Corps de classes local} par Michiel Hazewinkel.

\bibitem{Higgins:1971}
P.~J.~Higgins, \emph{Notes on categories and groupoids}, Van Nostrand  Reinhold, Mathematical Studies~32, London 1971. 

\bibitem{Jelenc:2013}
B.~Jelenc. \emph{Serre fibrations in the Morita category of topological groupoids}, Topol.  Appl. 2003, 160, pp.~9--13.

\bibitem{GiraCohoNAb}
J.~Giraud, Cohomologie non-abélienne, Die Grundlehren der mathematischen Wissenschaften in Einzeldarstellungen, vol.~179, Springer-Verlag, Berlin, Heidelberg, New York, 1971.

\bibitem{ElKaoutit:2017}
L.~El Kaoutit, \emph{On geometrically transitive Hopf algebroids}, J. Pure Appl. Algebra Vol.~222 (2018), no.~11, pp.~3483--3520. \url{https://doi.org/10.1016/j.jpaa.2017.12.019}.


\bibitem{Kaoutit/Kowalzig:14} 
L.~El Kaoutit and N.~Kowalzig,
\emph{Morita theory for Hopf algebroids, principal bibundles, and weak equivalences}
Doc. Math.~22 (2017), pp.~551–609.


\bibitem{LusSpaIndUniCl}
G.~Lusztig and N.~Spaltenstein,
\emph{Induced unipotent classes},
J. London Math. Soc.~(2) 19 (1979), no.~1, pp.~41–52.

\bibitem{MackeyIndRepGr}
G.~W.~Mackey,
\emph{On induced representations of groups},
Amer. J. Math.~73, (1951), pp.~576–592.

\bibitem{MackeyIndRepLocCoGr}
G.~W.~Mackey,
\emph{Induced representations of locally compact groups I},
Ann. of Math.~(2) 55, (1952), pp.~101–139.

\bibitem{Mackenzie:2005} K.~C.~H.~Mackenzie, \emph{General Theory of Lie Groupoids and Lie Algebroids},
London Math. Soc. Lecture Note Ser., vol.~213, Cambridge Univ. Press, Cambridge, 2005.


\bibitem{Moedijk/Mrcun:2005}
I.~Moerdijk and J.~Mr\v{c}un, \emph{Lie groupoids, sheaves and cohomology}, Poisson geometry, deformation quantisation and group representations, London Math. Soc. Lecture Note Ser., vol.~323, Cambridge Univ. Press, Cambridge, 2005, pp.~145--272.



\bibitem{Renault:1980}
J.~Renault,
\emph{A groupoid approach to $C^*$-algebras},
Lecture Notes in Mathematics~793, Springer Verlag, 1980.


\bibitem{SerreRepFinGr}
J.~P.~Serre,
\emph{Linear Representations of Finite Groups},
Translated from the second French edition by Leonard L. Scott, Graduate Texts in Mathematics, vol.~42, Springer-Verlag, New York-Heidelberg, 1977, x+170~pp.


\bibitem{TaylorMackeyConGr}
J.~Taylor, \emph{On the Mackey formula for connected centre groups}.
Journal of Group Theory, Vol.~21 (2018), no.~3, pp.~439--448.
\url{https://doi.org/10.1515/jgth-2018-0006}.

\bibitem{WeiGrpdUnInExtSym}
A.~Weinstein,
\emph{Groupoids: unifying internal and external symmetry. A tour through some examples.}
Notices Amer. Math. Soc.~43 (1996), no.~7, pp.~744–752.

\end{thebibliography}
\end{document}